\documentclass[12pt]{article}
\usepackage{amsthm}
\usepackage{amsfonts}
\usepackage{amssymb}
\usepackage{graphicx}

\def\inter{\mbox{Int}}
\newtheorem{theorem}{Theorem}
\newtheorem{corollary}[theorem]{Corollary}
\newtheorem{lemma}[theorem]{Lemma}

\newcounter{claims}
\newenvironment{claims}{\refstepcounter{claims}\par\medskip\noindent%
{{\bf (\theclaims)}~~}}{\par\medskip}
\newcommand{\claim}[2]{\begin{claims}{\em #2}\label{#1}\end{claims}}
\newcommand{\refclaim}[1]{(\ref{#1})}

\title{$3$-choosability of planar graphs with $(\le\!4)$-cycles far apart}
\author{Zden\v{e}k Dvo\v{r}\'ak\thanks{Department of Applied Mathematics, Faculty of Mathematics and 
           Physics, Charles University, Malostransk\'e n\'am\v{e}st\'{\i}~25, 118~00 Prague, 
           Czech Republic. E-mail: {\tt rakdver@kam.mff.cuni.cz}. 
           Supported by Institute for Theoretical Computer Science~(ITI),
           project 1M0021620808 of Ministry of~Education of~Czech Republic,
	   and by project GA201/09/0197 (Graph colorings and flows: structure and applications)
	   of Czech Science Foundation.}}
\date{}

\begin{document}
\maketitle

\begin{abstract}
A graph is $k$-choosable if it can be colored whenever every vertex
has a list of at least $k$ available colors.  We prove that
if cycles of length at most four in a planar graph $G$ are pairwise far
apart, then $G$ is $3$-choosable.  This is analogous to the problem
of Havel regarding $3$-colorability of planar graphs with triangles
far apart.
\end{abstract}

\section{Introduction}

All graphs considered in this paper are simple and finite. The concepts 
of list coloring and choosability were introduced by Vizing~\cite{vizing1976} and independently by 
Erd\H{o}s et al.~\cite{erdosrubintaylor1979}.
A {\em list assignment} of $G$ is a function $L$ that assigns to each vertex 
$v \in V(G)$ a list $L(v)$ of available colors. An {\em $L$-coloring} is a function 
$\varphi: V(G) \rightarrow \bigcup_v L(v)$ such 
that $\varphi(v) \in L(v)$ for every $v \in V(G)$ and
$\varphi(u) \neq \varphi(v)$ whenever $u$ and $v$ are adjacent vertices 
of $G$. If $G$ admits an $L$-coloring, then it is {\em $L$-colorable}. 
A graph $G$ is {\em $k$-choosable} if it is $L$-colorable for
every list assignment $L$ such that $|L(v)| \ge k$ 
for all $v \in V(G)$.  The {\em distance} between two vertices is the length (number of edges) of a
shortest path between them.  The distance $d(H_1,H_2)$ between two subgraphs $H_1$ and $H_2$ is the minimum
of the distances between vertices $v_1\in V(H_1)$ and $v_2\in V(H_2)$.

The well-known Four Color Theorem (Appel and Haken~\cite{AppHak1,AppHakKoc}) states that every planar
graph is $4$-colorable.  Similarly, Gr\"{o}tzsch~\cite{grotzsch1959} proved that every triangle-free planar
graph is 3-colorable.  For some time, the question whether these results hold in the list coloring
setting was open; finally, Voigt \cite{voigt1993, voigt1995} found a planar graph that is not $4$-choosable,
and a triangle-free planar graph that is not $3$-choosable.
On the other hand, Thomassen~\cite{thomassen1994,thomassen1995-34} proved that every planar graph is $5$-choosable
and every planar graph of girth at least $5$ is $3$-choosable.  Also, Kratochv\'\i{}l and Tuza \cite{krattuza1994}
observed that every planar triangle-free graph is $4$-choosable.

Motivated by Gr\"{o}tzsch's result, Havel asked whether there exists a constant $d$ such that if
the distance between any two triangles in a planar graph is at least $d$, then the graph is $3$-colorable.
This question was open for many years, finally being answered in affirmative by Dvo\v{r}\'ak, Kr\'al' and Thomas~\cite{dkt}
(although their bound on $d$ is impractically large).
Due to the result of Voigt~\cite{voigt1995}, an analogous question for $3$-choosability needs also to restrict $4$-cycles:
does there exist a constant $d$ such that if the distance between any two $(\le\!4)$-cycles in a planar graph is
at least $d$, then the graph is $3$-choosable?  We give a positive answer to this question:

\begin{theorem}\label{thm-main}
If $G$ is a planar graph such that the distance between any two $(\le\!4)$-cycles is at least $26$,
then $G$ is $3$-choosable.
\end{theorem}

This bound is quite reasonable compared to one given for Havel's problem~\cite{dkt}.  However, it is
far from the best known lower bound of $4$, given by Aksionov and Mel'nikov~\cite{aksmel}.

\section{Proof of Theorem~\ref{thm-main}}

For a subgraph $H$ of a graph $G$, let $d(H)=\min_F d(H, F)$, where the minimum ranges over all $(\le\!4)$-cycles $F$ of $G$ distinct
from $H$.  In case that the graph $G$ is not clear from the context, we write $d_G(H)$ instead.
Let $t(G)=\min_H d(H)$, where the minimum ranges over all $(\le\!4)$-cycles $H$ of $G$.
A {\em path of length $k$} (or a {\em $k$-path}) is a path with $k$ edges and $k+1$ vertices.
For a path or a cycle $X$, let $\ell(X)$ denote its length.  Let $r$ be the function defined by
$r(0)=0$, $r(1)=2$, $r(2)=4$, $r(3)=9$, $r(4)=13$ and $r(5)=16$.  For a path $P$ of length at most $5$, let $r(P)=r(\ell(P))$.
Let $B=26$.

A \emph{relevant configuration} is a triple $(H,Q,f)$, where $H$ is a plane graph, $Q$ is a subpath of the outer
face of $H$ and $f:V(H)\setminus V(Q)\to \{2,3\}$ is a function.  
A graph $G$ with list assignment $L$ and a specified path $P$ \emph{contains} the relevant configuration $(H,Q,f)$
if there exists an isomorphism $\pi$ between $H$ and a subgraph $I$ of $G$ (which we call the \emph{image} of the configuration in $G$)
such that $\pi$ maps $Q$ to $I\cap P$, and $|L(\pi(v))|=f(v)$ for every $v\in V(H)\setminus V(Q)$.  Figures~\ref{fig-obsta} and \ref{fig-obstb}
depict the relevant configurations that are used in the proof of Theorem~\ref{thm-main} using the following conventions.
The graph $H$ is drawn in the figure.  The path $Q$ consists of the vertices drawn by full circles.  The vertices
to that $f$ assigns value $2$ are drawn by empty squares, and the vertices to that $f$ assigns value $3$ are drawn by empty circles.

Using the precoloring extension technique developed by Thomassen~\cite{thomassen1995-34}, we show
the following generalization of Theorem~\ref{thm-main}:

\begin{figure}
\center{\includegraphics{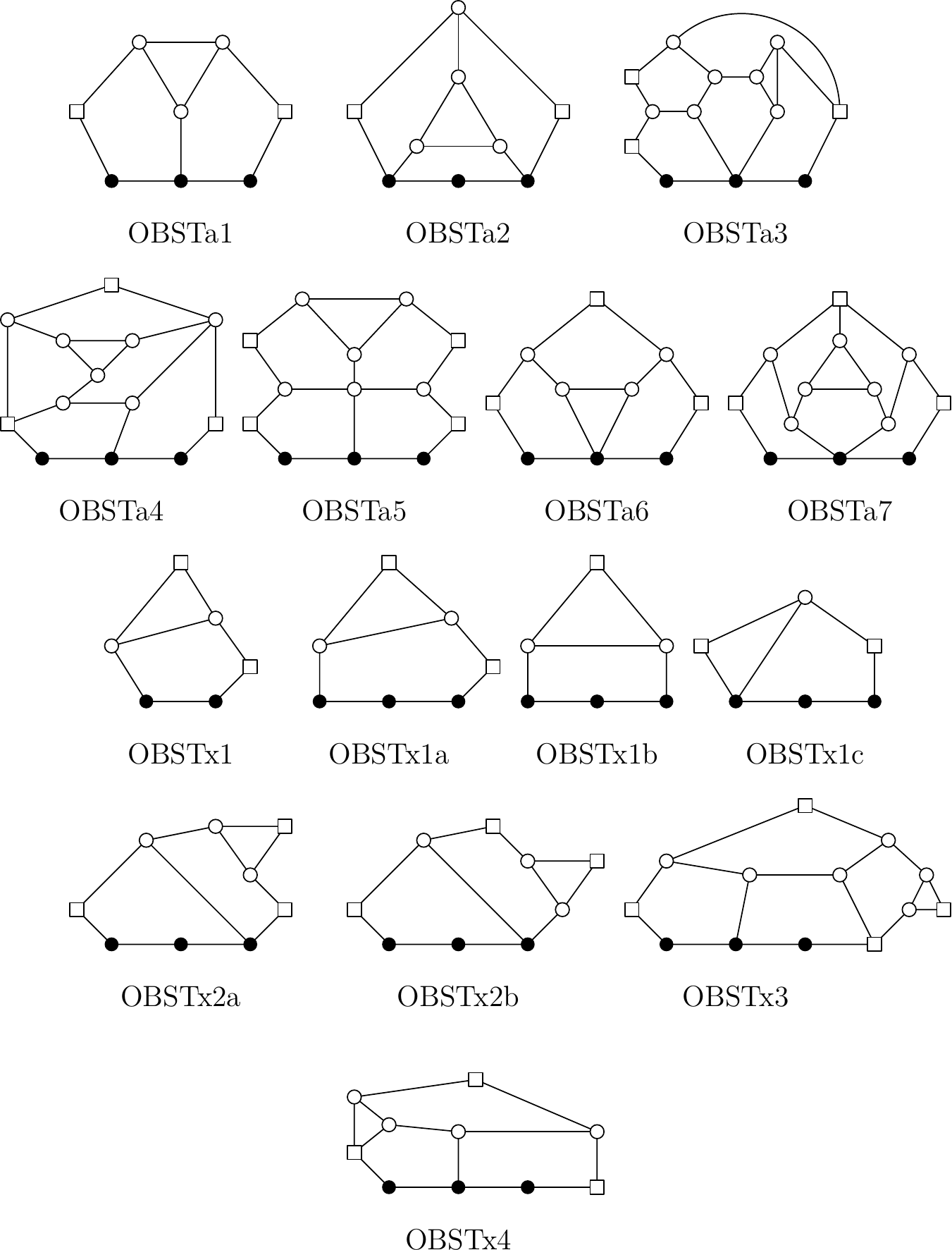}}
\caption{Relevant configurations of Theorem~\ref{thm-maingen}, $\ell(P)\le 2$.}
\label{fig-obsta}
\end{figure}

\begin{figure}
\center{\includegraphics{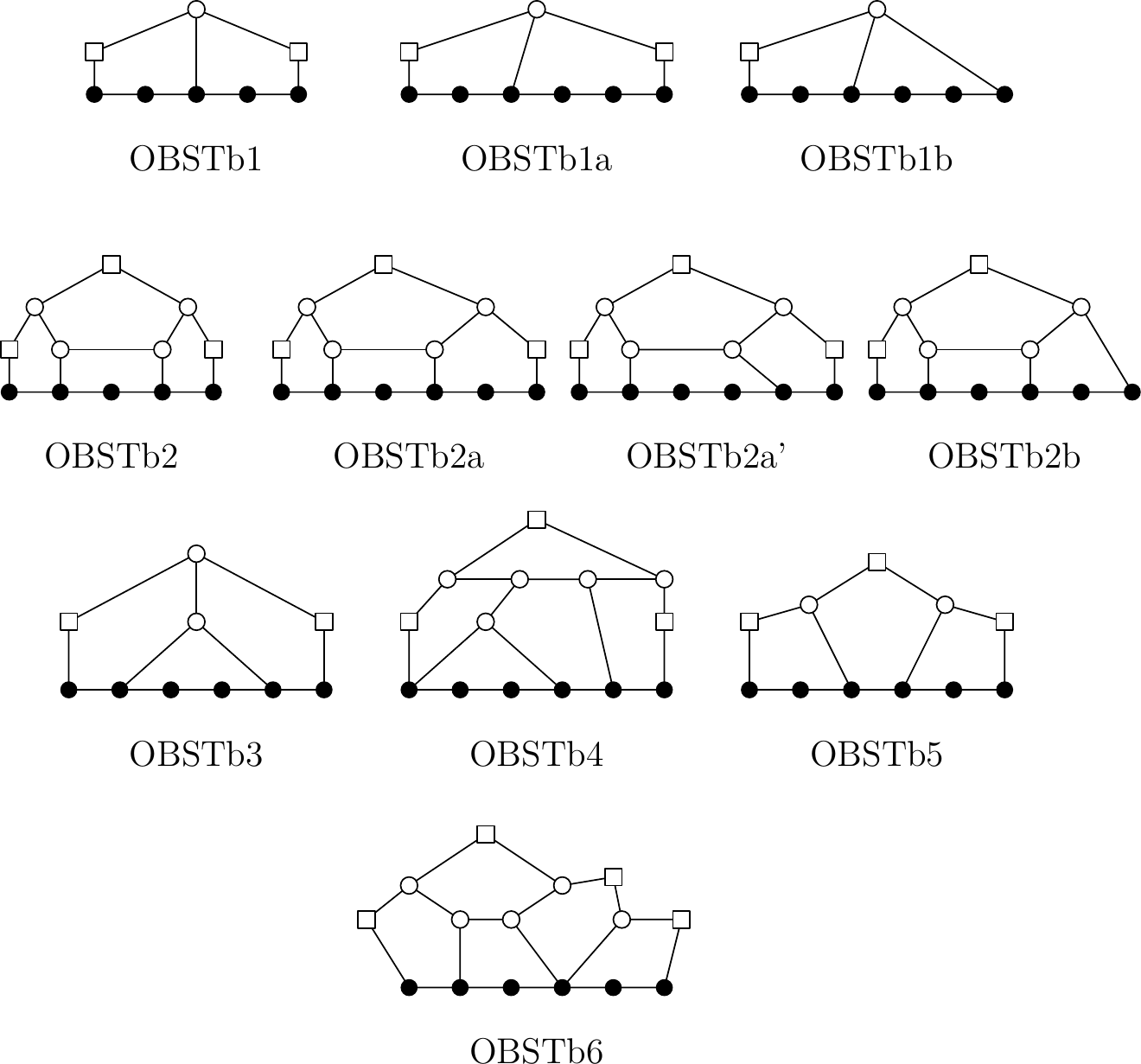}}
\caption{Relevant configurations of Theorem~\ref{thm-maingen}, $\ell(P)\le 5$.}
\label{fig-obstb}
\end{figure}

\begin{theorem}\label{thm-maingen}
Let $G$ be a planar graph with the outer face $C$ such that $t(G)\ge B$, and let $P$ be a path such
that $V(P)\subseteq V(C)$.  Let $L$ be a list assignment such that
\begin{itemize}
\item[(S1)] $|L(v)|=3$ for all $v\in V(G)\setminus V(C)$;
\item[(S2)] $2\le |L(v)|\le 3$ for all $v\in V(C)\setminus V(P)$;
\item[(S3)] $|L(v)|=1$ for all $v\in V(P)$, and the colors in the lists give a proper coloring of the subgraph of $G$ induced by $V(P)$;
\item[(I)] the vertices with lists of size two form an independent set;
\item[(T)] if $uvw$ is a triangle, $|L(u)|=2$ and $v$ has a neighbor with list of size two distinct from $u$,
then $w$ has no neighbor with list of size two distinct from $u$; and,
\item[(Q)] if a vertex $v$ with list of size two has two neighbors $w_1$ and $w_2$ in $P$, then $L(v)\neq L(w_1)\cup L(w_2)$.
\end{itemize}
Furthermore, assume that at least one of the following conditions is satisfied:
\begin{itemize}
\item[(OBSTa)] $\ell(P)\le 2$ and all images in $G$ of every relevant configuration drawn in Figure~\ref{fig-obsta} are $L$-colorable, or
\item[(OBSTb)] $\ell(P)\le 5$, $d(P)\ge r(P)$ and all images in $G$ of every relevant configuration drawn in Figure~\ref{fig-obstb} are $L$-colorable.
\end{itemize}
Then $G$ is $L$-colorable.
\end{theorem}

\noindent Note that we view the single-element lists as a precoloring of the vertices of $P$.
Also, $P$ does not have to be a part of the facial walk of $C$, as we only require $V(P)\subseteq V(C)$.

The assumptions (S3), (Q), (OBSTa) and (OBSTb) of Theorem~\ref{thm-maingen} are necessary conditions for the existence of an $L$-coloring.
The assumptions (S1), (S2) and (I) are the same as in the proof of Thomassen~\cite{thomassen1995-34}.
Finally, the assumption~(T) cannot be elimnated entirely---Figure~\ref{fig-cex} shows a non-$L$-colorable graph $G_1$ with only one
precolored vertex $x_1$ satisfying all assumptions of Theorem~\ref{thm-maingen} except for~(T).
By repeating the left part of this graph, $x_1$ can be made arbitrarily far apart from the triangle,
hence we cannot replace~(T) by adding further (finitely many) relevant configurations.  Furthermore,
we can construct a counterexample even without the precolored path $P$:
Let $G_2$ and $G_3$ with precolored vertices $x_1$ and $x_2$
be the copies of $G_1$ with the color $A$ replaced by colors $A'$ and $A''$, respectively, in the lists of all vertices.
Let $G$ be the graph obtained from $G_1$, $G_2$ and $G_3$ by identifying the vertices $x_1$, $x_2$ and $x_3$ to a single
vertex whose list is $\{A, A', A''\}$.  Note that $G$ is a counterexample to Theorem~\ref{thm-maingen} without the assumption~(T)
and that $G$ has no precolored vertices and $t(G)$ can be arbitrarily large.

\begin{figure}
\center{\includegraphics[width=120mm]{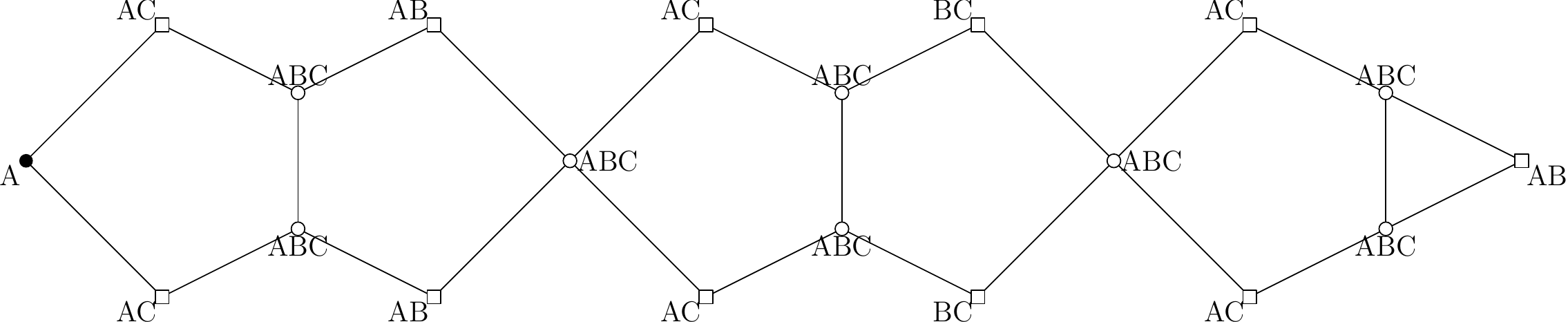}}
\caption{Assumption~(T) is necessary.}
\label{fig-cex}
\end{figure}

In his paper showing that every planar graph with at most three triangles is $3$-colorable, Aksionov~\cite{aksenov} also
proved that if $G$ is a plane graph with exactly one $(\le\!4)$-cycle, then any precoloring of a $5$-face of $G$
extends to a $3$-coloring of $G$.  Thomassen~\cite{thomassen1995-34} showed that in a planar graph of girth $5$,
any precoloring of an induced cycle $C$ of length at most $9$ extends to a $3$-coloring, unless a vertex has three neighbors in $C$.
Walls~\cite{walls} extended this characterization to cycles of length at most $11$ (giving more subgraphs that prevent the
coloring from being extended), Thomassen~\cite{thomassen-surf} generalized it for list-coloring, and Dvo\v{r}\'ak and Kawarabayashi~\cite{dk}
extended both of these results to cycles of length $12$.  Similarly, Theorem~\ref{thm-maingen} implies a result regarding extension of
a precoloring of a $(\le\!8)$-cycle, assuming that $(\le\!4)$-cycles are far apart.

Let $C$ be a $(\le\!8)$-cycle.  We call a plane graph $F$ a {\em $C$-obstacle} if $C\subseteq F$ bounds the outer face of $F$,
$F$ contains exactly one $(\le\!4)$-cycle, all vertices in $V(F)\setminus V(C)$ have degree at least three, and
\begin{itemize}
\item[O1:] $F-V(C)$ is a tree with at most $\ell(C)-6$ vertices, or
\item[O2:] $F-V(C)$ is a graph with at most $\ell(C)-3$ vertices whose only cycle is a triangle, or
\item[O3:] $F$ is one of the graphs drawn in Figure~\ref{fig-cobsta}.
\end{itemize}

\begin{figure}
\center{\includegraphics{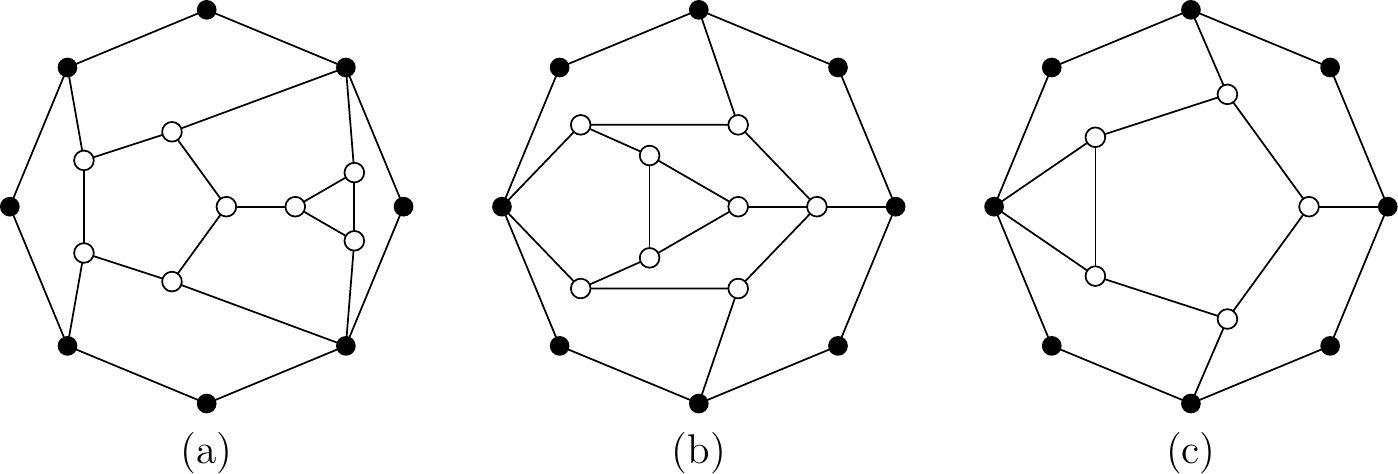}}
\caption{$C$-obstacles.}
\label{fig-cobsta}
\end{figure}

\begin{corollary}\label{cor-sepcycles}
Let $G$ be a plane graph with the outer face bounded by an induced $(\le\!8)$-cycle $C$,
such that $t(G)\ge B$.  Furthermore, assume that
$G$ does not contain a $C$-obstacle as a subgraph.
Let $L$ be an assignment of lists of size $1$ to the vertices of $C$ and lists of size $3$ to the
other vertices of $G$.  If $L$ prescribes a proper coloring of $C$, then $G$ is $L$-colorable.
\end{corollary}

Let us give a proof of this result in a slightly more general setting, which we are going to use in the
inductive proof of Theorem~\ref{thm-maingen}.  A graph $G_1$ is {\em smaller} than $G_2$ if
\begin{itemize}
\item $G_1$ has fewer $(\le\!4)$-cycles than $G_2$, or
\item $G_1$ and $G_2$ have the same number of $(\le\!4)$-cycles and satisfy $|V(G_1)|<|V(G_2)|$, or
\item $G_1$ and $G_2$ have the same number of $(\le\!4)$-cycles, $|V(G_1)|=|V(G_2)|$ and $|E(G_1)|<|E(G_2)|$.
\end{itemize}
If $G$ is a plane graph and $C\subseteq G$ is a cycle, then let $\inter_G(C)$ denote the subgraph of $G$
drawn in the closed disk bounded by $C$ (in particular, a chord of $C$ belongs to $\inter_G(C)$ only if
it is drawn inside of $C$).

\begin{lemma}\label{lemma-sepcycles}
Let $G$ be a plane graph and $L$ a list assignment satisfying the assumptions of Corollary~\ref{cor-sepcycles}.
If Theorem~\ref{thm-maingen} holds for all graphs smaller than $G$,
then $G$ is $L$-colorable.
\end{lemma}
\begin{proof}
Suppose that $G$ is a non-$L$-colorable graph satisfying the assumptions,
such that Lemma~\ref{lemma-sepcycles} holds for all graphs smaller than $G$.
Consider a $(\le\!8)$-cycle $K\neq C$ in $G$, and let $H=\inter_G(K)$.
If $H\neq K$, then, by the minimality of $G$, $G-(E(H)\setminus E(K))$ has an $L$-coloring $\varphi$,
and since $G$ is not $L$-colorable, the precoloring of $K$ given by $\varphi$ does not extend to an $L$-coloring of $H$.
By the minimality of $G$, we conclude the following.
\claim{lsc-cl-inter1}{Every $(\le\!8)$-cycle $K\neq C$ in $G$ either bounds a face, has a chord drawn inside
the disk bounded by $K$, or $\inter_G(K)$ contains a $K$-obstacle.}
Suppose that $K$ neither is a face nor has a chord drawn inside it, and let $F$ be a $K$-obstacle
drawn inside $K$.  Consider an internal face $K'$ of $F$ and observe that $K'$ is bounded by a cycle of length at most seven.
Since $F$ contains a $(\le\!4)$-cycle and $t(G)\ge B$, neither a chord nor a $K'$-obstacle can be drawn inside $K'$.
Therefore, \refclaim{lsc-cl-inter1} implies that $K'$ is also a face of $G$.  Since this holds for every internal face of $F$,
we conclude that $H=F$.  Therefore, \refclaim{lsc-cl-inter1} can be strengthened as follows.
\claim{lsc-cl-inter}{every $(\le\!8)$-cycle $K\neq C$ in $G$ either bounds a face, has a chord drawn inside
the disk bounded by $K$, or $\inter_G(K)$ is a $K$-obstacle.}
In particular, every $(\le\!5)$-cycle bounds a face. 

Consider a vertex $v\in V(G)\setminus V(C)$, and assume that $v$ has more than one neighbor in $C$.
If $v$ has at least three neighbors in $C$, then $G$ contains the $C$-obstacle consisting of $v$, $C$ and three
edges incident with $v$ (satisfying the condition O1).
Thus, suppose that $v$ has exactly two neighbors $w_1,w_2\in V(C)$.  Furthermore, suppose that $\ell(C)\le 7$
or that $w_1$ and $w_2$ are non-adjacent.  Let $K_1$ and $K_2$ be the two cycles formed by $w_1vw_2$ and the
two paths between $w_1$ and $w_2$ in $C$, and note that $\ell(K_1), \ell(K_2)\le 8$ and both $K_1$ and $K_2$
are induced cycles.  By \refclaim{lsc-cl-inter} and the assumption that $t(G)\ge B$,
we conclude that at least one of $K_1$ and $K_2$
(say $K_1$) bounds a face.  By the minimality of $G$, $v$ has degree at least three, thus $K_2$ does not bound a face.
Again, since $t(G)\ge B$, this implies that $\ell(K_1)\ge 5$
and $6\le\ell(K_2)\le 7$.  Thus, the subgraph $F_2$ drawn inside $K_2$ is a $K_2$-obstacle satisfying condition O1 or O2, and
$F_2\cup K_1$ is a $C$-obstacle in $G$, a contradiction.  It follows that \claim{lsc-cl-no2}{no vertex $v\in V(G)\setminus V(C)$
has more than one neighbor in $C$, unless $\ell(C)=8$ and the neighbors of $v$ in $C$ are adjacent.}
Also, observe that
\claim{lsc-cl-no2p}{if $\ell(C)=8$ and $v$ has two adjacent neighbors $w_1$ and $w_2$ in $C$, then no neighbor $x$ of $v$
distinct from $w_1$ and $w_2$ is adjacent to a vertex in $C$,} as otherwise \refclaim{lsc-cl-inter} together with $t(G)\ge B$
implies that either $x$ has degree two or $x$ has two non-adjacent neighbors in $C$.

Suppose now that two adjacent vertices $v_1,v_2\in V(G)\setminus V(C)$ both have a neighbor in $C$.
By \refclaim{lsc-cl-no2} and \refclaim{lsc-cl-no2p}, each of them has exactly one such neighbor;
let $w_i\in V(C)$ be the neighbor of $v_i$, for $i\in\{1,2\}$.  Furthermore, suppose that
both (induced) cycles $K_1$ and $K_2$ consisting of $w_1v_1v_2w_2$ together with a path joining $w_1$ with $w_2$ in $C$ have
length at least $6$.  Note that $\ell(K_1)+\ell(K_2)=\ell(C)+6$, thus $\ell(K_1),\ell(K_2)\le \ell(C)$ and $\ell(C)\ge 6$.
Since $t(G)\ge B$, \refclaim{lsc-cl-inter} implies that,
say, $K_1$ bounds a face and $\inter_G(K_2)$ is a $K_2$-obstacle.
Consider the graph $G'$ with outer face $C'$ obtained from $G$ by contracting an edge $e$ of the path $K_1-\{w_1,v_1,v_2,w_2\}$
and giving the resulting vertex a color different from the color of its neighbors.  By \refclaim{lsc-cl-inter}, $e$ does
not belong to a $(\le\!5)$-cycle in $G$, thus the contraction does not create any $(\le\!4)$-cycle.
Also, as $G$ contains only one cycle of length at most $4$ (drawn inside $K_2$), the restriction on the
distance between $(\le\!4)$-cycles in $G'$ is vacuously true.  The graph $G'$ is not $L$-colorable, and by the minimality
of $G$, it contains a $C'$-obstacle satisfying O1 or O2.  However, this gives a corresponding $C$-obstacle in $G$.
Therefore, \claim{lsc-cl-3chord}{if each of two adjacent vertices $v_1,v_2\in V(G)\setminus V(C)$ has a neighbor in $C$,
then they together with a path in $C$ bound a face of length at most $5$.}

If $3\le \ell(C)\le4$, then consider the graph $G'$ obtained from $G$ by subdividing an edge of $C$ by $5-\ell(C)$ new
vertices, and giving these vertices distinct colors that do not appear in any of the lists of $G$.
Note that $G'$ is smaller than $G$, since it contains fewer $(\le\!4)$-cycles, and by the minimality of $G$, we
conclude that $G'$ is $L$-colorable.  However, that gives an $L$-coloring of $G$, thus we can assume that $\ell(C)\ge 5$.

Let us now show that
\claim{lsc-cl-setx}{there exists a set $X\subseteq V(C)$ of $\max(1,\ell(C)-5)$ consecutive vertices
of $C$ such that
\begin{itemize}
\item every walk in $G$ of length at most $3$ whose endvertices belong to $X$ is contained in the subgraph of $G$ induced by $X$, and
\item no vertex of $X$ has a neighbor in a triangle.
\end{itemize}}
If $\ell(C)\le 7$, then by \refclaim{lsc-cl-no2}, at most three vertices of $C$ belong to or have a neighbor in a triangle,
and at most two vertices belong to a $4$-cycle.  Since $t(G)\ge B$, these cases are mutually exclusive,
thus we can choose $X$ as a subset of the remaining (at least $\ell(C)-3$) vertices.
Hence, suppose that $\ell(C)=8$ and $C=v_1v_2\ldots v_8$.  If say $v_2v_3$ is an edge of a triangle,
then none of $v_5$, \ldots, $v_8$ has a neighbor in a triangle.  If $v_5v_6v_7$ is not a part of the boundary
walk of a $5$-face, then set $X=\{v_5,v_6,v_7\}$; otherwise, $v_6v_7v_8$ is not a part of the boundary
walk of a $5$-face by \refclaim{lsc-cl-no2}, and we set $X=\{v_6,v_7,v_8\}$.  We choose the set $X$ in the
same way in case that a triangle shares a single vertex $v_2$ with $C$, or a $4$-cycle shares
at most two vertices $v_2$ and $v_3$ with $C$, or no $(\le\!4)$-cycle intersects $C$ and
at least $4$ consecutive vertices $v_5$, $v_6$, $v_7$ and $v_8$ have no neighbor in a triangle.
It remains to consider the case that no $(\le\!4)$-cycle intersects $C$ and among each $4$ consecutive vertices,
at least one has a neighbor in a triangle.  If three vertices of $C$ had a neighbor in a triangle, then
\refclaim{lsc-cl-inter} would imply that $G-V(C)$ is a triangle, giving a $C$-obstacle satisfying O2.  Therefore, two opposite vertices of $C$,
say $v_1$ and $v_5$, have a neighbor in a triangle.  However, this contradicts \refclaim{lsc-cl-no2} or \refclaim{lsc-cl-3chord}.

Therefore, the set $X$ with properties specified by \refclaim{lsc-cl-setx} exists.
Let $C-X=v_1v_2\ldots v_k$, where $k=\ell(C)-|X|\le 5$.
Let $G'=G-X$, with the list assignment $L'$ obtained from $L$ by removing from the list of each vertex the colors
of its neighbors (if any) in $X$.  Furthermore, we set $L'(v_1)=L(v_1)\cup L(v_2)$ and
$L'(v_k)=L(v_k)\cup L(v_{k-1})$.  We apply Theorem~\ref{thm-maingen} with the list assignment $L'$
and the precolored path $P'=v_2v_3\ldots v_{k-1}$.  The conditions~(S1) and~(S3) are clearly satisfied,
and~(S2) holds by \refclaim{lsc-cl-no2} and \refclaim{lsc-cl-setx}.  Every vertex with list of size two
has a neighbor in $X$, and if two of them were adjacent, $G$ would contain a walk of length three
contradicting \refclaim{lsc-cl-setx}; therefore,~(I) holds.  Similarly, \refclaim{lsc-cl-setx} implies
that no vertex with list of size two is contained in a triangle, which together with \refclaim{lsc-cl-no2}
implies that no vertex with list of size two has two neighbors in $P'$; hence, the conditions~(T) and~(Q) are true.

Observe that an $L'$-coloring of $G'$ would extend to an $L$-coloring of $G$, and thus $G'$ is not
$L'$-colorable.  Since $\ell(P')=k-3\le 2$, we conclude that the condition~(OBSTa) of Theorem~\ref{thm-maingen}
is false.  Therefore, $k=5$ (and hence $\ell(C)\ge 6$) and $G'$ contains a relevant configuration $H$
which is one of OBSTa1 -- OBSTa7 drawn in Figure~\ref{fig-obsta}.  However, a case analysis shows that
\begin{itemize}
\item if $H$ is OBSTa1 or OBSTa2, then $G$ contains a $C$-obstacle satisfying~(O2),
\item if $H$ is OBSTa3, then $G$ contains the $C$-obstacle drawn in Figure~\ref{fig-cobsta}(a).
\item if $H$ is OBSTa4, OBSTa5 or OBSTa7, then $G$ contains the $C$-obstacle drawn in Figure~\ref{fig-cobsta}(b).
\item if $H$ is OBSTa6, then $G$ contains the $C$-obstacle drawn in Figure~\ref{fig-cobsta}(c).
\end{itemize}
\end{proof}

Let us now give a short outline of the proof of Theorem~\ref{thm-maingen}.  We basically follow the proof of the Gr\"otzsch theorem
by Thomassen~\cite{thomassen1995-34}, which the reader should be familiar with.  We consider a hypothetical smallest counterexample.
First, we give constraints on short paths $Q$ whose endvertices belong to $C$ and internal vertices do not belong to $C$ (claims \refclaim{cl-2conn}, \refclaim{cl-nochords} and
\refclaim{cl-schords} in the proof), by splitting the graph along $Q$, coloring one part and extending the coloring to the second
one, with $Q$ playing the role of the precolored path in the second part.  However, due to the
existence of counterexamples to the statement ``every precoloring of a path of length two can be extended''
(depicted in Figure~\ref{fig-obsta}), we cannot exclude such paths entirely.  Nevertheless, using the ability to color
vertices of a path of length up to $5$ if we can in the process ensure that there are no $(\le\!4)$-cycles nearby,
we can strengthen these constraints sufficiently if the vertices of $Q$ are close to $P$ (claims \refclaim{cl-nosimt} and
\refclaim{cl-nonearchords}).  Then, as in Thomassen's proof, we try to color up to five appropriately chosen vertices of $G$ near to $P$ and
remove their colors from the lists of their neighbors, so that the resulting graph $G'$ satisfies the assumptions of
Theorem~\ref{thm-maingen}.  This may only fail if a $(\le\!4)$-cycle $T$ appears near to the colored vertices, making~(I) or~(T) false
(claims \refclaim{cl-neart} and \refclaim{cl-neartb}).  Note that this implies that $\ell(P)\le 2$.  Many of these problematic configurations 
(those where $T$ is a $4$-cycle, or where~(T) is false in $G'$) can be reduced by precoloring up to three more vertices near to $T$,
adding them to the precolored path, and at the same time removing some vertices so that $T$ disappears.
Still, some cases (e.g., when $T$ contains a vertex in $C$ whose distance from $P$ is at most four)
remain.  However, then we observe that we can apply the symmetric argument on the other side of $P$, and if that
fails as well, a $(\le\!4)$-cycle $T'$ must be close to the vertices that we try to color there as well.  Since the distance
between any two $(\le\!4)$-cycles in $G$ is at least $B$, it follows that $T'=T$, which implies that $G$ contains
a short path $Q$ with endvertices in $C$.  Using the constraints on such paths,
we can find a suitable set of vertices to color and remove in this case as well, finally finishing the proof.

Let us now provide the details of this argument, which unfortunately turns out to be rather lengthy and technical.

\begin{proof}[Proof of Theorem~\ref{thm-maingen}.]
Suppose that $G$ together with lists $L$ is a smallest counterexample, i.e.,
Theorem~\ref{thm-maingen} holds for every graph smaller (in the sense as defined before Lemma~\ref{lemma-sepcycles}) than $G$, and
$G$ satisfies the assumptions of Theorem~\ref{thm-maingen},
but $G$ is not $L$-colorable.  Let $C$ be the outer face of $G$ and $P$ a path with $V(P)\subseteq V(C)$
as in the statement of the theorem.  We first derive several properties of this counterexample.  Note that each
vertex $v$ of $G$ has degree at least $\max(2,|L(v)|)$, and if two vertices $u$ and $v$ are adjacent, then $L(u)\cap L(v)\neq\emptyset$,
unless $uv$ is an edge of $P$.  In particular, if $v\not\in V(P)$ is adjacent to a vertex $p\in V(P)$, then $L(p)\subset L(v)$.

Lemma~\ref{lemma-sepcycles} implies that
\claim{cl-inter}{every $(\le\!8)$-cycle $K$ in $G$ bounds a face, or has a chord drawn inside
the disk bounded by $K$, or $\inter_G(K)$ is a $K$-obstacle.}
In particular, every $(\le\!5)$-cycle in $G$ bounds a face.  Furthermore,

\claim{cl-2conn}{the graph $G$ is $2$-connected.}
\begin{proof}
Clearly, $G$ is connected.  Suppose that $G$ is not $2$-connected, and let $G=G_1\cup G_2$, where $V(G_1)\cap V(G_2)=\{v\}$
and $|V(G_1)|,|V(G_2)|\ge 2$.  If say $P\subseteq G_1$, then by the minimality of $G$, an $L$-coloring $\varphi_1$ of $G_1$ exists.
Let $L_2$ be the list assignment such that $L_2(x)=L(x)$ for $x\neq v$ and $L_2(v)=\{\varphi_1(v)\}$.
By the minimality of $G$, we have that $G_2$ is $L_2$-colorable.  However, this gives an $L$-coloring of $G$.
Similarly, in case that the cut-vertex $v$ is an internal vertex of $P$, the minimality of $G$ implies that both
$G_1$ and $G_2$ are $L$-colorable, giving an $L$-coloring of $G$.  This is a contradiction.
\end{proof}

A {\em chord} of a cycle $K$ is an edge $e\not\in E(K)$ joining two vertices of $K$.
A vertex of a path is {\em internal} if its degree in the path is two, and an {\em endvertex} otherwise.

\claim{cl-nochords}{Every chord of $C$ joins two vertices $u$ and $v$ with lists of size three, such that either
\begin{itemize}
\item $u$ and $v$ have a common neighbor with list of size two; or,
\item there exist a triangle $w_1w_2w_3$ and a vertex $z\not\in\{w_2,w_3\}$ adjacent to $w_1$, such that $|L(w_2)|=|L(z)|=2$ and either $uz, vw_3\in E(G)$ or $uw_3,vz\in E(G)$.
\end{itemize}}
\begin{proof}
Let $uv$ be a chord of $C$.  Let $G=G_1\cup G_2$, where $V(G_1)\cap V(G_2)=\{u,v\}$ and $|V(G_1)|,|V(G_2)|\ge 3$.
By symmetry, we can assume that $|V(G_1)\cap V(P)|\ge |V(G_2)\cap V(P)|$.
If $u,v\in V(P)$, then by the minimality of $G$, both $G_1$ and $G_2$ are $L$-colorable, and their colorings
combine to an $L$-coloring of $G$.  This is a contradiction, thus we can assume that $v\not\in V(P)$.
Let $P_2=(P\cap G_2)\cup \{uv\}$; note that if $u\not\in V(P)$, then $P\subset G_1$, hence $P_2$ is a path.
If $u\in V(P)$, then let $P_1$ denote the path $(P\cap G_1)\cup \{uv\}$.

By the minimality of $G$, we can apply Theorem~\ref{thm-maingen} to $G_1$ with the precolored
path $P\cap G_1$ and the list assignment $L$ and conclude that there exists an $L$-coloring $\varphi$ of $G_1$.
Let $L'$ be the list assignment such that $L'(x)=L(x)$ for $x\not\in \{u,v\}$ and
$L'(x)=\{\varphi(x)\}$ for $x\in \{u,v\}$.  Since $G$ is not $L$-colorable,
$G_2$ with the precolored path $P_2$ is not $L'$-colorable, and thus it violates either~(Q) or both~(OBSTa) and~(OBSTb).

Suppose first that $u$ is not an internal vertex of $P$.  Then $\ell(P_2)=1$, and thus $G_2$ contains either a vertex with list of size two adjacent to
$u$ and $v$ or OBSTx1.  By~(I) and~(T), neither $L(u)$ nor $L(v)$ has size two.  Furthermore, note that $u$ cannot be an endvertex of $P$:
Otherwise, we have $d(P)\le 2$, thus $\ell(P)\le 2$. Let $c\neq \varphi(v)$ be a color in $L(v)\setminus L(u)$
and $L_2$ the list assignment such that $L_2(v)=\{c\}$ and $L_2(x)=L(x)$ for $x\neq v$.
Note that $G_2$ with list assignment $L_2$ and precolored path $uv$ satisfies~(Q) and~(OBSTa),
and by the minimality of $G$, $G_2$ is $L_2$-colorable.  It follows that $G_1$ (with the precolored path $P_1$) cannot
be $L_2$-colorable.  However, the distance between $u$ and a triangle in $G_2$ is at most two, hence in $G_1$,
we have $d(P_1)\ge B-4\ge r(P_1)$.
Since $G_1$ is not $L_2$-colorable, it follows that $G_1$ violates~(Q).  However, that
implies that $G$ contains a non-$L$-colorable OBSTx1c, OBSTx2a or OBSTx2b, which is a contradiction.
Therefore, if $u$ is not an internal vertex of $P$, then the chord $uv$ satisfies the conclusion of \refclaim{cl-nochords}.

Let us now consider the case that $u$ is an internal vertex of $P$. 
By the choice of $G_1$ and $G_2$,
we have $2\ell(P_2)\le \ell(P)+2$.  Suppose first that $\ell(P_2)=2$.  By the minimality of $G$, we conclude that~(S3),
(Q) or~(OBSTa) fails for $G_2$ with the list assignment $L'$ and precolored path $P_2$. This implies that $d(P)\le 3$, and since $G$ satisfies the assumptions of
Theorem~\ref{thm-maingen}, we have $\ell(P)=2$.  However, then we have $\ell(P_1)=2$
and we can apply a symmetrical argument to $G_1$, concluding that $G_1$ also contains a $(\le 4)$-cycle
at distance at most three from $P$.  This implies that $t(G)\le 8$, which is a contradiction.

Therefore, we can assume that $\ell(P_2)=3$, and thus $\ell(P)\ge 4$ and $d(P)\ge r(P)$.  Note that
$d(P_2)\ge d(P)-1$, and thus $d(P_2)\ge r(P_2)$.  By the minimality of $G$, we have that $G_2$ fails~(Q),
and $G_2$ contains a vertex $w$ with $|L(w)|=2$ adjacent both to $v$ and to an endvertex of $P$.
Let $\varphi'$ be an $L$-coloring of $G_2$ and let $L''$ be the list assignment for $G_1$
such that $L''(v)=\{\varphi'(v)\}$ and $L''$ matches $L$ on other vertices of $G_1$.
Note that $G_1$ is not $L''$-colorable.  By the minimality of $G$, we can apply Theorem~\ref{thm-maingen}
to $G_1$ with the list assignment $L''$ and the precolored path $P_1$.  We conclude that
either $G_1$ fails~(Q), or $\ell(P)=5$ and $G_1$ fails either~(S3) or~(OBSTb)
due to an image of OBSTb1 or OBSTb2.  This together with a $5$-cycle in $G_2$ gives an image $H$ of one of relevant configurations OBSTb1, OBSTb1a, OBSTb1b,
OBSTb5 or OBSTb6; and by \refclaim{cl-inter}, this image is unique.  By~(OBSTb), $H$ has an $L$-coloring $\psi$.
Let $L^*$ be the list assignment such that $L^*(v)=\{\psi(v)\}$ and $L^*$ matches $L$ on other vertices of $G$.
Observe that $G_2$ with the list assignment $L^*$ and precolored path $P_2$ satisfies~(Q) and
that $G_1$ with the list assignment $L^*$ and precolored path $P_1$ satisfies~(Q),~(S3) and~(OBSTb).
By the minimality of $G$, we conclude that both $G_1$ and $G_2$ are $L^*$-colorable,
which gives an $L$-coloring of $G$.  This is a contradiction.
\end{proof}

Let us note that \refclaim{cl-nochords} implies that $P$ is a subpath of $C$.  Furthermore, observe that
there exists an $L$-coloring of the subgraph of $G$ induced by $V(C)$, unless
$G$ contains a non-$L$-colorable OBSTx1, OBSTx1a or OBSTx1b.  Lemma~\ref{lemma-sepcycles} then implies that
\claim{cl-length}{the length of $C$ is at least $9$.}
\begin{proof}
If $\ell(C)\le 8$, then $G$ would contain a $C$-obstacle $H$, and by \refclaim{cl-inter}, it would actually
be equal to this $C$-obstacle.  Since each $C$-obstacle contains a $(\le\!4)$-cycle whose
distance from any vertex of $C$ is at most $4$, this is only possible if $\ell(P)\le 2$.  However, a straightforward
case analysis shows that either $G$ is $L$-colorable or violates~(OBSTa), which is a contradiction.  More precisely,
\begin{itemize}
\item If $H$ satisfies~(O1) and $|V(H)\setminus V(C)|=1$, then either $G$ is an image of OBSTa1 or $G$ is $L$-colorable.
\item If $H$ satisfies~(O1) and $|V(H)\setminus V(C)|=2$, then either $G$ is an image of OBSTa6 or OBSTx4 or $G$ is $L$-colorable.
\item If $H$ satisfies~(O2) and $|V(H)\setminus V(C)|=3$, then either $G$ is an image of OBSTa2 or $G$ is $L$-colorable.
\item If $H$ satisfies~(O2) and $|V(H)\setminus V(C)|=4$, then $G$ is $L$-colorable.
\item If $H$ satisfies~(O2) and $|V(H)\setminus V(C)|=5$, then either $G$ is an image of OBSTa3, OBSTa4 or OBSTa7, or $G$ is $L$-colorable.
\item If $H$ satisfies~(O3), then $G$ is $L$-colorable.
\end{itemize}
\end{proof}

For $k\ge 2$, a {\em $k$-chord} of a cycle $K$ is a path $Q=q_0q_1\ldots q_k$ of length $k$ joining two distinct vertices
of $K$, such that $V(K)\cap V(Q)=\{q_0,q_k\}$.  We consider a chord to be a {\em $1$-chord}.

Suppose that $Q=q_0q_1\ldots q_k$ is a $k$-chord of the outer face $C$ of $G$, such that neither $q_0$ nor $q_k$ is an internal vertex of $P$.
Let $G_1$ and $G_2$ be subgraphs of $G$ such that $G_1\cup G_2=G$, $G_1\cap G_2=Q$, $P\subset G_1$ and $G_2\neq Q$.
We say that $Q$ {\em splits off a face} if $G_2$ is a cycle.
If $O=(H,R, f)$ is a relevant configuration, we say that the $k$-chord $Q$ {\em splits off $O$} if
there exists an isomorphism $\pi$ from $O$ to $G_2$ such that
\begin{itemize}
\item $\pi$ maps $R$ to the (not necessarily proper) subpath of $Q$ consisting of all vertices $x\in V(Q)$ such that $|L(x)|\in\{1,3\}$, and
\item $|L(\pi(v))|=f(v)$ for every $v\in V(H)\setminus V(R)$.
\end{itemize}

\claim{cl-schords}{Let $Q=q_0q_1\ldots q_k$ be a $k$-chord of $C$ such that no endvertex of $Q$ is an internal vertex of $P$
and $Q$ does not split off a face.
If $k\le 2$, or if $k=3$ and $q_3$ has list of size two, then $Q$ splits off
one of the relevant configurations drawn in Figure~\ref{fig-obsta}.}
\begin{proof}
Suppose for a contradiction that there exists a $k$-chord $Q$ violating \refclaim{cl-schords}.
Let $G_1$ and $G_2$ be subgraphs of $G$ such that $G_1\cup G_2=G$, $G_1\cap G_2=Q$, $P\subset G_1$ and $G_2\neq Q$.
Let us choose $Q$ among all $(\le\!3)$-chords of $C$ that violate \refclaim{cl-schords} so that $|V(G_2)|$ is minimal.

By the minimality of $G$, there exists an $L$-coloring $\varphi$ of $G_1$.
Let $L'$ be the list assignment such that $L'(x)=L(x)$ if $x\not\in V(Q)$, $L'(q_3)=\{\varphi(q_2),\varphi(q_3)\}$ if
$k=3$ and $L'(q_i)=\{\varphi(q_i)\}$ for $0\le i\le 2$.  Observe that $G_2$ with the precolored path $q_0q_1q_2$ is not $L'$-colorable, thus
it violates~(Q) or~(OBSTa).  Let $H$ be the minimal subgraph of $G_2$ that contains $q_0q_1q_2$ and violates~(Q) or~(OBSTa).
Note that $H$ contains a $(\le\!4)$-cycle $T$ whose distance to any vertex of $H$ is at most four.
By \refclaim{cl-inter}, each face of $H$ except for the outer one is also a face of $G$.

We claim that $G_2=H$, that is, $Q$ splits off the corresponding face or relevant configuration.  Otherwise, consider a $k'$-chord $Q'\neq Q$ of $G_2$ that is a subpath of
the union of $Q$ and of the outer face of $H$.  If $Q'$ satisfies the assumptions of \refclaim{cl-schords}, then by the choice of $Q$,
we have that that $Q'$ splits off a subgraph $H'$ that is either a face or an image of a relevant configuration drawn in Figure~\ref{fig-obsta}.  However, $H'$ contains
a $(\le\!4)$-face $T'$, whose distance to $Q'$ is at most three.  It follows that $d(T, T')\le 7 < B$, which is a contradiction.
Therefore, $Q'$ does not satisfy the assumptions of \refclaim{cl-schords}.  Since every vertex with list of size two in $H$ belongs to the
outer face of $G$, an inspection of the graphs in Figure~\ref{fig-obsta} shows that
this is only possible if $k=3$, $H$ is OBSTx1 and $Q'=q_3q_2q_1uv$ for vertices $u,v\in V(H)\setminus V(Q)$ such that
$|L(u)|=3$ and $|L(v)|=2$.  However, in this case let $G'_1$ and $G'_2$ be the subgraphs of $G$ such that $G'_1\cup G'_2=G$, $G'_1\cap G'_2=Q'$,
$P\subset G'_1$ and $G'_2\neq Q'$, let $\varphi'$ be an $L$-coloring of $G'_1$ and let $L_2$ be the list assignment such that $L_2(x)=\{\varphi'(x)\}$ for
$x\in \{v,q_1,q_2\}$, $L_2(q_3)=\{\varphi'(q_3),\varphi'(q_2)\}$, $L_2(v)=\{\varphi'(u),\varphi'(v)\}$
and $L_2(x)=L(x)$ for other vertices $x\in V(G'_2)$.  Since $t(G)\ge B$ and $H$ contains $T$, we conclude that
$G'_2$ satisfies the assumptions of Theorem~\ref{thm-maingen}, hence $G'_2$ is $L_2$-colorable.  This gives an $L$-coloring of $G$,
which is a contradiction.
\end{proof}

\refclaim{cl-inter} and \refclaim{cl-schords} imply that
\claim{cl-noobstr}{none of the relevant configurations depicted in Figures~\ref{fig-obsta} and \ref{fig-obstb}
is contained in $G$.}
Indeed, if $H$ were an image of such a configuration in $G$, we would conclude that $G=H$ as in the proof of \refclaim{cl-schords},
and by the assumptions, $G$ would be $L$-colorable.

\claim{cl-2pchord}{If $Q=q_0q_1q_2$ is a $2$-chord of $C$ in $G$, then at most one endvertex of $Q$ belongs to $P$.}
\begin{proof}
Suppose that both $q_0$ and $q_2$ belong to $P$.  Then $Q$ together with a subpath of $P$ forms a cycle $K$ of length at most
$\ell(P)+2$, and by \refclaim{cl-inter} together with the assumption
that $d(P)\ge r(P)$ if $\ell(P)>2$, this cycle bounds a face.  Observe that $q_1$ cannot have a neighbor in $P$ distinct
from $q_0$ and $q_2$. Furthermore, we have $|L(q_1)|=3$, since $\ell(C)>7$ and $C$ has no chords.  Let $L'$ be the list assignment such that $L'(q_1)\subseteq L(q_1)\setminus~(L(q_0)\cup L(q_2))$ has size one and $L'(x)=L(x)$ for $x\neq q_1$.
Let $G'=G-q_0q_2$ if $K$ is a triangle and $G'=G-(V(K)\setminus V(Q))$ otherwise.
Note that the vertices with list of size one form an induced path $P'$ in $G'$, and the length of $P'$ is
at most $\ell(P)-1$ if $K$ has length at least $5$ and at most $\ell(P)+1$ otherwise.  In the former case,
if $d_G(P)\ge r(P)$, then $d_{G'}(P')\ge r(P')$, since $d_{G'}(P')\ge d_G(P)-1$. In the latter case,
we have $\ell(P)\le 2$ and $d_{G'}(P')\ge r(P')$, since $d_G(K)\ge B$.
Since $G'$ is smaller than $G$ and is not $L'$-colorable, we conclude that it violates~(Q) or~(OBSTb).
However, in these cases, $G$ itself would violate~(OBSTb): If $G'$ violates~(Q), then $G$ contains
OBSTb1b; if $G'$ contains OBSTb1, then $G$ contains OBSTb3; and if $G'$ contains OBSTb2, then $G'$ contains OBSTb4.
\end{proof}

\claim{cl-4chord}{Suppose that $C$ has either a $3$-chord $Q=q_0q_1q_2q_3$,
or a $4$-chord $Q=q_0q_1q_2q_3q_4$ such that $|L(q_4)|=2$, where no endvertex of $Q$
is an internal vertex of $P$.  
Let $G_1$ and $G_2$ be subgraphs of $G$ such that $G_1\cup G_2=G$, $G_1\cap G_2=Q$, $P\subset G_1$ and $G_2\neq Q$.
Assume that either
\begin{itemize}
\item $\ell(P)\ge 4$ and $d(P, q_i)\le r(4)-r(3)=4$ for $0\le i\le 3$, or
\item $G_1$ contains a $(\le\!4)$-cycle $T$ such that $d(T, q_i)\le B - r(3)$ for $0\le i\le 3$.
\end{itemize}
Then $G_2$ is a $5$-cycle, and hence $q_0$ and $q_3$ have a common neighbor with list of size two (equal to $q_4$
if $Q$ is a $4$-chord).}
\begin{proof}
Let $\varphi$ be an $L$-coloring of $G_1$ that exists by the minimality of $G$.
Let $L_2$ be the list assignment such that $L_2(q_i)=\{\varphi(q_i)\}$ for $0\le i\le 3$,
if $Q$ is a $4$-chord, then $L_2(q_4)=\{\varphi(q_3),\varphi(q_4)\}$,
and $L_2(x)=L(x)$ for $x\not\in V(Q)$.
The graph $G_2$ is not $L_2$-colorable.  Furthermore, we have $d_{G_2}(q_0q_1q_2q_3)\ge r(q_0q_1q_2q_3)$, since
either $\ell(P)\ge 4$ and $d_{G_2}(q_0q_1q_2q_3)+(r(4)-r(3))\ge d(P)\ge r(P)$, or $d_{G_2}(q_0q_1q_2q_3)+(B - r(3))\ge B$.
By the minimality of $G$, we conclude that $G_2$ violates~(Q), hence a vertex $x$ with a list of size two is adjacent to
both $q_0$ and $q_3$.  Furthermore, by \refclaim{cl-inter} and \refclaim{cl-schords}, $G_2$ is equal to the $5$-face $q_0q_1q_2q_3x$.
\end{proof}

Let $P=p_0p_1\ldots p_m$, where $m=\ell(P)$.
We can assume that $m\ge 2$; otherwise, we can color $2-m$ vertices adjacent to $P$ in $C$ so that
the resulting list assignment $L'$ either still satisfies the assumptions of Theorem~\ref{thm-maingen} or violates~(OBSTa).
But, in the latter case, \refclaim{cl-inter} and \refclaim{cl-schords} would imply that $G$ with the list assignment $L'$ is equal
to one of the obstructions in Figure~\ref{fig-obsta}.  Then it is easy to see that $G$ either would be $L$-colorable or contain OBSTx1,
which is a contradiction.

A subgraph $H$ of $G$ is a {\em near-obstruction} if for one of the relevant configurations $(H',Q,f)$ drawn in Figure~\ref{fig-obsta} or \ref{fig-obstb},
there exists an isomorphism $\pi$ from $H'$ to $H$ such that $\pi(Q)=H\cap P$ and $|L(\pi(v))|\ge f(v)$ for every $v\in V(H')\setminus V(Q)$.
This near-obstruction is {\em tame} when every vertex $v\in V(H')\setminus V(Q)$ such that $f(v)=2$ and $\pi(v)$ has a neighbor in $P$
satisfies $\pi(v)\in V(C)$.

\claim{cl-notame}{The graph $G$ contains no tame near-obstruction.}
\begin{proof}
Suppose that $H$ is a tame near-obstruction in $G$, and let $K$ be the cycle bounding the outer face of $H$.
Let $Q_0=q_0q_1\ldots q_k$ be the subpath of $K$ vertex-disjoint with $P$ such that $V(K)\subseteq V(Q_0)\cup V(P)$.

Suppose first that both $q_0$ and $q_k$ are adjacent to an endvertex of $P$, say $q_0$ to $p_0$ and $q_k$ to $p_m$;
by the assumption that $d(P)\ge r(P)$ and that $H$ is tame and by \refclaim{cl-nochords}, this is the case unless
$H$ is OBSTx1 and $\ell(P)=2$.  Let $Q$ be the path consisting of $Q_0$ and those of the edges $q_0p_0$ and
$q_kp_m$ that do not belong to $C$.

Note that $|V(H)|<|V(G)|$, since otherwise either $G$ violates~(OBSTa) or~(OBSTb), or is $L$-colorable.
Let $G'=G-(V(H)\setminus V(Q))$.  Note that $G'$ is connected by \refclaim{cl-inter}.
By the minimality of $G$, the graph $H$ is $L$-colorable.  Let $\varphi$ be an $L$-coloring of $H$, and
let $L'$ be the list assignment such that $L'(x)=\{\varphi(x)\}$ if $x\in V(Q)$ and $L'(x)=L(x)$ otherwise.
Note that $G'$ is not $L'$-colorable, and by the minimality of $G$, it cannot satisfy the assumptions of Theorem~\ref{thm-maingen}.
But, clearly $G'$ satisfies~(I) and~(T).  Let us now discuss several cases; we always assume that the precolored
vertices of the drawing of $H$ in Figure~\ref{fig-obsta} or \ref{fig-obstb} are labeled from left to right, i.e.,
$p_0$ is the the leftmost precolored vertex in the drawing.

\begin{itemize}
\item {\bf $H$ is OBSTx2a or OBSTx2b:} Since $q_1p_2$ is not a chord by \refclaim{cl-nochords}, we have $q_1\not\in V(C)$.
By \refclaim{cl-schords}, the $2$-chord $q_0q_1p_2$ splits off one of the relevant configurations drawn in Figure~\ref{fig-obsta};
let $H'$ be its image.  Since $V(H)\neq V(G)$, $H'$ is not OBSTx1.
Since $H\subseteq G$, we have that $q_1$ has degree at least three in $H'$ and that $q_1$, $p_2$
and two vertices of a triangle are incident with a common $5$-face in $H'$.
This implies that $H'$ is OBSTa1, OBSTa3 or OBSTx4.  However, then $q_0$ is adjacent to a vertex with list
of size two in $H'$, and thus $|L(q_0)|=3$.  It follows that the $5$-cycle $p_0p_1p_2q_1q_0$ has at least
two $L$-colorings, and at least one of them extends to $H'$.  Therefore, $G$ is $L$-colorable, which is a contradiction.

\item {\bf $\mathbf{\ell(Q)\le 5}$:}  Since $t(G)\ge B$ and $d(P)\ge r(P)$ if $\ell(P)\ge 3$, no vertex of $Q$ is contained in a $(\le\!4)$-cycle.
The inspection of the graphs depicted in Figures~\ref{fig-obsta} and \ref{fig-obstb} shows that among any three consecutive
internal vertices of $Q$, at least one has degree two in $H$.  This implies that $Q$ is an
induced path in $G$, since otherwise by \refclaim{cl-inter}, $G$ would contain a vertex of degree two with list of size three.
Similarly, we conclude that in $G$, no vertex with list of size two has two neighbors in $Q$, unless $H$ is OBSTa1
(or OBSTx2a, but that was already excluded).
However, if $H$ is OBSTa1 and $q_0$ and $q_3$ have a common neighbor $x$ with list of size two, then \refclaim{cl-inter} and \refclaim{cl-schords}
imply that $V(G)=V(H)\cup \{x\}$, and it is easy to see that $G$ is $L$-colorable.
We conclude that $G'$ satisfies~(S3) and~(Q).

Let us discuss several subcases regarding $m$:

\begin{itemize}
\item $m=2$: That is, $H$ is one of the obstructions drawn in Figure~\ref{fig-obsta}, except for OBSTa5, OBSTx1 or OBSTx3
(or OBSTx2a or OBSTx2b, which were already excluded).  Note that in all these cases, $\ell(Q)\le4$.
It follows that $G'$ violates~(OBSTb).  Since $H$ contains a triangle whose distance from any vertex of $Q$ is at most three, $G'$ satisfies $d(Q)\ge r(Q)$.  
Therefore, $\ell(Q)=4$, $H$ is OBSTa3, OBSTa4, OBSTa6, OBSTa7, OBSTx1a, OBSTx1b or OBSTx4 and $G'$ is OBSTb1 or OBSTb2.
Since $G$ does not contain a vertex of degree two with list of size three, if $G'$ is OBSTb2, then $H$ is OBSTa7.
The case analysis of the possible combinations of $H$ and $G'$ shows that $G$ is $L$-colorable, which is a contradiction.

\item $m=4$: The case that $H$ is OBSTb1 is excluded by \refclaim{cl-schords}, since $d(P)\ge d(T)$, thus $H$ is OBSTb2.
\refclaim{cl-schords} furthermore implies that $|L(q_2)|=3$, and thus we may choose the $L$-coloring $\varphi$ so that $\varphi(q_1)\not\in L(q_0)\setminus L(p_0)$.
Let $L''$ be the list assignment defined by $L''(q_0)=(L(q_0)\setminus L(p_0))\cup \{\varphi(q_1)\}$ and $L''(x)=L'(x)$
otherwise.  Note that only a path $q_1q_2q_3q_4$ of length three is precolored in $G'$ according to this list assignment and
$d(q_1q_2q_3q_4)\ge d(P)-3\ge r(P)-3\ge r(q_1q_2q_3q_4)$ and thus $G'$ is $L''$-colorable.  This gives an $L$-coloring of $G$,
which is a contradiction.

\item $m=5$: By \refclaim{cl-2pchord}, $H$ cannot be OBSTb3 or OBSTb4.  Thus, $H$ is
OBSTb1a, OBSTb1b, OBSTb2a, OBSTb2a', OBSTb2b or OBSTb5, and $\ell(Q)\le4$.
We conclude that $G'$ is OBSTb1 or OBSTb2 and $\ell(Q)=4$ (excluding the cases that
$H$ is OBSTb1a or OBSTb1b).  Note that $q_2$ has degree two in $H$,
and since it has degree at least three in $G$, we conclude that $G'$ is OBSTb1.
The case analysis of the possible combinations of $H$ and $G'$ shows that $G$ is $L$-colorable, which is a contradiction.
\end{itemize}

\item {\bf $\mathbf{\ell(Q)>5}$:} Thus, $H$ is OBSTa5, OBSTx3 or OBSTb6.  Let us discuss these cases separately:

\begin{itemize}
\item $H$ is OBSTa5:  Let $w$ be the common neighbor of $q_1$ and $q_6$, and $w'$ the common neighbor of
$w$, $q_3$ and $q_4$.  If there exist colors
$c_1\in L(q_1)\setminus (L(q_0)\setminus L(p_0))$ and $c_2\in L(q_6)\setminus (L(q_7)\setminus L(p_2))$
so that $L(w)\neq L(p_1)\cup \{c_1,c_2\}$, then consider the graph $G_1=G-V(P)$ with the list assignment
$L_1$ such that $L_1(q_1)=\{c_1\}$, $L_1(q_6)=\{c_2\}$, $L_1(w)$ chosen as an arbitrary one-element subset of $L(w)\setminus (L(p_1)\cup \{c_1,c_2\})$,
$L_1(q_0)=(L(q_0)\setminus L(p_0))\cup \{c_1\}$, $L_1(q_7)=(L(q_7)\setminus L(p_2))\cup \{c_2\}$ and $L_1(x)=L(x)$ otherwise.
The graph $G_1$ (with the precolored path $q_1wq_6$) cannot be $L_1$-colorable, thus it violates~(OBSTa).  This is only possible if $G_1$ contains OBSTa1, but then
$V(G)=V(H)$ and thus $G$ is $L$-colorable.

So, we have $|L(q_0)|=|L(q_7)|=3$, $L(q_1)=(L(q_0)\setminus L(p_0))\cup \{c_1\}$, $L(q_6)=(L(q_7)\setminus L(p_2))\cup \{c_2\}$ and
$L(w)=L(p_1)\cup \{c_1,c_2\}$.
Let $\psi$ be an $L$-coloring of $q_1q_0p_0p_1p_2q_7q_6$ such that $\psi(q_1), \psi(q_6)\not\in L(w)\setminus L(p_1)$.
Let $G_2=G-(V(P)\cup\{w'\})$, with the list assignment $L_2$ such that $L_2(x)=\{\psi(x)\}$ for $x\in \{q_0,q_1,q_6,q_7\}$, $L_2(w)$ is an
arbitrary singleton list disjoint with $L_2(q_1)$ and $L_2(q_6)$, and $L_2(x)=L(x)$ otherwise.  
Since an $L_2$-coloring of $G_2$ corresponds to an $L$-coloring of $G$
(choosing a color of $w'$ different from the colors of $q_3$ and $q_4$, then a color of $w$ different from the colors of $p_1$ and $w'$),
we have that $G_2$ is not $L_2$-colorable.  By \refclaim{cl-inter}, $G_2$ satisfies~(S3) and~(Q).
Since $d_G(q_3q_4w')\ge B$, we have $d_{G_2}(q_0q_1wq_6q_7)\ge B-3\ge r(q_0q_1wq_6q_7)$.
The internal face of $G_2$ incident with $w$ has length at least six, and thus $G_2$ satisfies~(OBSTb).
Therefore, $G_2$ is a counterexample to Theorem~\ref{thm-maingen} smaller than $G$, which is a contradiction.

\item $H$ is OBSTx3: Let $q_1w_1w_2q_3$ be the path in $H$ such that $w_1,w_2\neq q_2$.
If $|L(q_0)|=2$, then consider an $L$-coloring $\psi$ of the subgraph of $G$ induced by $\{q_0, q_1,w_1,w_2,p_0, p_1\}$
such that $\psi(w_2)\not\in L(q_7)\setminus L(p_2)$.  Let $L'$ be the list assignment defined by
$L'(q_0)=\{\psi(q_0),\psi(q_1)\}$, $L'(x)=\{\psi(x)\}$ for $x\in\{q_1,w_1,w_2\}$, $L'(q_7)=(L(q_7)\setminus L(p_2))\cup \{\psi (w_2)\}$
and $L'(x)=L(x)$ otherwise.  We conclude that $G-V(P)$ with the precolored path $q_1w_1w_2$ is not $L'$-colorable, thus it violates~(OBSTa).
Note that $w_1$ has degree two in $G-V(P)$ and the face with which it is incident does not share any
vertex with a triangle, and that $q_7$ is not incident with a triangle, thus $G-V(P)$ contains OBSTx2a.
By \refclaim{cl-inter} and \refclaim{cl-schords}, $G-V(P)$ is equal to OBSTx2a.  However, then $q_2$, $q_5$ and $q_7$ have lists
of size two and $G$ contains OBSTx3, which is a contradiction.

So, we have $|L(q_0)|=3$.  Then, there exist
$c_1\in L(q_1)\setminus (L(w_1)\setminus L(p_1))$ and $c_0\in L(q_0)\setminus L(p_0)$ such that $c_0\neq c_1$.
Let $G_1$ be the graph obtained from $G-\{p_0,p_1, w_1,w_2\}$ by adding the edge
$q_1p_2$.  Let $c$ be a color that does not appear in any of the lists of $G$.
Let $L_1$ be the list assignment such that $L_1(q_0)=\{c_0\}$, $L_1(q_1)=\{c_1\}$,
$L_1(p_2)=\{c\}$, $L_1(q_7)=(L(q_7)\setminus L(p_2))\cup \{c\}$ and $L_1(x)=L(x)$ for all other vertices of $G_1$.
Observe that $G_1$ with the precolored path $q_0q_1p_2$ is not $L_1$-colorable.  Furthermore, the distance of $q_1$ from the triangle
$q_4q_5q_6$ is three both in $G$ and $G_1$, and the distance of $q_1$ and $q_7$ to any other $(\le\!4)$-cycle
is at least $B-3$, thus $t(G_1)\ge B$.  The internal face $F$ of $G_1$ incident with $q_1p_2$ has length at least six,
as otherwise the cycle $F-q_1p_2+q_1w_1p_1p_2$ has length at most seven and contradicts \refclaim{cl-inter}.
Furthermore, observe that neither $q_0$ nor $q_1$ is adjacent to a vertex of the triangle $q_4q_5q_6$,
thus $G_1$ contains neither OBSTx1 nor OBSTx1a.  It follows that $G_1$ satisfies~(OBSTa), and thus it is a counterexample to Theorem~\ref{thm-maingen} smaller than $G$.
This is a contradiction.

\item $H$ is OBSTb6: Let $q_1w_1w_2p_3$ be the path in $H$
with $w_1$ adjacent to $p_1$.  If $|L(q_6)|=2$, then let $c'$ be the unique color in $L(q_6)\setminus L(p_5)$, and note that there exists
$c\in L(q_5)\setminus (L(p_3)\cup \{c'\})$.  Let $G_1=G-\{p_4,p_5\}$ and let $L_1$ be the list assignment such that
$L_1(q_5)=\{c\}$, $L_1(q_6)=\{c,c'\}$ and $L_1(x)=L(x)$ for $x\not\in\{q_5,q_6\}$.  Note that $G_1$ with the precolored path $p_0p_1p_2p_3q_5$ is not $L_1$-colorable,
and since $H$ is a subgraph of $G$, we
conclude that $G_1$ contains OBSTb2.  However, this implies that $G$ contains OBSTb6, which is a contradiction.

Therefore, $|L(q_6)|=3$.  Then, there exists an $L$-coloring $\psi$ of the subgraph of $G$ induced by $\{q_3,q_4,q_5,q_6, p_3, p_5\}$ such that
$\psi(q_3)\not\in L(w_2)\setminus L(p_3)$.  Let $G_2$ be the graph obtained from $G-(V(P)\cup\{w_1,w_2\})$ by adding a vertex $w$
adjacent to $q_0$ and $q_3$.  Let $c$ be a new color that does not appear in $L(q_0)\cup L(q_3)$.
Let $L_2$ be the list assignment such that $L_2(x)=\psi(x)$ for $x\in\{q_3, q_4, q_5, q_6\}$, $L_2(w)=\{c\}$,
$L_2(q_0)=(L(q_0)\setminus L(p_0))\cup \{c\}$ and $L_2(x)=L(x)$ otherwise.  Observe that an $L_2$-coloring of $G_2$
corresponds to an $L$-coloring of $G$, thus $G_2$ is not $L_2$-colorable.
Furthermore, a path $P_2=wq_3q_4q_5q_6$ of length $4$ is precolored in $G_2$.
Let us remark that the newly added vertex $w$ is not incident with a $(\le\!4)$-cycle, as otherwise either $t(P)<r(P)$ in $G$, or
\refclaim{cl-inter} implies that $q_2$ is a vertex of degree two with list of size three.  Furthermore, $t(G_2)\ge B$,
since only the added path $q_0wq_3$ could result in shortening the distance between $(\le\!4)$-cycles, and we have
$d_G(q_0)\ge d(P)-1\ge r(P)-1$ and $d_G(q_3)\ge d(P)-2\ge r(P)-2$, and $2r(P)-1 > B$.  Also, $d_{G_2}(P_2)\ge d_G(P)-2\ge r(P_2)$.

Note that $G_2$ satisfies~(S3), since $w$ is not adjacent to $q_6$ and $d(P)\ge r(P)$.  Similarly,
$G_2$ satisfies~(Q), since otherwise \refclaim{cl-inter} would imply that $q_4$ is a vertex of degree two with list of size three.
Hence, $G_2$ violates~(OBSTb).  Since $q_4$ has degree at least three, $G_2$ contains OBSTb1.  But then $q_4$ and $q_0$
have a common neighbor $x$, and the existence of $q_2$ together with $d(P)\ge r(P)$ contradicts \refclaim{cl-inter}
applied to the $7$-cycle $q_0q_1w_1w_2q_3q_4x$.
\end{itemize}
\end{itemize}

Therefore, we have excluded all the cases where both $q_0$ and $q_k$ are adjacent to an endvertex of $P$.
Finally, let us consider the case that say $q_0$ is not adjacent to an endvertex of $P$, that is, $\ell(P)=2$,
$H$ is OBSTx1, $q_0$ is adjacent to $p_1$ and $q_3$ is adjacent to $p_2$.  An $L$-coloring of $H$ does not
extend to an $L$-coloring of the subgraph $G'$ that is split off by the path $p_0p_1q_0q_1q_2q_3$.  If $p_0$
and $q_1$ have a common neighbor with list of size two, then either $G$ is $L$-colorable
or contains OBSTa1.  Otherwise, $G'$ satisfies~(S3) and~(Q), as $q_1$ cannot be a vertex of degree two
with list of size three.  Therefore, $G'$ violates~(OBSTb).  If $|L(q_3)|=2$, then $G'$ may only be OBSTb1, OBSTb1b, OBSTb2 or OBSTb2b.
The cases OBSTb1 and OBSTb1b are excluded, since $q_1$ must have degree at least three; if $G'$ is OBSTb2, then $G$ is $L$-colorable,
and if $G'$ is OBSTb2b, then $G$ contains OBSTa3.  If $|L(q_3)|=3$, then there exist $L$-colorings $\psi_1$ and $\psi_2$ of
$H$ such that $\psi_1(q_0)=\psi_2(q_0)$, $\psi_1(q_1)\neq \psi_2(q_1)$, $\psi_1(q_2)\neq\psi_2(q_2)$ and $\psi_1(q_3)\neq\psi_2(q_3)$.
The inspection of the graphs in Figure~\ref{fig-obstb} shows that at least one of $\psi_1$ and $\psi_2$ extends to
an $L$-coloring of $G'$, unless $G'$ contains a subgraph $H'$ isomorphic to OBSTb1, OBSTb1a, OBSTb1b, OBSTb3 or OBSTb5.
By \refclaim{cl-inter} and \refclaim{cl-schords} we conclude that $G'=H$ and $G=H\cup H'$.  However, all possible combinations
of $H$ and $H'$ result in an $L$-colorable graph, which is a contradiction.
\end{proof}

Let $v_1$, $v_2$, \ldots, $v_s$ be the vertices of $C-V(P)$ labeled so that $C=p_0\ldots p_mv_1v_2\ldots v_s$, 
where $s=\ell(C)-m-1$.  Let us also define $v_0=p_m$.

\claim{cl-nopchord}{For $1\le i\le 4$, if the edge $v_{i-1}v_i$ is not contained in a cycle of length at most $4$
and a vertex $v\in V(G)$ is adjacent to both $v_i$ and an endvertex $p$ of $P$, then $v\in V(C)$.}
\begin{proof}
Suppose that $v\not\in V(C)$.  Let $G_2$ be the subgraph of $G$ that is split off by the $2$-chord $v_ivp$ according to \refclaim{cl-schords},
and $G_1=G-(V(G_2)\setminus \{v_i,v,p\})$.  If $p=p_m$, then $i\in \{3,4\}$, since $v_{i-1}v_i$ does not belong to a $(\le\!4)$-cycle. 
By \refclaim{cl-inter} and the fact that every vertex of degree two has list of size two, we have that
$i=4$ and $G_2$ contains a triangle.  It follows that $m\le 2$.  Consider an $L$-coloring $\psi$ of $G_2$,
and let $L_1$ be the list assignment such that $L_1(v)=\{\psi(v)\}$, $L_1(v_4)=\{\psi(v_i)\}$ and $L_1(x)=L(x)$ for
$x\not\in\{v,v_4\}$.  Note that $G_1$ with the precolored path $p_0\ldots p_mvv_i$ is not $L_1$-colorable.
By \refclaim{cl-nochords}, \refclaim{cl-length}, \refclaim{cl-schords} and the assumption that
$v\not\in V(C)$, we conclude that $G_1$ satisfies~(S3) and~(Q).  It follows that $G_1$ violates (OBSTb),
and by \refclaim{cl-inter} and \refclaim{cl-schords}, $G_1$ is equal to OBSTb1 or OBSTb2.
However, all combinations of OBSTb1 or OBSTb2 with a $p_mv_1v_2v_3v_4v$-obstacle are $L$-colorable.

Let us now consider the case that $p=p_0$.  Since a $(\le\!4)$-cycle in $G_2$ is within distance at most $4$ from $P$, we have $\ell(P)=2$.
Let $K$ be the cycle of length at most $8$ formed by the $2$-chord $v_ivp_0$, the path $P$, and the vertices $v_1$, $v_2$, \ldots, $v_i$.  Since
$t(G)\ge B$, $G_1$ cannot be a $K$-obstacle, and if $K$ is not a face, then $\ell(K)=8$ and $K$ has a chord splitting $K$ into two $5$-faces.
If $K$ is not a face, then since each vertex with list of size three has degree at least three, we conclude that
$|L(v_1)|=|L(v_3)|=2$, $|L(v_2)|=3$ and the chord of $K$ is $v_2p_0$.  However, this contradicts \refclaim{cl-nochords}.
Therefore, $K$ is a face.  Since $v$ has degree at least three, $G_2$ is not a face.  Furthermore, $G_2$ is not~OBSTx1b, thus $|L(v_i)|=3$.  Hence,
there exist $L$-colorings $\psi_1$ and $\psi_2$ of $K$ such that $\psi_1(v)\neq \psi_2(v)$ and $\psi_1(v_i)\neq \psi_2(v_i)$.
The inspection of the graphs in Figure~\ref{fig-obsta} shows that at least one of $\psi_1$ and $\psi_2$ extends to
an $L$-coloring of $G_2$, giving an $L$-coloring of $G$.  This is a contradiction.
\end{proof}

\claim{cl-nopchordm}{Suppose that $m=5$.  For $1\le i\le 4$,
if a vertex $v\in V(G)$ is adjacent to both $v_i$ and to $p\in\{p_1,p_4\}$, then $v\in V(C)$,
unless $p=p_4$ and $i=2$, or $p=p_1$ and $i=s-1$.}
\begin{proof}
Suppose that $v\not\in V(C)$ is adjacent to $p_4$ and $v_i$.  Since $d(P)\ge r(P)$ and every vertex with list
of size three has degree at least three, \refclaim{cl-inter} implies that $i=2$.

Hence, assume that $v\not\in V(C)$ is adjacent to $p_1$ and $v_i$.  Let $Q=p_0p_1vv_i$, let $G_1$ be
the subgraph of $G$ drawn in the cycle bounded by $vp_1\ldots p_5v_1\ldots v_i$ and $G_2=G-(V(G_1)\setminus V(Q))$.
By the minimality of $G$, there exists an $L$-coloring $\varphi$ of $G_1$.  Let $L_2$ be the list assignment
such that $L_2(x)=\varphi(x)$ for $x\in\{v,v_i\}$ and $L_2(x)=L(x)$ otherwise; the graph $G_2$ cannot be $L_2$-colorable.
Since only an induced path $Q$ of length three is precolored in $G_2$ (and $d_{G_2}(Q)\ge d_G(P)-2\ge r(P)-2\ge r(Q)$), we
conclude that $G_2$ violates~(Q), thus there exists a vertex $w$ with list of size two adjacent to $p_0$ and $v_i$.
By \refclaim{cl-nochords}, we have $C=p_0p_1\ldots p_5v_1\ldots v_iw$, and thus $i=s-1$.
\end{proof}

\claim{cl-nosimt}{If $v_i$ has degree two and is incident with a triangle, then $i\ge 4$.
Furthermore, if $4\le i\le 6$, $v_i$ has degree two and is incident with a triangle, then $|L(v_{i+2})|\neq 2$.}
\begin{proof}
Suppose first that $v_i$ is such a vertex, with $1\le i\le 3$.  Clearly, this is only possible if $\ell(P)=2$.
By the minimality of $G$, the subgraph $G_0$ of $G$ induced by $V(P)\cup \{v_1,\ldots, v_{i+1}\}$ has an $L$-coloring $\psi$.
Let $L'$ be the list assignment such that $L'(x)=\{\psi(x)\}$ for $x\in \{v_1,\ldots, v_{i+1}\}$ and
$L'(x)=L(x)$ otherwise, and let $Q=p_0p_1p_2v_1\ldots v_{i-1}v_{i+1}$.  Let $G'=G-v_i$.  Then, $G'$ is not $L'$-colorable.
Furthermore, by \refclaim{cl-nochords} and \refclaim{cl-length}, $G'$ with the precolored path $Q$ satisfies~(Q).
Note that $d_{G'}(Q)\ge d_G(v_{i-1}v_iv_{i+1})-4\ge B-4\ge r(Q)$.
Observe that $G'$ violates~(OBSTb),
and by \refclaim{cl-inter} and \refclaim{cl-schords}, $G'$ is equal to one of the graphs drawn in Figure~\ref{fig-obstb}.
If $i=2$, then either $G'$ is OBSTb1 and thus $G$ contains OBSTx2b, or $G'$ is OBSTb2 and $G$ is $L$-colorable.
Therefore, $i=3$.  If $|L(v_1)|=3$, then we can assume that $\psi(v_2)\not\in L(v_1)\setminus L(p_2)$, thus there exist two
$L$-colorings of the subgraph of $G_0$ that differ only in the color of $v_1$.
Furthermore, the degree of $v_1$ in $G'$ is at least three.  The inspection of the graphs drawn in Figure~\ref{fig-obstb} shows that at least one of these colorings extends to $G'$,
which is a contradiction.  If $|L(v_1)|=2$, then by~(T) we have that $v_{i+1}$ has no neighbor with list
of size two in $G'$; hence, either $G'$ is OBSTb1b and $G$ contains OBSTx2a, or
$G'$ is OBSTb2b and $G$ contains OBSTx3.

Suppose now that $4\le i\le 6$ and $|L(v_{i+2})|=2$.  Again, $m=2$.  By~(T), $|L(v_{i-2}|=3$, and by \refclaim{cl-nochords},
$Q_0=p_0p_1p_2v_1\ldots v_{i-1}v_{i+1}$ is an induced path.  Thus, there exists
an $L$-coloring $\psi$ of this path such that $L(v_i)\neq \{\psi(v_{i-1}),\psi(v_{i+1})\}$ and $\psi(v_{i+1})\not\in L(v_{i+2})$.
Let $G'=G-\{v_{i-1},v_i, v_{i+1}\}$ with the list assignment $L'$ such that $L'(v_j)=\{\psi(v_j)\}$ for $1\le j\le i-3$,
$L'(v_{i-2})=\{\psi(v_{i-3}),\psi(v_{i-2})\}$, $L'(x)=L(x)\setminus\{\psi(y)\}$ for a vertex $x\in V(G')$ with a neighbor
$y\in \{v_{i-1},v_{i+1}\}$ and $L'(x)=L(x)$ otherwise.  The graph $G'$ is not $L'$-colorable.  Furthermore, by
\refclaim{cl-nochords} and \refclaim{cl-schords}, $G'$ with the list assignment $L'$ and 
the precolored path $Q=p_0p_1p_2v_1\ldots v_{i-3}$ satisfies (S2) and (I).  Let $w$ be a common neighbor
of two vertices of $Q$.  By \refclaim{cl-nochords}, we have $w\neq v_{i-2}$
and $|L(w)|=3$.  Furthermore, $|L'(w)|=3$, since otherwise $w$ would be adjacent to $v_{i-1}$ or $v_{i+1}$ as
well, and \refclaim{cl-inter} would imply that $v_{i-2}$ has degree two in $G$.  This shows that~(Q) is true.
Note that $d_{G'}(Q)\ge B-7 > r(P)$ and $G'$ violates~(OBSTb).  This implies that $i\ge 5$. 

Note that $v_{i-4}$ is not adjacent to a vertex $x$ with $|L'(x)|=2$ and
that $v_{i-2}$ is the only such vertex adjacent to $v_{i-3}$, by \refclaim{cl-nochords}, \refclaim{cl-inter} and the fact that $v_{i-2}$ has degree at least three in $G$.
Also, observe that there exist
at least two possible choices of the coloring $\psi$ that differ in the colors of $v_{i-2}$ and $v_{i-3}$,
may also differ on $v_{i-4}$, but match on all other vertices of $Q_0$ (and thus the list assignments
derived from them as in the definition of $L'$ differ only in the lists of $v_{i-2}$, $v_{i-3}$ and possibly $v_{i-4}$).
Since neither of these colorings extends to $G'$,
an inspection of the graphs depicted in Figure~\ref{fig-obstb} shows that $i=6$ and $G'$ contains OBSTb3.
If $v_8$ is adjacent to $p_0$, then $G$ contains OBSTx3.  Otherwise, \refclaim{cl-nochords} and \refclaim{cl-schords} imply that the edge
of OBSTb3 incident with $v_{i-2}$ (distinct from $v_{i-3}v_{i-2}$) is a chord of $C$ that splits off OBSTx1 in $G$; however, the resulting graph is $L$-colorable.
\end{proof}

\claim{cl-lv12l}{We have $|L(v_1)|=2$ or $|L(v_2)|=2$.}
\begin{proof}
Suppose that $|L(v_1)|=|L(v_2)|=3$.  Let $L'$ be the list assignment such that $L'(v_1)=L(v_1)\setminus L(p_m)$
and $L'(x)=L(x)$ otherwise.  Let $G'=G-p_mv_1$.  By \refclaim{cl-nochords}, $G'$ with the list assignment $L'$ (and the precolored path $P$)
satisfies~(I).
Suppose that~(T) is violated.  Then there exists a triangle $w_1w_2w_3$ such that either $v_1=w_2$ and both $w_1$ and $w_3$
have a neighbor with list of size two, or $|L(w_2)|=2$, $w_1$ is adjacent to $v_1$ and $w_3$ has a neighbor $w$ distinct from
$w_1$ with list of size two.  By \refclaim{cl-schords}, the former is not possible, and in the latter case,
we have $w_1=v_2$, $w_2=v_3$ and $w_3=v_4$.  However, that contradicts \refclaim{cl-nosimt}.  Therefore, (T) holds.
Furthermore, by \refclaim{cl-nochords}, $v_1$ is not adjacent to any vertex of $P$ other than $p_m$, and thus~(Q) is satisfied.
Since an $L'$-coloring of $G'$ would give an $L$-coloring of $G$, it follows that
$G'$ with the assignment $L'$ violates~(OBSTa) and~(OBSTb).  However, this implies that $G$ with the list
assignment $L$ contains a tame near-obstruction $H$, contradicting \refclaim{cl-notame}.
\end{proof}

\claim{cl-5imp10}{If $\ell(P)=5$, then $\ell(C)\ge 10$.}
\begin{proof}
By \refclaim{cl-length}, we have $\ell(C)\ge 9$.  Suppose that $\ell(C)=9$.  By \refclaim{cl-lv12l}, either $|L(v_1)|=2$ or $|L(v_2)|=2$.
Applying \refclaim{cl-lv12l} symmetrically on the other end of $P$, we also have that $|L(v_2)|=2$ or $|L(v_3)|=2$.  Therefore,
either $|L(v_1)|=|L(v_3)|=2$ and $|L(v_2)|=3$, or $|L(v_1)|=|L(v_3)|=3$ and $|L(v_2)|=2$.  In the former case, $L$-color the path $v_1v_2v_3$ so that $v_1$
gets a color different from the color of $p_5$ and $v_3$
a color different from the color of $p_0$.  Let $G'=G-\{v_1,v_2,v_3\}$, with the list assignment $L'$ obtained
from $L$ by removing the colors of the vertices $v_1$, $v_2$ and $v_3$ from the lists of their neighbors.
Note that $G'$ satisfies~(I), since otherwise $v_1v_2v_3$ would be a part of a $5$-cycle, and by \refclaim{cl-inter},
$v_2$ would have degree two.  Furthermore, (T) is satisfied since $d(P)\ge r(P)$ and~(Q) is satisfied by
\refclaim{cl-2pchord}.  Note also that no vertex adjacent to $p_0$ or $p_5$ has list of size $2$, thus $G'$ satisfies~(OBSTb). 
This is a contradiction, since an $L'$-coloring of $G'$ corresponds to an $L$-coloring of $G$.

In the latter case, let $G'$ be the graph with list assignment $L'$ obtained from $G$ by coloring $v_2$ from its
list arbitrarily, removing $v_2$ and removing its color from the lists of its neighbors.  Again,
(I), (T) and~(Q) are obviously satisfied by $G'$.  Furthermore, since $d(P)\ge r(P)$, the distance between any
pair of vertices of $G'$ with list of size two is at least three.  This implies that $G'$ satisfies~(OBSTb),
unless it contains OBSTb1b.  However, that is excluded by \refclaim{cl-2pchord}.
\end{proof}

Let $X$ be the set of vertices defined as follows.
If $|L(v_1)|=3$ (and thus $|L(v_2)|=2$ by \refclaim{cl-lv12l} and $|L(v_3)|=3$)
and $|L(v_4)|=3$, then $X=\{v_2\}$.  If $|L(v_1)|=3$ and $|L(v_4)|=2$, then $X=\{v_2,v_3\}$.
If $|L(v_1)|=2$ (and thus $|L(v_2)|=3$) and $|L(v_3)|=3$, then $X=\{v_1\}$.  If $|L(v_1)|=|L(v_3)|=2$ (and thus $|L(v_4)|=3$)
and $v_5=p_0$ or $|L(v_5)|=3$, then $X=\{v_2,v_3\}$.  Otherwise, $X=\{v_2,v_3,v_4\}$.

\claim{cl-nonearchords}{Let $Q=q_0q_1\ldots q_k$ be a $k$-chord of $C$ such that no endvertex of $Q$ is an internal vertex of $P$
and $Q$ does not split off a face.
If $k\le 2$, or if $k=3$ and $q_3$ has list of size two, then $q_0\not\in X$.}
\begin{proof}
We exclude a slightly more general case.  Suppose that $Q$ is a $k$-chord of
$C$ such that no endvertex of $Q$ is an internal vertex of $P$ and $Q$ does not split off a face,
such that either $k\le 2$, or $k=3$ and $q_3$ has list of size two.  Suppose furthermore that
either $q_0\in X$, or $q_0=v_1$ and $q_k\not\in \{v_2,v_3,v_4\}$.

Let $G_2$ be the subgraph of $G$ that is split off by $Q$ and $G_1=G-(V(G_2)\setminus V(Q))$.  Let $Q$ be chosen so
that $G_2$ is as large as possible.  Let $i$ be the index such that $v_i=q_0$.
By \refclaim{cl-schords} we can assume that $\ell(P)=2$, since otherwise $G_2$ contains a triangle whose distance
from $q_0$ is at most four, hence its distance from $P$ is at most $8$, contradicting $d(P)\ge r(P)$.

By \refclaim{cl-nochords} and \refclaim{cl-nosimt}, the path consisting of $P$ and $v_1v_2v_3v_4$ is induced.
We now distinguish two cases regarding whether $q_k$ belongs to $\{v_1,v_2,v_3,v_4\}$ or not.
\begin{itemize}
\item Let us first consider the case that
$q_k\in \{v_1,v_2,v_3,v_4\}$, and let $K$ be the cycle bounded by $Q$ and a subpath of $v_1v_2v_3v_4$.
Since $Q$ does not split off a face, \refclaim{cl-inter} implies that $\ell(K)\ge 6$, thus $k=3$ and $\{q_0,q_k\}=\{v_1,v_4\}$.
If $q_0=v_1\in X$, then $|L(v_1)|=2$ and $|L(v_2)|=|L(v_3)|=3$ by the choice of $X$.
However, \refclaim{cl-inter} implies that $v_2$ or $v_3$ has degree two,
which is a contradiction.

If $q_0=v_4\in X$, then $|L(v_1)|=|L(v_3)|=|L(v_5)|=2$ by the choice of $X$.
Furthermore, \refclaim{cl-inter} and \refclaim{cl-nosimt} imply
that either $v_2q_2\in E(G)$, or $v_2$, $q_2$ and $q_0$ are adjacent to vertices of a triangle $T$.
\begin{itemize}
\item In the former
case, let $\psi_1$ and $\psi_2$ be $L$-colorings of the subgraph of $G$ induced by $V(P)\cup \{v_1,v_2,q_2\}$
such that $\psi_1(v_1)=\psi_2(v_1)$, $\psi_1(v_2)\neq \psi_2(v_2)$ and $\psi_1(q_2)\neq \psi_2(q_2)$,
let $G'=G-v_1v_2$ and let $L_1$ and $L_2$ be the list assignments such that $L_j(x)=\{\psi_j(x)\}$ for $x\in \{v_1,v_2,q_2\}$
and $L_j(x)=L(x)$ otherwise.  Note that $G'$ with the precolored path $p_0p_1p_2v_1q_2v_2$ satisfies~(Q) by \refclaim{cl-schords} and that $G'$ is not $L_j$-colorable for $j\in\{1,2\}$,
thus $G'$ with both of these assignments violates~(OBSTb).  This is only possible if $G'$ contains OBSTb3, but then $G$ contains
OBSTx4.

\item In the latter case, let $t_1$ and $t_2$ be the vertices of $T$ adjacent to $v_2$ and $v_4$, respectively, 
let $\psi$ be an $L$-coloring of $p_mv_1v_2v_3v_4$ such that either $\psi(v_2)\not\in L(t_1)$ or
$L(t_1)\setminus \{\psi(v_2)\}\neq L(t_2)\setminus \{\psi(v_4)\}$,
and let $G'$ be the graph obtained from $G-V(T)$ by identifying $v_2$ with $v_3$ to a new vertex $z$.
Note that $z$ is not contained in a $(\le\!4)$-cycle by \refclaim{cl-inter}, and observe that $t(G')\ge B$.
let $L'$ be the list assignment defined in the following way: $L'(v_i)=\{\psi(v_i)\}$ for $i\in\{1,4\}$,
$L'(z)=\{c\}$ for a new color $c$ that does not appear in any of the lists,
and $L'(x)=L(x)$ for any other vertex $x\in V(G')$.  Observe that $G'$ with the precolored path
$p_0p_1p_2v_1zv_2$ is not $L'$-colorable and satisfies~(Q)
by \refclaim{cl-nochords} and \refclaim{cl-length}, hence $G'$ violates~(OBSTb).  Let $H$ be
the image of the corresponding relevant configuration.  Since $q_1$ has degree at least three,
\refclaim{cl-inter} implies that $v_1zv_4q_1q_2$ is the only cycle of length at most $5$ in $G'$ containing $z$,
and that every cycle of length $6$ containing $z$ also contains $q_1$. It follows that $q_1\in V(H)$.
Note that $H$ is neither OBSTb1b nor OBSTb2b, since in these cases, a $(\le\!3)$-chord
contained in the outer face of $H$ incident with $v_4$ would contradict \refclaim{cl-schords}.
Then, $|L'(q_1)|=3$ implies that $v_5\in V(H)$, thus $v_4$ has degree at least three in $H$.
The only relevant configuration among those drawn in Figure~\ref{fig-obstb} in that an endvertex of the
precolored path has degree greater than two is OBSTb4, however $H$ is not OBSTb4 since $q_1$ is not adjacent to $p_m$.
\end{itemize}

\item Next, consider the case that
$q_k\not\in \{v_1,v_2,v_3,v_4\}$.  By \refclaim{cl-schords}, $G_2$ is one of the graphs depicted in Figure~\ref{fig-obsta}. 
Observe that there exists a color $c\in L(q_0)$ such that every $L$-coloring of $Q$ that assigns $c$ to $q_0$
extends to an $L$-coloring of $G_2$.  Suppose first that there exists an $L$-coloring $\psi$ of the path
$P'=p_0p_1p_2v_1\ldots v_i$ such that $\psi(q_0)=c$.  Let $L_1$ be the list assignment such that
$L_1(x)=\{\psi(x)\}$ for $x\in \{v_1,\ldots, v_{i-1}\}$, $L_1(v_i)=\{\psi(v_i),\psi(v_{i-1})\}$ and $L_1(x)=L(x)$ otherwise.
Note that the path $P_1=P'-v_i$ that is precolored in $G_1$ has length at most $5$.
Furthermore, $G_2$ contains a triangle whose distance from $v_i$ is at most $4$, thus $d_{G_1}(P_1)\ge B-10\ge r(P_1)$.
Since $G$ is not $L$-colorable, $G_1$ is not $L_1$-colorable.  By \refclaim{cl-nochords}, $G_1$ satisfies~(I) and~(Q).
Note that the distance in $G_1$ from $v_i$ to any triangle is at least $B-4>1$, thus $G_1$ satisfies~(T).  We conclude that $G_1$ violates~(OBSTb), and thus
$i\in\{3,4\}$.  The choice of $Q$ (so that $G_2$ is maximal) implies that if $Q'$ is a $(\le 3)$-chord in $G_1$ from a vertex $v_j$ with
$1\le j\le i$ to a vertex $x$ with list of size two, then $Q'$ splits off a face and $x=v_i$.  The inspection
of the graphs in Figure~\ref{fig-obstb} shows that $G_1$ can only satisfy this condition if it contains
OBSTb1, OBSTb1a or OBSTb1b.  However, in these cases \refclaim{cl-inter} and \refclaim{cl-nochords} imply that
either $v_1$ or $v_2$ has list of size three and degree two, which is a contradiction.

Let us now consider the case that there is no $L$-coloring of the path $P'$ assigning the color $c$ to $v_i$.
Since the path $P'$ is induced, this is only possible if $i=1$, or if $i=2$ and $|L(v_1)|=2$. 
If $|L(v_i)|=2$, then $i=1$ and \refclaim{cl-schords} implies that $k=2$ and $G_2$ is OBSTx1b.
However, that is excluded by \refclaim{cl-nosimt}.  Therefore, $|L(v_i)|=3$,
and there exist two $L$-colorings $\psi_1$ and $\psi_2$ of $P'$ such that $\psi_1(v_i)\neq \psi_2(v_i)$.
By the minimality of $G$, we can apply Theorem~\ref{thm-maingen} to $G_1$ with the precolored
path $P'$, and we conclude that $\psi_1$ and $\psi_2$ extend to $L$-colorings $\varphi_1$ and $\varphi_2$ of $G_1$, respectively.
Furthermore,
neither $\varphi_1$ nor $\varphi_2$ extends to an $L$-coloring of $G_2$.  The inspection of the graphs in Figure~\ref{fig-obsta}
shows that this is only possible if $G_2$ is OBSTa1 or OBSTx1c, or if $k=3$ and $G_2$ is OBSTa2 or OBSTx2a.
The case that $G_2$ is OBSTx2a is excluded by \refclaim{cl-nosimt}.
Let us discuss the rest of the cases separately:
\begin{itemize}
\item If $G_2$ is OBSTa1, then there exists a color
$c_1\in L(q_1)\setminus\{\psi_1(q_0)\}$
such that every coloring of $Q$ that assigns $\psi_1(q_0)$ to $q_0$ and $c_1$ to $q_1$ extends to an $L$-coloring of $G_2$.
By \refclaim{cl-schords}, no neighbor of $q_1$ has list of size two.
Let $L'$ be the list assignment such that $L'(v_j)=\{\psi_1(v_j)\}$ for $1\le j\le i$, $L'(q_1)=\{\psi_1(q_0),c_1\}$ and $L'(x)=L(x)$
otherwise.  Note that $G_1$ with the precolored path $p_0p_1p_2v_1\ldots v_i$ is not $L'$-colorable, thus it violates~(Q) or~(OBSTb).  If~(OBSTb) is violated, i.e., $G_1$ contains OBSTb1 or OBSTb2,
then $G$ contains a $(\le\!3)$-chord contradicting the choice of $Q$, thus suppose that~(Q) is false.
Then, \refclaim{cl-schords} implies that $i=2$ and $q_1$ is adjacent to $p_1$.  However,
consider the path $Q'=p_0p_1q_1q_2$ (or $Q'=p_0p_1q_1q_2q_3$ if $k=3$).  Similarly to \refclaim{cl-4chord}, we conclude that $p_0$ and $q_2$ have a common neighbor with list of size two,
and since $q_2$ has degree at least three, this common neighbor is not equal to $q_3$.  However, then $G$ contains OBSTa5.
\item If $G_2$ is OBSTx1c, then by \refclaim{cl-nosimt}, $q_0$ has degree two in $G_2$.  Since neither $\varphi_1$ nor $\varphi_2$ extends to an $L$-coloring of $G_2$,
this implies that $Q$ is a $3$-chord.  Note that there exists an $L$-coloring
$\varphi$ of the path $p_mv_1\ldots v_{i+2}$ such that $\varphi(v_{i+2})\not\in L(q_3)$.  Let $L'$ be the list assignment such that
$L'(v_j)=\{\varphi(v_j)\}$ for $1\le j\le i+1$, $L'(v_{i+2})=\{\varphi(v_{i+1}),\varphi(v_{i+2})\}$ and $L'(x)=L(x)$ otherwise.
The graph $G'=G-v_{i+2}q_3$ with the precolored path $p_0p_1p_2v_1\ldots v_{i+1}$ is not $L'$-colorable,
thus it contains a subgraph $H$ which is an image of one of the relevant configurations from Figure~\ref{fig-obstb}.
By \refclaim{cl-schords}, if $i=2$ then $G'$ does not contain OBSTb1 or OBSTb2,
hence $v_{i+1},v_{i+2}\in V(H)$.  By \refclaim{cl-inter}, we conclude that $v_i$ has degree at least three in $H$, and by the choice of $Q$
so that $G_2$ is maximal, we have $Q\subset H$.
By \refclaim{cl-inter} and \refclaim{cl-schords}, we have $G'=H$.  If $H$ is OBSTb3, then $G$ is OBSTx4.  Otherwise, $G$ contains a subgraph $H'$ depicted in Figure~\ref{fig-redu}.
Observe that every $L$-coloring of $G-V(H')$ extends to an $L$-coloring of $G$, contradicting the minimality of $G$.

\begin{figure}
\center{\includegraphics{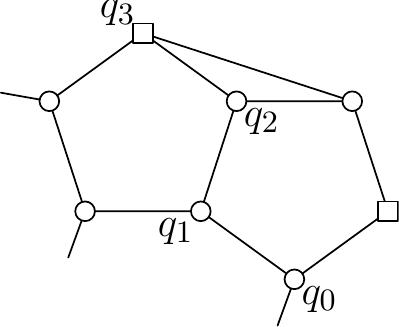}}
\caption{A configuration from claim \refclaim{cl-nonearchords}.}
\label{fig-redu}
\end{figure}

\item If $G_2$ is OBSTa2, then let $w_1$ and $w_2$ be the neighbors of $v_i$ and $v_{i+2}$, respectively, that are incident with the triangle $T$ of the configuration.
Since neither $\varphi_1$ nor $\varphi_2$ extends to an $L$-coloring of $G_2$, we have $L(w_1)=L(w_2)$.
Let $\varphi$ be a coloring of the path $p_mv_1\ldots v_{i+2}$ such that $\varphi(v_i)\neq\varphi(v_{i+2})$.
Let $G'$ be the graph obtained from $G-(V(T)\cup \{v_{i+1}\})$ by adding the edge $v_iv_{i+2}$, and $L'$ the list assignment
such that $L'(v_j)=\{\varphi(v_j)\}$ for $1\le j\le i+2$ and $L'(x)=L(x)$ otherwise.  Note that $G'$ is not $L'$-colorable.
By \refclaim{cl-inter}, no $(\le\!4)$-cycle in $G'$ contains the edge $v_iv_{i+2}$, thus the minimality of $G$ implies that $G'$
with the precolored path $p_0p_1p_2v_1\ldots v_iv_{i+2}$
violates~(Q) or~(OBSTb).  If $G'$ violates~(Q), then $q_3$ is adjacent to $p_0$, and since $q_1$ has degree at least three,
\refclaim{cl-inter} applied to the cycle $p_0p_1\ldots q_0q_1q_2q_3$ shows that $i=2$ and $q_1$ is adjacent to $p_1$.  It follows
that $G$ contains OBSTa4.  Suppose now that $G'$ contains an image $H$ of a relevant configuration from Figure~\ref{fig-obstb}.  Observe that $v_{i+3}$ belongs to $H$;
and, an inspection of the graphs in Figure~\ref{fig-obstb} shows that $v_{i+3}$ has degree two in $H$.  However, since $Q$ is
a $3$-chord, $v_{i+3}=q_3$ has degree at least three in $G$.  We conclude that either a face of $H$ or a path in the outer face
of $H$ gives a cycle or a $(\le 3)$-chord in $G$ contradicting \refclaim{cl-inter} or \refclaim{cl-schords}.
\end{itemize}
\end{itemize}
Therefore, we obtained a contradiction in all the cases, excluding the existence of a $(\le 3)$-chord $Q$ with the specified properties.
\end{proof}

Let $k$ be the index such that $v_k\in X$ and $v_{k+1}\not\in X$.  We now show that $G$ contains one of several subgraphs
near to $P$; see Figure~\ref{fig-AB} for cases~(A4) and~(A5).

\begin{figure}
\center{\includegraphics{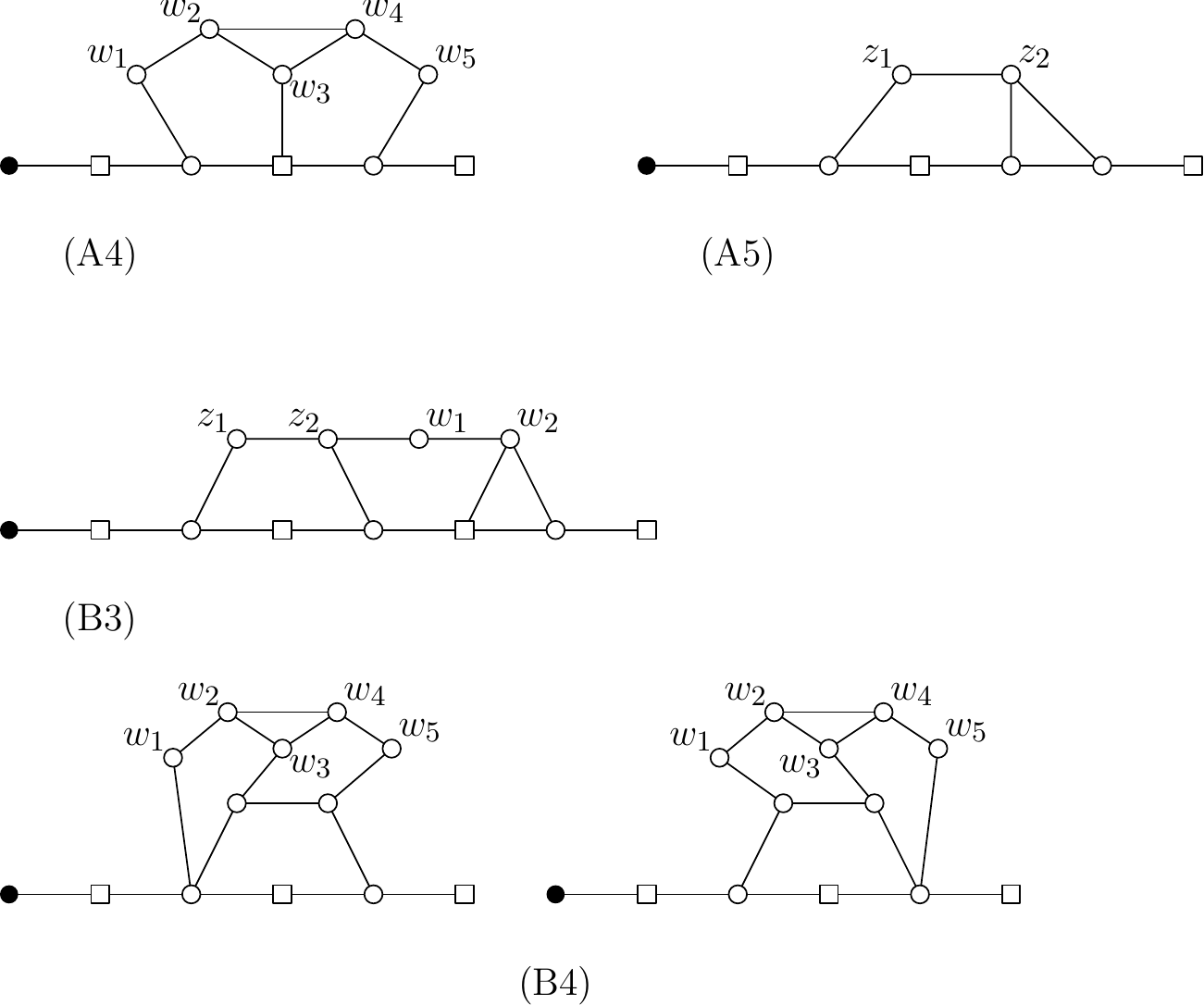}}
\caption{Configurations from claims \refclaim{cl-neart} and \refclaim{cl-neartb}.}
\label{fig-AB}
\end{figure}

\claim{cl-neart}{One of the following holds:
\begin{itemize}
\item[{\rm (A1)}] $|X|=3$ and $v_2v_3v_4$ is a part of the boundary walk of a $5$-face, or
\item[{\rm (A2)}] a vertex of $X$ is incident with a triangle, or
\item[{\rm (A3)}] an edge of the path $p_mv_1v_2\ldots v_k$ is incident with a $4$-face, or
\item[{\rm (A4)}] $|X|=3$ and there exists a path $w_1w_2w_3w_4w_5$ in $G-(X\cup V(P))$ such that
$w_2w_4, v_2w_1,v_3w_3, v_4w_5\in E(G)$, or
\item[{\rm (A5)}] $|L(v_1)|=|L(v_3)|=|L(v_6)|=2$ and there exist adjacent vertices $z_1,z_2\in V(G)\setminus (X\cup V(P))$
such that $z_1v_2, z_2v_4, z_2v_5\in E(G)$.
\end{itemize}}
\begin{proof}
Suppose that $X$ satisfies none of these conditions.
Since no vertex of $X$ is incident with a triangle, \refclaim{cl-nochords} implies that the subgraph $R$ induced by $V(P)\cup \{v_1,\ldots, v_k\}$
is either a path or equal to the cycle $C$.
Observe that there exists a proper $L$-coloring $\psi$ of $R$ such that if $|L(v_{k+1})|=2$, then $\psi(v_k)\in L(v_k)\setminus L(v_{k+1})$.
Let us remark that if $v_1\in X$, then $\psi(v_1)\not\in L(p_m)$, and
if $v_1\not\in X$ and $|L(v_1)|=2$, then $\psi(v_2)$ is different from the unique color in $L(v_1)\setminus L(p_m)$.

Let $G'=G-X$ and let $L'$ be the list assignment obtained from $L$ by removing the colors of vertices of $X$ according to $\psi$ from the lists of
their neighbors, with the following exception: if $v_1\not\in X$ and $|L(v_1)|=2$, then $L'(v_1)=L(v_1)$ (note that still,
an $L'$-coloring of $G'$ corresponds to an $L$-coloring of $G$, since $L(v_1)\setminus L(p_m)$ contains only one color, and this color
is different from $\psi(v_2)$).
By \refclaim{cl-nochords}, no neighbor of a vertex of $X$ other than $v_1$ and $v_{k+1}$ has list of size less than three in $L$; furthermore,
since~(A2) and~(A3) are false, no vertex of $G'$ has two neighbors in $X$.  It follows that
$G'$ satisfies~(S2).  By \refclaim{cl-nochords} and \refclaim{cl-2pchord}, no vertex of $V(G)\setminus V(P)$ has two neighbors in $P$,
thus~(Q) holds.

Let us now show that~(I) holds: otherwise, there would exist adjacent vertices $w_1, w_2\in V(G')$ such that
$|L'(w_1)|=|L'(w_2)|=2$.  We can assume that $|L(w_1)|=3$, and thus $w_1$ has a neighbor $x\in X$.  If $|L(w_2)|=3$, then
$w_2$ has a neighbor in $X$ as well, and by \refclaim{cl-inter}, it follows that~(A1), (A2) or~(A3) holds.  Suppose that $|L(w_2)|=2$.
If $w_1\not\in V(C)$,
then by \refclaim{cl-nonearchords}, the $2$-chord $xw_2w_1$ splits off a face.  This face is a triangle, as otherwise $w_2$ would be adjacent
to a vertex with degree two and list of size three, and thus~(A2) holds.  If $w_1\in V(C)$, then since~(A2) is false, \refclaim{cl-nonearchords} implies that
$w_1\in \{v_1,v_{k+1}\}$.  If $w_1=v_1$, then the chord $w_1w_2$ contradicts \refclaim{cl-nochords}, hence this is only possible if $w_1=v_{k+1}$ and $w_2=v_{k+2}$.
However, the set $X$ was chosen so that if $|L(v_{k+1})|=3$, then $|L(v_{k+2})|=3$.  In all cases, we conclude that (I) is satisfied.

Suppose now that~(T) is violated, that is, there exists a path $w_1w_2w_3w_4w_5$ in $G'$ such that $|L'(w_1)|=|L'(w_3)|=|L'(w_5)|=2$
and $w_2w_4\in E(G)$.  Let us discuss several cases:
\begin{itemize}
\item If $|L(w_3)|=2$, then by the condition~(T) for $G$ and list assignment $L$ and by the symmetry between $w_1$ and $w_5$,
we can assume that $|L(w_1)|=3$, and hence $w_1$ has a neighbor $x\in X$.
\begin{itemize}
\item If $w_1\not\in \{v_1,v_{k+1}\}$, then \refclaim{cl-nonearchords} implies that a subpath of $xw_1w_2w_3$ splits off a face $F$.  Since $|L(w_3)|=2$
and no vertex with list of size three has degree two, we have $\ell(F)\le 4$.  However, $d(F, w_2w_3w_4)<B$, which is a contradiction.
\item If $w_1=v_1$, then by \refclaim{cl-schords}, a subpath of
$w_1w_2w_3$ splits off a triangle or OBSTx1.   However, the triangle $T$ in the split-off part intersects $X$, and thus it is not equal to $w_2w_3w_4$.
Again, we have $d(T,w_2w_3w_4)<B$, which is a contradiction.
\item It remains to consider the case that $w_1=v_{k+1}$.  If $|L(w_5)|=3$, by symmetry we have $w_5=v_{k+1}=w_1$, which is a contradiction.
Therefore, $|L(w_5)|=2$ and by \refclaim{cl-schords}, $w_3w_4w_5$ is a subpath of $C$.  Similarly,
we conclude that that $w_1w_2w_3\subset C$, and thus $w_j=v_{k+j}$ for $1\le j\le 5$.
However, then $k\le 3$, since both $v_{k+1}$ and $v_{k+2}$ have a list of size three, and $|L(v_{k+5})|=2$ and $v_{k+3}$ is a vertex of degree two incident
with a triangle, contradicting \refclaim{cl-nosimt}.
\end{itemize}
\item If $|L(w_3)|=3$, then $w_3$ has a neighbor $y\in X$.  If $|L(w_1)|=|L(w_5)|=3$, then both $w_1$ and $w_5$ have a neighbor in $X$, and thus~(A4) holds.
Therefore, assume that say $|L(w_1)|=2$.  If $w_3\not\in \{v_1,v_{k+1}\}$, then by \refclaim{cl-nonearchords} a subpath of $yw_3w_2w_1$ splits off a
face of length at most four whose distance from $w_2w_3w_4$ is less than $B$, which is a contradiction.
It follows that $w_3\in \{v_1,v_{k+1}\}$.  By \refclaim{cl-schords}, we have $w_1w_2w_3\subset C$, hence $w_j=v_{k+4-j}$ for $1\le j\le 3$.
If $|L(w_5)|=2$, a symmetrical argument would show that $w_5=v_{k+3}=w_1$, thus we have $|L(w_5)|=3$ and $w_5$ has a neighbor in $X$.
By the choice of $X$, it follows that~(A5) holds.
\end{itemize}

Therefore, $G'$ satisfies~(S1), (S2), (S3), (I), (Q) and~(T),
and by the minimality of $G$, we conclude that $G'$ violates~(OBSTa) and~(OBSTb).
Thus $G$ contains a near-obstruction $H$. By \refclaim{cl-notame}, $H$ is not tame, hence there exists a vertex $v\in V(H)\setminus V(C)$ such that $|L'(v)|=2$
and $v$ has a neighbor in $P$.  By \refclaim{cl-nopchord}, $v$ is not adjacent to an endvertex of $P$, hence either $m=2$ and $H$ is OBSTx1, or $m=5$ and $H$ is
OBSTb1 or OBSTb2, with a vertex $p\in \{p_0,p_m\}$ not contained in $H$.  
Let us discuss the cases separately.  Let $v_t$ be the neighbor of $v$ in $X$.  

\begin{itemize}
\item Suppose first that $m=2$.  The outer face of $H$ consists of an edge of $P$ and of a path $q_0q_1q_2q_3$, where $q_0q_2\in E(G)$ and $q_3=v$.

If $p=p_0$, then $H$ is drawn inside the closed disk bounded by $K=p_2p_1vv_tv_{t-1}\ldots v_1$.  Then, \refclaim{cl-inter}
implies that $t\ge 3$.
Note that $q_1\not\in V(C)$, as otherwise we would have $q_1\in X$.  Therefore, $q_1$ has a neighbor $z\in X$.
By~\refclaim{cl-4chord} applied to $p_0q_1q_1z$, we have $z=v_2$ and $|L(v_1)|=2$.
Considering the path $Q=p_0p_1vv_t$, as in \refclaim{cl-4chord} we conclude that $p_0$ and $v_t$ have a common
neighbor with list of size two.  It follows that $|L(v_t)|=3$, and by the choice of $X$, we have $t=4$.  By~\refclaim{cl-inter},
this uniquely determines the graph $G$, which however is $L$-colorable.

Hence, suppose that $p=p_2$.  Since no vertex of degree two has list of size three, \refclaim{cl-inter} implies that $t=2$ and $v_2$ has
list of size three.  Therefore, $v_2$ has degree at least three, and $q_2,q_3\not\in V(C)$ by \refclaim{cl-4chord}.
If $q_1\in V(C)$, then \refclaim{cl-4chord} implies that $q_1=v_4$, $|L(v_4)|=3$ and $q_0=v_5$; the resulting graph is $L$-colorable.

Hence, we can assume that $q_1\not\in V(C)$.  Then, we have $|L(q_1)|=3$ and $q_1$ is adjacent to a vertex $x\in X$.
Note that $x$ and $p_0$ have a common neighbor with list of size two by \refclaim{cl-4chord} applied to $xq_1q_0p_0$.
Hence, $|L(x)|=3$, and by the choice of $X$ we have $x=v_4$.
Again, the resulting graph is $L$-colorable.

\item Let us now consider the case that $m=5$.  By \refclaim{cl-nopchordm} and symmetry (we will no longer use any properties of the set $X$),
we can assume that $p=p_5$ and $v$ is adjacent to $v_2$ and $p_4$.  Let $K$ be the cycle bounding
the outer face of $H$ and $Q=K-(V(P)\cup \{v_1\})=q_0q_1\ldots$, where $q_0$ is adjacent to $p_0$.  By \refclaim{cl-nopchord}, we have $q_0\in V(C)$.
Let $G_1=G-(V(H)\setminus V(Q))$.

If $H$ is OBSTb1, then note that $v_2$ has degree at least three, thus by \refclaim{cl-4chord} $q_0$ and $v_2$ have a common neighbor with list of size two.
However, then $G$ contains OBSTb2b.  Therefore, $H$ is isomorphic to OBSTb2.  There exists an $L$-coloring $\varphi$ of $H$ such that
$\varphi(q_1)\not\in L(q_0)\setminus L(p_0)$.  Let $L_1$ be the list assignment defined by $L_1(x)=\varphi(x)$ for
$x\in V(Q)\setminus \{q_0\}$, $L_1(q_0)=(L(q_0)\setminus L(p_0))\cup\{\varphi(q_1)\}$ and $L_1(x)=L(x)$ otherwise;
$G_1$ cannot be $L_1$-colorable.  Since a path $Q-q_0$ of length $4$ is precolored in $G_1$ and $d_{G_1}(Q-q_0)\ge d_G(P)-3\ge r(P)-3=r(Q-q_0)$,
the minimality of $G$ implies that $G_1$ violates~(S3), (Q) or~(OBSTb).

As $q_2$ cannot be a vertex of degree two with a list of size three,
$G_1$ satisfies (S3), and if $G_1$ violates (Q), then
\refclaim{cl-schords} implies that $G$ consists of $H$ and a vertex with list of size two adjacent to $q_2$ and $v_2$,
and $G$ is $L$-colorable.  Similarly, if $G_1$ violates (OBSTb), then $G_1$ is OBSTb2 and $G$ is $L$-colorable.
\end{itemize}
In all cases, we concluded that $G$ is $L$-colorable, which is a contradiction.
\end{proof}

Let $O=(H,Q,f)$ be one of the relevant configurations from Figure~\ref{fig-obsta} or \ref{fig-obstb}.
A set $U\subseteq V(H)\setminus V(Q)$ {\em has lists determined by the rest of the configuration $O$}
if whenever $L_1$ and $L_2$ are two list assignments to $H$ such that
\begin{itemize}
\item $|L_1(v)|=|L_2(v)|=1$ for $v\in V(Q)$ and $|L_1(v)|=|L_2(v)|=f(v)$ for $v\in V(H)\setminus V(Q)$,
\item $L_1(x)=L_2(x)$ for each $x\not\in U$,
\item vertices with list of size one give a proper coloring of the path induced by them, and
\item $H$ is neither $L_1$-colorable nor $L_2$-colorable,
\end{itemize}
then $L_1=L_2$.  That is, the list assignment that does not extend to $H$ is uniquely determined
once it is known on all the vertices except for those in $U$.
Let $D=\{v\in V(H)\setminus V(Q):f(v)=2\}$.
We call $O$ {\em $k$-determined} if every set $U\subseteq D$ of size at most $k$
has lists determined by the rest of $O$.  A straightforward case
analysis shows the following.

\claim{cl-obstun}{All relevant configurations in Figures~\ref{fig-obsta} and \ref{fig-obstb} are $1$-determined.
All except OBSTa2, OBSTx1c, OBSTx2b, OBSTb1, OBSTb1a, OBSTb3, OBSTb5 and OBSTb6 are $2$-determined.}

Let us now further discuss the subcase~(A1) of \refclaim{cl-neart}; see Figure~\ref{fig-AB} for cases~(B3) and~(B4).

\claim{cl-neartb}{If $|X|=3$ and $v_2v_3v_4z_2z_1$ is a $5$-face, then there exists 
\begin{itemize}
\item[{\rm (B1)}] a triangle incident with $v_2$, $v_4$, $z_1$ or $z_2$, or
\item[{\rm (B2)}] a $4$-face incident with $z_1$ or $z_2$, or
\item[{\rm (B3)}] adjacent vertices $w_1$, $w_2\in G-(X\cup \{z_1,z_2\})$ such that
$w_1z_2,w_2v_5,w_2v_6\in E(G)$, and furthermore, $|L(v_7)|=2$, or
\item[{\rm (B4)}] a path $w_1w_2w_3w_4w_5$ in $G-(X\cup \{z_1,z_2\})$ such that
$w_2w_4\in E(G)$, and either $v_2w_1,z_1w_3, z_2w_5\in E(G)$ or $z_1w_1,z_2w_3, v_4w_5\in E(G)$
(possibly with $w_1=v_1$ in the former case or $w_5=z_5$ in the latter case).
\end{itemize}}
\begin{proof}
Suppose that none of these conditions is satisfied.
Since $v_2$ and $v_4$ have lists of size three, they must have degree at least three in $G$,
and thus \refclaim{cl-nonearchords} implies that $z_1,z_2\not\in V(C)$, unless~(B1) holds.  Let $\varphi$ be
the coloring of $X=\{v_2,v_3,v_4\}$, $G'=G-X$ and $L'$ the list assignment to $G'$ as chosen in the proof of \refclaim{cl-neart}.
Note that $|L'(z_1)|,|L'(z_2)|\ge 2$.

There exist at least two $L'$-colorings $\varphi_1$ and $\varphi_2$ of the path $z_1z_2$ such that
$\varphi_1(z_1)\neq \varphi_2(z_1)$ and $\varphi_1(z_2)\neq \varphi_2(z_2)$.  For $i\in \{1,2\}$, let
$L_i$ be the list assignment obtained from $L'$ by removing the colors of $z_1$ and $z_2$ according to $\varphi_i$
from the lists of their neighbors.

Let $G''$ be the graph obtained from $G'-\{z_1,z_2\}$ by repeatedly removing the vertices whose degree is less than
the size of their list both in $L_1$ and in $L_2$.  Since $G$ is not $L$-colorable,  $G''$ is not $L_i$-colorable, for $i\in \{1,2\}$.
By the minimality of $G$, we can apply Theorem~\ref{thm-maingen} to $G''$ with the precolored path $P$ and list assignment $L_1$ or $L_2$.

Note that \refclaim{cl-nonearchords} implies that $G''$ with either of the assignments satisfies~(S2), and by \refclaim{cl-2pchord},
(Q) holds as well.
Let us argue that~(I) is satisfied.  Unless~(B1) or~(B2) holds, \refclaim{cl-nonearchords} implies that no neighbor of $z_1$ and $z_2$
other than $v_2$ and $v_4$ lies in $C$, and furthermore, there exists no path $wxy$, where
$w\in \{z_1,z_2\}$, $x\not\in \{v_2,v_4,z_1,z_2\}$ and $|L(y)|=2$.  Thus, (I) holds unless there exists a path $wxyv$ with
$w\in \{z_1,z_2\}$, $v\in \{v_2,v_4,z_1,z_2\}$ and $x,y\in V(G)\setminus (V(C)\cup \{z_1,z_2\})$.
Since~(B1) and~(B2) are false, we have $w=z_1$ and $v=v_4$ or $w=z_2$ and
$v=v_2$.  However, then \refclaim{cl-inter} implies that $z_1$ or $z_2$ has degree two, which is a contradiction.

Let us now consider the condition~(T) for $G''$.  Suppose that there exists a path $w_1w_2w_3w_4w_5$ with
$w_2w_4\in E(G)$ and $|L_i(w_1)|=|L_i(w_3)|=|L_i(w_5)|=2$ for some $i\in\{1,2\}$.  We distinguish two cases depending
on the size of the list of $w_3$.
\begin{itemize}
\item If $|L(w_3)|=2$, then by~(T) and symmetry, we
can assume that $|L(w_1)|=3$, and thus $w_1\not\in\{v_1,v_5\}$ and by \refclaim{cl-nonearchords}, $w_1\not\in V(C)$.
Consider the shortest $(\le\!5)$-chord $Q$ contained in $X\cup \{z_1,z_2,w_1,w_2,w_3,w_4\}$ such that the subgraph $F$ of $G$
that is split off by $Q$ does not contain the triangle $w_2w_3w_4$.  We have $d_F(Q)\ge B-3\ge r(Q)$, since the triangle $w_2w_3w_4$ intersects $Q$.
Note that by the minimality of $G$, there exists an $L$-coloring $\psi$ of $Q$ that does not extend to an $L$-coloring of $F$.
Let $L'$ be the list assignment such that $L'(v)=\{\psi(v)\}$ for $v\in V(Q)$ and $L'(v)=L(v)$ otherwise.
We conclude that $F$ with the list assignment $L'$ and the precolored path $Q$ violates~(S3), (Q) or~(OBSTb).

If $F$ violates~(OBSTb), then by \refclaim{cl-inter} and \refclaim{cl-schords}, $F$ is an image of one of the relevant configurations in Figure~\ref{fig-obstb}.
Since $|L(w_3)|=2$, this is only possible if $\ell(Q)=5$.  In this case, we have $w_5\in V(F)\setminus V(Q)$.  However, note that $v_5$ has degree two in $F$ and thus
it has degree one in $G-X$.  It follows that $v_5\not\in V(G'')$.  However, when we remove $v_5$ during the construction of $G''$, we decrease the degree
of $v_6$ to two; hence also $v_6\not\in V(G'')$.  Repeating this observation, we conclude that $(V(F)\setminus V(Q))\cap V(G'')=\emptyset$.
This implies that $w_5\not\in V(G'')$, which is a contradiction.

If $F$ violates~(S3) or~(Q), then \refclaim{cl-inter} and \refclaim{cl-schords} imply that $Q$ splits off a face.
In particular, we have $v_4\in V(Q)$.  If~(S3) fails, then since $|L(v_5)|=2$, we have that $v_5=w_3$.  Since $z_2$ has degree at least three, we conclude that
$w_1$ is adjacent to $z_2$ by~\refclaim{cl-inter}.  Since $w_1$ has degree at least three, \refclaim{cl-inter} implies
that $w_5$ is not adjacent to $v_2$, $z_1$ or $z_2$; therefore, $|L(w_5)|=2$, and by \refclaim{cl-schords} we have $w_5=v_7$ and $G$ satisfies~(B3).

If~(Q) is violated, then note that $v_5$ has degree one in $G-X$, hence $v_5\not\in V(G'')$ and consequently, $v_5\neq w_5$.  By~\refclaim{cl-inter} and \refclaim{cl-schords},
this excludes the case that $w_4$ is the endvertex of $Q$ adjacent to $v_5$.  Since $|L(v_5)|=|L(w_3)|=2$, $v_5$ is not adjacent to $w_3$.  It follows that $v_5$ is adjacent to $w_2$.
By~(T) for $G$, we have $|L(w_5)|=3$.  By symmetry of the path $w_1w_2w_3w_4w_5$, we conclude that $v_5$ is also adjacent to $w_4$, which is a contradiction since $v_5\neq w_3$.

\item Suppose now that $|L(w_3)|=3$ and $w_3$ has a neighbor in $X\cup \{z_1,z_2\}$.  If $|L(w_i)|=3$ or $w_i\in\{v_1,v_5\}$ holds for each $i\in\{1,5\}$, then since both
$z_1$ and $z_2$ have degree at least three, \refclaim{cl-inter} implies that~(B4) holds.  Therefore, by symmetry we can assume that $|L(w_1)|=2$ and $w_1\not\in \{v_1,v_5\}$.
Again, we consider the shortest $(\le\!5)$-chord $Q$ contained in $X\cup \{z_1,z_2,w_1,w_2,w_3,w_4\}$ such that the subgraph $F$ of $G$
does not contain the triangle $w_2w_3w_4$.  As in the previous paragraph, we conclude that $F$ is a face and violates~(S3) or~(Q).
If $|L(w_5)|=2$, then by symmetry we can assume that $w_5\in V(F)$, and thus $w_5=v_5$.  However, in that case $w_5\not\in V(G'')$, which is a contradiction.
Therefore, $|L(w_5)|=3$ and $w_5\not\in V(F)$.  Since $z_1$ has degree at least three, $w_5$ is adjacent to $z_1$ by \refclaim{cl-inter}.
However, $v_5$ is adjacent to $w_2$, and the path $v_5w_2w_3w_4w_5$ satisfies~(B4).
\end{itemize}
Therefore, we can assume that $G''$ satisfies~(T).
Let us now show that $G''$ is $L_1$-colorable or $L_2$-colorable, thus obtaining an $L$-coloring of $G$ and a contradiction.

We first consider the case that neither $z_1$ nor $z_2$ has a neighbor in $P$.  Then both $L_1$ and $L_2$ satisfy~(S3).
We conclude that $G''$ with either of these assignments violates~(OBSTa) and~(OBSTb).  Thus, $G$ contains a (unique) near-obstruction $H$.
Note that the sizes of the lists of the vertices of $H$ are determined by the corresponding relevant configuration, i.e., $|L_1(v)|=|L_2(v)|$ for every $v\in V(H)$.
The case that $|L_i(v)|=|L'(v)|$ holds for every $v\in V(G)$ is excluded similarly to \refclaim{cl-neart}.
Therefore, assume that there exists $u_1\in V(H)$ with $|L'(u_1)|=3$ and $|L_i(u_1)|=2$.
Let $K$ be the outer face of $H$, and let $q_0q_1\ldots q_t=K-V(P)$, where $q_0$ is the neighbor of $p_0$ (or of
$p_1$, if $H$ is OBSTb1, OBSTb2 or OBSTx1 and $p_0\not\in V(H)$).

The vertex $u_1$ cannot be adjacent to both $z_1$ and $z_2$, thus $L_1(u_1)\neq L_2(u_1)$.
Since $H$ is neither $L_1$-colorable nor $L_2$-colorable and $H$ is $1$-determined by \refclaim{cl-obstun},
it follows that $H$ contains another vertex $u_2$ such that $|L'(u_2)|=3$ and $|L_i(u_2)|=2$.
Suppose that $u_1$ and $u_2$ are both adjacent to $z_1$ or both adjacent to $z_2$.  Since~(B1) and~(B2) are false,
the distance between $u_1$ and $u_2$ must be at least three.  Furthermore, we can assume that no other
vertex between $u_1$ and $u_2$ in $K-V(P)$ has list of size two.  This is only possible if $H$ is OBSTa1, OBSTa5, OBSTx2a,
or OBSTx3. Note that $H$ is not OBSTa1, OBSTa5 or OBSTx3, since OBSTa1 is $2$-determined and OBSTa5 and OBSTx3 are $4$-determined.
Therefore, either $H$ is OBSTx2a or we can assume that $u_1$ is adjacent to $z_1$, $u_2$ is adjacent to $z_2$, and that
$L_i(x)=L'(x)$ for $i\in\{1,2\}$ and $x\in V(H)\setminus \{u_1,u_2\}$.  In the latter case, we conclude that $H$ is not $2$-determined.
By \refclaim{cl-obstun}, $H$ is one of OBSTa2, OBSTx1c, OBSTx2b, OBSTb1, OBSTb1a, OBSTb3, OBSTb5 or OBSTb6.

Let us make one more useful observation: suppose that $\ell(P)=2$, $q_0$ is adjacent to $p_0$ and $|L_1(q_0)|=2$.
If $|L'(q_0)|=3$, then consider the subgraph $G_2$ of $G$ that is split off by the path $Q=p_0q_0zv$,
where $z\in\{z_1,z_2\}$ and $v\in \{v_2,v_4\}$.  Since $H$ contains a triangle whose distance to $Q$ is at most $3$,
\refclaim{cl-4chord} implies that $v$ and $p_0$ have a common neighbor with list of size two.
We conclude that $v=v_4$ and that $v_5$ is adjacent to $p_0$.  However, then $\ell(C)\le 8$,
contradicting \refclaim{cl-length}.
Therefore, $|L'(q_0)|=2$, and since (B1) is false, by \refclaim{cl-nonearchords} we have $|L(q_0)|=2$.
Since $u_1$ and $u_2$ exist, in this situation $H$ has at least three vertices with list of size two.
This implies that $H$ is neither OBSTa2 not OBSTx1c.  It also implies that $H$ is not OBSTx2a, since OBSTx2a is $2$-determined.

Let us consider the remaining obstructions (OBSTx2b, OBSTb1, OBSTb1a, OBSTb3, OBSTb5 and OBSTb6) separately:
\begin{itemize}
\item {\em $H$ is OBSTx2b:} If $p_0$ has degree two in $H$, then by the observation, we have $|L(q_0)|=2$, and thus $H$ is a tame near-obstruction.
This contradicts \refclaim{cl-notame}, and thus $p_0$ has degree three in $H$.  Furthermore, \refclaim{cl-notame} implies that
$q_5\not\in V(C)$, and thus $q_5$ is adjacent to $z_1$ and $q_3$ is adjacent to $z_2$.  If $|L(q_1)|=2$,
then by \refclaim{cl-4chord} applied to (a subpath of) $v_4z_2q_3q_2q_1$, $v_5$ is adjacent to $q_2$ (possibly $v_5=z_1$).  However, by \refclaim{cl-inter}
and \refclaim{cl-schords}, $G$ does not contain any other vertices, and such a graph is $L$-colorable.
Thus, $|L(q_1)|=3$ and $q_1$ is adjacent to $v_4$.  By \refclaim{cl-4chord} for $p_0q_0q_1v_4$, we
conclude that $v_5$ is adjacent to $p_0$, contradicting \refclaim{cl-length}.
\item {\em $H$ is OBSTb1 or OBSTb1a:} If $p_0\in V(H)$, then by \refclaim{cl-4chord} for the path $v_4z_2u_2p_0$, we have that $v_5$ is adjacent to $p_0$.
However, then $G$ contains no other vertices and is $L$-colorable.  Thus, $p_0\not\in V(H)$ and $H$ is OBSTb1.
In this case, we similarly conclude that the path $p_0p_1u_2z_2v_4$ splits off a face, OBSTb1 or OBSTb2.
In all these cases, $G$ is $L$-colorable.
\item {\em $H$ is OBSTb3:} This is excluded by \refclaim{cl-2pchord}.
\item {\em $H$ is OBSTb5:} Suppose that $u_2=q_0$.  Then $u_1=q_2$ and $q_4=v_1$, and by \refclaim{cl-4chord} applied to
$v_4z_2q_0p_0$, we conclude that $v_5$ is adjacent to $p_0$.  However, such a graph is $L$-colorable.
So, $u_2=q_2$ and $u_1=q_4$.  If $|L(q_0)|=3$, then $q_0$ would be adjacent to $v_4$, contradicting \refclaim{cl-nonearchords}.
Thus, $|L(q_0)|=2$. Consider the path $q_0q_1q_2z_2v_4$.  By \refclaim{cl-4chord}, $v_5$ is adjacent to $q_1$
(possibly $v_5=q_0$).  However, then $G$ is $L$-colorable.
\item {\em $H$ is OBSTb6:} Let us note that only one two-element subset of vertices of $H$ with list of size two
does not have lists determined by the rest of the configuration---the one consisting of the two rightmost square vertices in the depiction of OBSTb6 in Figure~\ref{fig-obstb}.
So, we can assume that $p_3$ has degree $4$ in $H$, $u_2=q_4$ and $u_1=q_6$, and $|L(q_0)|=2$.
If $v_4$ is adjacent to $q_2$, then by \refclaim{cl-4chord} for (a subpath of) $q_0q_1q_2v_4$,
we conclude that $v_5=q_2$ and $|L(q_2)|=2$.  If $v_4$ is not adjacent to $q_2$, then $|L(q_2)|=2$ as well.
By \refclaim{cl-4chord} applied to $q_2q_3q_4z_2v_4$, we have that $v_5$ is adjacent to $q_3$.
And again, we conclude that $G$ is $L$-colorable.
\end{itemize}

Let us now consider the case that $z_1$ or $z_2$ is adjacent to a vertex of $P$. By \refclaim{cl-nonearchords}, this vertex
must be an internal vertex of $P$.
\begin{itemize}
\item If exactly one of $z_1$ and $z_2$ has a neighbor in $P$, then by \refclaim{cl-2pchord} at least one of
$L_1$ and $L_2$, say $L_1$, satisfies~(S3).  It follows that $G''$ with the list assignment $L_1$ must violate~(OBSTa) and~(OBSTb),
and $G$ contains a near-obstruction $H$.  However, since one of $z_1$ and $z_2$ has an internal vertex $p\in P$ as a neighbor,
$p$ is a cut-vertex in $G''$, thus this is only possible if $p\in\{p_1,p_{m-1}\}$ and either
$\ell(P)=2$ and $H$ is OBSTx1, or $\ell(P)=5$ and $H$ is OBSTb1 or OBSTb2.
\begin{itemize}
\item Suppose that there exists a vertex $v\in V(H)$ adjacent to $p$ such that $|L_1(v)|=2$.  By \refclaim{cl-nochords}, $v$ is adjacent to
$v_2$, $v_4$, $z_1$ or $z_2$.  Since $z_1$ or $z_2$ is adjacent to $p$ and neither $z_1$ nor $z_2$ is incident with a $(\le\!4)$-cycle,
\refclaim{cl-inter} implies that $z_1$ or $z_2$ has degree two.  This is a contradiction.
\item It remains to consider the case that no vertex with list of size two is adjacent
to $p$, and thus $\ell(P)=2$.  By \refclaim{cl-notame}, the two vertices of $H$ with list of size two are adjacent to $z_2$ and $v_4$, respectively.
However, then $p_0$ and $v_4$ are joined by a $2$-chord contradicting \refclaim{cl-nonearchords}.
\end{itemize}

\item If both $z_1$ and $z_2$ have a neighbor in $P$, then since neither~(B1) nor~(B2) holds, 
the neighbors of $z_1$ and $z_2$ are internal vertices of $P$ by \refclaim{cl-nonearchords}, and $\ell(P)\ge 4$.
Let $p_i$ be the neighbor of $z_1$ and $p_j$ the neighbor of $z_2$.
\begin{itemize}
\item Suppose first that $i<m-1$ or $j<m-3$. By \refclaim{cl-inter}, $P$ contains two adjacent vertices of degree two that are not contained in any $(\le\!5)$-cycle.
In that case, contract these two vertices into one (and change its color so that
it is consistent with the colors of its neighbors).  The resulting graph is a smaller counterexample to Theorem~\ref{thm-maingen},
which is a contradiction.
\item Therefore, we can assume that $i=m-1$ and $j=m-3$.  Let $Q=p_0p_1\ldots p_{m-3}z_2v_4$, and let $\varphi$
be an $L$-coloring of the subgraph of $G$ induced by $V(P)\cup X\cup \{v_1,z_1,z_2\}$ that exists by the minimality of $G$.
Let $G_3=G-(V(P)\setminus V(Q))-\{v_1,v_2,v_3, z_1\}$.  Let $L_3$ be the list assignment such that $L_3(x)=\varphi(x)$
for $x\in V(Q)$ and $L_3(x)=L(x)$ otherwise.  The graph $G_3$ with the precolored path $Q$ is not $L_3$-colorable, thus it violates~(Q)
or contains OBSTb1 or OBSTb2.  If $G_3$ violates~(Q), then \refclaim{cl-nonearchords} implies that
$v_5$ is adjacent to $p_0$ and $G$ contains OBSTb2 or OBSTb2a.  If $G_3$ contains
OBSTb1, then $G$ contains OBSTb6.  Otherwise, $G$ is $L$-colorable.
\end{itemize}
\end{itemize}
In all cases, we obtained a contradiction.
\end{proof}

Let $T$ be the $4$-cycle at distance at most one from $X$ or the triangle at distance at most two from $X$,
which exists by \refclaim{cl-neart} and \refclaim{cl-neartb}.  Since $d(P,T)\le4$, we have $\ell(P)=2$.

\claim{cl-A3false}{(A3) is false.}
\begin{proof}
Suppose that~(A3) holds, i.e., $T$ is a $4$-cycle sharing an edge with the path $p_2v_1\ldots v_k$.
Let $v_iv_{i+1}$ be such an edge with $i$ minimal and let $\varphi$ be an $L$-coloring of the path $p_2v_1\ldots v_i$.
Let $G'$ be the graph obtained from $G-v_iv_{i+1}$ by adding a vertex $v$ adjacent to $v_i$ and $v_{i+1}$.
Let $c$ be a color that does not appear in the lists of $v_i$ and $v_{i+1}$.
Let $L'$ be a list assignment such that $L'(x)=\{\varphi(x)\}$ for $x\in \{v_1,\ldots, v_i\}$,
$L'(v)=\{c\}$ if $|L(v_{i+1})|=2$ and $L'(v)=\{\varphi(v_i),c\}$ if $|L(v_{i+1})|=3$,
$L'(v_{i+1})=(L(v_{i+1})\setminus \{\varphi(v_i)\})\cup \{c\}$ and $L'(x)=L(x)$ for other
vertices $x\in V(G')$.  Note that $G'$ is not $L'$-colorable.  Furthermore, by the choice of $X$, if
$k=4$ then $|L(v_k)|=3$, hence a path $R$ of length at most $5$ is precolored in $G'$ with the list assignment $L'$.  Furthermore,
since $T$ contains the edge $v_iv_{i+1}$, we have $d_{G'}(R)\ge B-5\ge r(R)$.

By \refclaim{cl-nochords} and
\refclaim{cl-nonearchords}, $R$ is an induced path and no vertex with list of size two other than
$v_s$, $v_{i+1}$ and $v$ is adjacent to it, and since $\ell(C)\ge 9$, it follows that~(S3) and~(Q) are satisfied.  Since $T$ is a $4$-cycle,
$v$ cannot be at distance at most one from a triangle in $G'$, thus~(T) holds as well.  By the minimality of $G$,
we conclude that $G'$ violates~(OBSTb); let $H$ be a minimal non-$L'$-colorable subgraph of $G'$.  We have
$\ell(R)\ge 4$, and consequently, $i\ge 1$.  If $i=1$, then we also have $|L'(v)|=1$, and thus $|L(v_2)|=2$ and
$|L(v_1)|=3$; let $w=v_1$.  If $i\ge 2$, then choose $w\in \{v_1,v_2\}$ such that $|L(w)|=3$. 
Such a vertex $w$ has degree at least three in $G$, and thus it has degree at least three in $H$ (even if $w$
is an endvertex of the precolored path of $H$, since then $w$ has a neighbor $x$ with list of size two in $H$, and the
edge $wx$ belongs to $C$ by \refclaim{cl-nonearchords}).

Let us note that there exist at least two ways
how to choose $\varphi$ that differ in the color of $w$ (and possibly of other vertices).
Since neither of them extends to $G'$, inspection of the relevant configurations
in Figure~\ref{fig-obstb} shows that $H$ is OBSTb1, OBSTb1a, OBSTb1b, OBSTb3 or OBSTb5. Since the edge $v_{i-1}v_i$
is not incident with $T$, the vertex $v_i$ has degree at least three in $G$, and thus also in $H$; therefore,
$H$ is OBSTb3 and $|L'(v)|=1$.  However, \refclaim{cl-inter} and \refclaim{cl-schords} imply that $V(G)=V(H)\setminus\{v\}$,
contradicting \refclaim{cl-length}.
\end{proof}

\claim{cl-B2false}{$T$ is a triangle.}
\begin{proof}
Suppose that $T$ is a $4$-cycle.  By~\refclaim{cl-neart}, \refclaim{cl-neartb} and \refclaim{cl-A3false},
it follows that (B2) holds, i.e., $T$ contains $z_1$ or $z_2$.  
If $v_4\in V(T)$, then let $Y=\{v_3,v_4\}$.  If $v_4\not\in V(T)$ and $z_2\in V(T)$,
then let $Y=\{v_3,v_4,z_2\}$; otherwise let $Y=\{v_3,v_4,z_2,z_1\}$.

If both $z_1$ and $z_2$ had a neighbor in $P$, then observe that $z_2$ is adjacent to $p_0$
and the $2$-chord $v_4z_2p_0$ contradicts \refclaim{cl-nonearchords}.  Therefore, at most one of $z_1$ and $z_2$ has a neighbor in $P$.
It follows that there exists an $L$-coloring $\psi$ of the subgraph $G_0$ of $G$ induced by $Y\cup V(P) \cup \{z_1, v_1,v_2\}$ such
that $\psi(v_4)\not \in L(v_5)$.  Let $G'=G-Y$ and let $L'$ be the list assignment such that
$L'(x)=\{\psi(x)\}$ for $x\in \{v_1,v_2\}$, $L'(x)=L(x)\setminus \{\psi(y)\}$
if $x\in V(G')\setminus \{v_1,v_2\}$ has a neighbor $y\in Y$, and $L'(x)=L(x)$ otherwise.  The graph $G'$ with the precolored path
$p_0p_1p_2v_1v_2$ is not
$L'$-colorable. 

Note that $Y$ was chosen so that it intersects $Y$ in exactly one vertex.
Since $z_2$ has degree at least three, \refclaim{cl-inter} and \refclaim{cl-nonearchords}
imply that $G'$ satisfies~(I) and~(S2).  Obviously, (T) is satisfied as well.  Suppose that
a vertex $v$ with $|L'(v)|=2$ has two neighbors in $p_0p_1p_2v_1v_2$.  By \refclaim{cl-nochords}, we have $|L(v)|=3$,
hence $v$ is adjacent to a vertex in $Y$.  Suppose that $v\neq z_1$.  Since~(A3) is false, $v$ is not adjacent to $z_1$;
but then \refclaim{cl-inter} implies that $z_1$ has degree two, which is a contradiction.  Therefore, $v=z_1$, and
since $\psi$ assigns a color to $z_1$, $G'$ satisfies~(Q).  

We conclude that $G'$ violates~(OBSTb); let $H$ be the image of OBSTb1 or OBSTb2 in $G'$.
Note that $v_2$ is adjacent to a vertex $x$ such that $|L'(x)|=2$.  Since $z_1$ has
degree at least three, \refclaim{cl-inter} implies that $x=z_1$, and thus $Y=\{v_3,v_4,z_2\}$.  Furthermore,
note that neither $z_1$ nor $z_2$ has a neighbor in $P$, thus there exist at least two choices for the
coloring $\psi$ that differ exactly in the colors of $z_1$ and $z_2$.  Since neither of these colorings
extends to $G'$ and both OBSTb1 and OBSTb2 are $1$-determined, $z_2$ has a neighbor in $H$ different from $z_1$.
Furthermore, $H$ is not OBSTb2, since OBSTb2 is $2$-determined and $z_2$ cannot have more than two neighbors in $H$
whose lists according to $L'$ have size two.  However, if $H$ is OBSTb1, then $p_0$ and $v_4$ are joined by a $3$-chord,
and by \refclaim{cl-4chord}, $v_5$ is a common neighbor of $p_0$ and $v_4$.  This contradicts \refclaim{cl-length}.
\end{proof}

We now similarly exclude some of the cases in that $T$ is a triangle.
\claim{cl-B4false}{(B4) is false.}
\begin{proof}
Suppose that (B4) holds.  We consider the two subcases of (B4) (whether $w_3$ is adjacent to $z_1$ or $z_2$)
separately.
\begin{itemize}
\item {\em Suppose first that $v_2w_1,z_1w_3, z_2w_5\in E(G)$.}  Note that $v_1$ may be equal to $w_1$.
Let $S=L(v_2)\setminus (L(v_1)\setminus L(p_2))$.
If $S\not\subseteq L(z_1)$, then let $L'$ be a list assignment such that $L'(v_1)=L(v_1)\setminus L(p_2)$,
$L'(v_2)\subseteq S\setminus L(z_1)$ has size one and $L'(x)=L(x)$ otherwise.  Observe that the graph $G-\{z_1,w_3\}$
with the precolored path $p_0p_1p_2v_1v_2$ is not $L'$-colorable
and that it satisfies the assumptions of Theorem~\ref{thm-maingen} (it satisfies~(OBSTb), since $v_3$ is the only neighbor of
$v_2$ with list of size two and $v_1v_2v_3$ cannot be a subpath of a $5$-cycle). This contradicts the minimality of $G$;
hence, assume $S\subseteq L(z_1)$.

If $S\neq L(v_3)$, then choose a color $c\in S\setminus L(v_3)$; let
$L'$ be the list assignment obtained from $L$ by removing $c$ from the lists of neighbors of $v_2$ other than $v_1$.
Note that $G-v_2$ is not $L'$-colorable, and as in \refclaim{cl-neart}, we conclude that $G-v_2$ is a smaller
counterexample to Theorem~\ref{thm-maingen}, which is a contradiction.  Similarly, we exclude the case
that a color $c'\in L(v_4)\setminus L(v_5)$ does not belong either to $S$ or to $L(z_2)$.  Therefore,
there exists a color $c'\in S\cap L(z_2)$.  By \refclaim{cl-inter} and \refclaim{cl-nonearchords},
$z_2$ is not adjacent to a vertex of $P$.

Suppose that $w_1$ and $w_5$ do not have a common neighbor.  Let $G'$ be the graph obtained from $G-\{w_3,z_1,v_3\}$
by identifying $v_2$ with $z_2$ to a new vertex $v$.  Let $L'$ be the list assignment such
that $L'(v_1)=L(v_1)\setminus L(p_2)$, $L'(v)=\{c'\}$, $L'(v_4)=\{c''\}$ for a color $c''\in L(v_4)\setminus \{c'\}$
such that $L(v_3)\neq \{c',c''\}$ and $L'(x)=L(x)$ otherwise.  Note that $t(G')\ge B$, since
both $v_2$ and $z_2$ are at distance at least $B-2$ in $G$ from any $(\le\!4)$-cycle different from $T$.
Since $w_1$ and $w_5$ do not have a common neighbor, \refclaim{cl-inter} implies that $v$ is not contained
in any $(\le\!4)$-cycle in $G'$.  Since $G'$ with the precolored path $p_0p_1p_2v_1vv_4$
is not $L'$-colorable, we conclude that it violates~(OBSTb).
Let $H$ be the image of a relevant configuration from Figure~\ref{fig-obstb} contained in $G'$.
By \refclaim{cl-nonearchords}, $v$ is not adjacent to a vertex with list of size two, hence
$v_4$ belongs to $H$.  Note that $v$ has degree at least three in $H$, as otherwise
$G$ contains a cycle $K$ of length at most $7$ such that $v_1v_2v_3v_4\subset K$ and the open disk bounded by $K$
contains $z_1$ and $w_3$, contradicting \refclaim{cl-inter}.  The inspection of the graphs in Figure~\ref{fig-obstb} shows that
$v$ has degree exactly three and that both internal faces incident with $v$ in $H$ have length five.
Similarly, \refclaim{cl-inter} implies that $vw_5\in E(H)$ and $w_1=v_1$.  But then $v_1vw_5w_4w_2$ is the only $5$-cycle in $G'$
containing the edge $v_1v$, thus $v_1w_2\in E(H)$ and $v_1$ has degree at least three in $H$.  This is only possible if $H$ is OBSTb4.
However, then $H$ is the graph in Figure~\ref{fig-B4}(a), which is $L$-colorable.

\begin{figure}
\center{\includegraphics{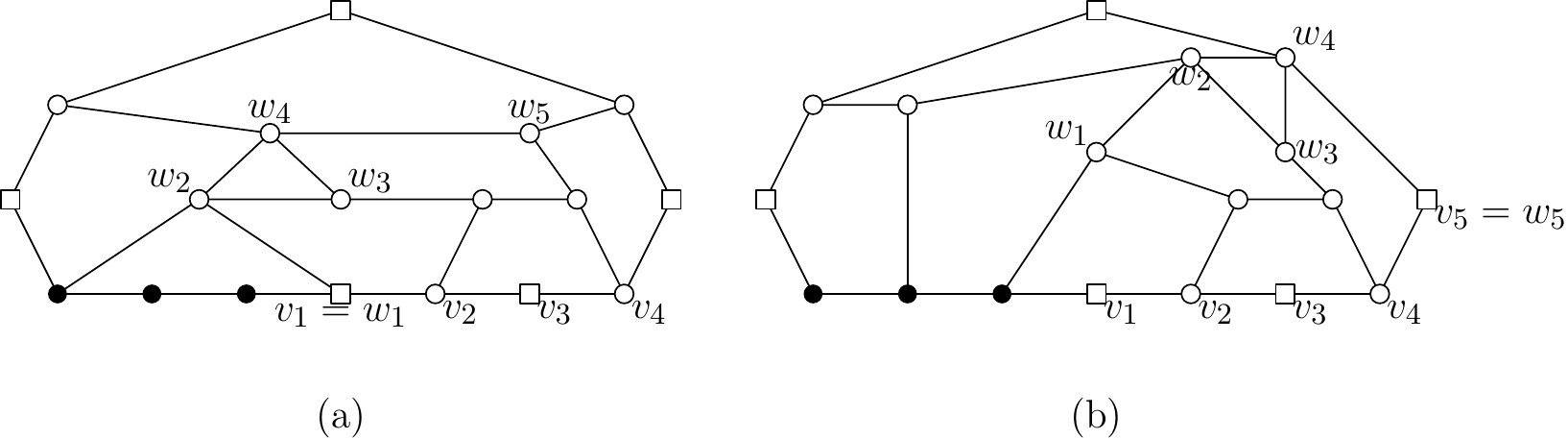}}
\caption{Configurations in the case that~(B4) holds.}
\label{fig-B4}
\end{figure}

So, suppose that $w_1$ and $w_5$ have a common neighbor $w$, and thus by \refclaim{cl-inter}, $w_2$ and $w_4$
have degree three.  By \refclaim{cl-nonearchords}, $|L(w)|=3$.  Let $\psi$ be an $L$-coloring of $p_2v_1v_2v_3v_4z_2$ such that
$\psi(v_4)=c'$.  Let $d$ be a color in $L(z_1)\setminus\{\psi(v_2),\psi(z_2)\}$.
If $w_1\neq v_1$, then let $d'$ be a color in
$L(w_1)\setminus\{\psi(v_2)\}$ such that $L(w_2)\setminus \{d'\}=L(w_3)\setminus \{d\}$,
if such a color exists, and an arbitrary color in $L(w_1)$ otherwise.
Note that $z_2$ has no neighbor in $P$ by \refclaim{cl-inter}; this enables us to choose $\psi$, $d$ and $d'$
so that the following additional conditions (which apply to mutually exclusive situations) hold:
\begin{itemize}
\item If $w_1$ is adjacent to $p_1$, then $L(w_1)\neq L(p_1)\cup \{\psi(v_2),d'\}$.
\item If $w_1=v_1$, then either $\psi(v_1)\not\in L(w_2)$ or $L(w_2)\setminus \{\psi(v_1)\}\neq L(w_3)\setminus\{d\}$.
\item If $w_1\neq v_1$, $w_1$ is not adjacent to $p_1$ and $p_1$ has exactly one neighbor $z\not\in V(C)$, then
$L(z)\setminus L(p_1)\neq L(w_5)\setminus \{\psi(z_2)\}$.
\end{itemize}
Let $G'=G-\{w_2,w_3,w_4,z_1, z_2,v_3,v_4\}$, with the list assignment $L'$ such that $L'(x)=\{\psi(x)\}$ for $x\in\{v_1,v_2\}$,
$L'(w_1)=L(w_1)\setminus \{d'\}$ if $w_1\neq v_1$, $L'(x)=L(x)\setminus \{\psi(y)\}$ for every vertex $x$ with a neighbor $y\in \{v_4,z_2\}$
and $L'(x)=L(x)$ otherwise.  The choice of $\psi$ and $d'$ ensures that every $L'$-coloring of $G'$ would extend to an $L$-coloring of $G$,
thus the graph $G'$ with the precolored path $p_0p_1p_2v_1v_2$ is not $L'$-colorable.

If $w_1$ had a common neighbor with $v_4$ or $z_2$,
then \refclaim{cl-inter} would imply that $w$ has degree two; hence \refclaim{cl-nonearchords} implies that $G'$ satisfies~(I).
If $G'$ violated~(Q), then \refclaim{cl-inter} and \refclaim{cl-nonearchords} would imply that $w_1$ is adjacent to $p_1$.  But, in that case
the choice of $\psi$, $d$ and $d'$ ensures that~(Q) holds.  Hence, $G'$ violates~(OBSTb); let $H$ be the image of OBSTb1 or OBSTb2 in $G'$.
Then $v_2$ is adjacent to a vertex with list of size two, and by \refclaim{cl-inter}, this vertex is $w_1$; hence, we have $w_1\neq v_1$.  Note that there exists a path $w_1xy$ in $H$
such that $y$ has list of size two.  By \refclaim{cl-nonearchords}, we have $|L(y)|=3$, hence $y$ is adjacent to $z_2$ or $v_4$.  Since $w$
has degree at least three, \refclaim{cl-inter} implies $x=w$ and $y=w_5$.  If $H$ were OBSTb1, then $w_5$ would be adjacent to $p_0$, and by \refclaim{cl-4chord} applied to
$v_4z_2w_5p_0$, we would have that $v_5$ is adjacent to $p_0$, contradicting \refclaim{cl-length}.  It follows that $H$ is OBSTb2.  Note that $w_1$ is not adjacent to $p_1$,
thus by the choice of $\psi$, the unique neighbor $z$ of $p_1$ in $V(H)\setminus V(C)$ satisfies $L'(z)\setminus L(p_1)\neq L'(w_5)$.  However, then $H$ is $L'$-colorable, contradicting the hypothesis
that~(OBSTb) does not hold.

\item {\em Next, consider the case that $z_1w_1,z_2w_3, v_4w_5\in E(G)$.} Note that $w_5$ may be equal to $v_5$.
Similarly to the previous case, we conclude that $L(v_2)\setminus (L(v_1)\setminus L(p_2))=L(v_3)\subseteq L(v_4)$,
that each color $c'\in L(v_4)\setminus L(v_5)$ belongs to both $L(v_3)$ and $L(z_2)$
and that $L(z_1)=L(z_2)$---otherwise, we can color a subset $Y$ of $X\cup \{z_2\}$, remove
the colors of the vertices of $Y$ from the lists of their neighbors, and obtain a smaller counterexample to Theorem~\ref{thm-maingen}.

If $L(z_2)\neq L(v_4)$, then let $\psi$ be an $L$-coloring of $p_2v_1v_2v_3v_4$ such that $\psi(v_4)\not\in L(z_2)$.
Let $G'$ be the graph obtained from $G-\{v_3,z_2,w_3\}$ by adding the edge $v_2v_4$.
Let $c$ be a color that does not appear in any of the lists and $L'$ the list assignment such that
$L'(x)=\{\psi(x)\}$ for $x\in \{v_1,v_2\}$, $L'(v_4)=\{c\}$, $L'(x)=(L(x)\setminus \{\psi(v_4)\})\cup \{c\}$ for
all other vertices $x$ adjacent to $v_4$, and $L'(x)=L(x)$ otherwise.  Note that $G'$ with the precolored path $R=p_0p_1p_2v_1v_2v_4$
is not $L'$-colorable.  Also, by \refclaim{cl-inter}, the edge $v_2v_4$ is not incident with a $(\le\!4)$-cycle, and thus $t(G')\ge B$.
Furthermore, the distance from $v_2$ to $T$ in $G$ is three, thus $d_{G'}(R)\ge B-7\ge r(5)$.
By the minimality of $G$, we conclude that $G'$ violates (OBSTb).
Since $v_2$ is not incident with a vertex with list of size two and every cycle containing the edge $v_2v_4$ has length
at least six, $G'$ contains an image $H$ of OBSTb1a or OBSTb2a.  Note that $v_2$ has degree two in $H$, and thus
$z_1\not\in V(H)$.  Let $F'$ be the cycle bounding the internal face of $H$ incident with $v_2v_4$,
and let $F$ be the cycle in $F$ obtained from $F'$ by replacing $v_2v_4$ with $v_2v_3v_4$.
Observe that $z_1,z_2,w_3\in \inter_G(F)$, and since $\ell(F)=7$, this contradicts \refclaim{cl-inter}.
Therefore, we can assume that $L(z_2)=L(v_4)$.

Let us now consider the case that $p_1$ is adjacent to $z_1$.  Let $G'=G-\{p_2,v_1,v_2,v_3,v_4,z_2,w_3\}$ and
let $\psi$ be an $L$-coloring of the subgraph of $G$ induced by $\{p_1,p_2,v_1,v_2,v_3,v_4,z_1,z_2,w_1,w_2\}$
such that $\psi(v_4)\not\in L(v_5)$ and $\psi(w_2)\not\in L(w_3)\setminus \{\psi(z_2)\}$.
Let $L'$ be the list assignment such that $L'(x)=\{\psi(x)\}$ for $x\in \{z_1,w_1,w_2\}$,
$L'(x)=L(x)\setminus\{\psi(v_4)\}$ if $x$ is a neighbor of $v_4$ and $L'(x)=L(x)$ otherwise.
Note that $G'$ with the precolored path $p_0p_1z_1w_1w_2$ is not $L'$-colorable.
By \refclaim{cl-inter}, neither $w_1$ nor $w_2$ has a common neighbor with $v_4$ (since if $w_5\neq v_5$, then
$w_5$ has degree at least three).  By \refclaim{cl-nonearchords}, $w_1$ has no neighbor with list of size two in $G'$, and
since $w_1$ has degree at least three, \refclaim{cl-inter} implies that $G'$ satisfies~(Q) and (S3).
By the minimality of $G$, we conclude that $G'$ violates~(OBSTb).
Because $w_1$ has degree at least three, \refclaim{cl-inter} implies that $G'$ contains OBSTb2.
In particular, $w_2$ has a neighbor $y$ with $|L'(y)|=2$, and as we observed before, $|L(y)|=2$.
Consider the path $Q=v_4w_5w_4w_2y$.  If $Q$ is not a subpath of $C$, then $v_4$ and $w_2$ have a common neighbor by \refclaim{cl-4chord},
implying that $w_5$ has degree two, which is a contradiction.  Therefore, $w_5=v_5$ and $Q\subset C$.  However, then there exists an $L$-coloring $\psi'$
of the subgraph of $G$ split off by the $3$-chord $p_1z_1w_1w_2$ that differs from $\psi$ exactly in the colors of $w_1$ and $w_2$,
and at least one of $\psi$ and $\psi'$ extends to an $L$-coloring of $G$.  This is a contradiction.
Therefore, we can assume that $p_1z_1\not\in E(G)$.

Suppose next that $w_1$ and $w_5$ do not have a common neighbor.
Then, let $G'$ be the graph obtained from $G-\{v_3,z_2,w_3\}$ by identifying $z_1$ with $v_4$
to a new vertex $v$, with the list assignment $L'$ such that $L'(v)=L(v_4)\setminus L(v_3)$,
$L'(v_1)=L(v_1)\setminus L(p_2)$, $L'(v_2)\subseteq L(v_2)\setminus (L'(v)\cup L'(v_1))$ has size one
and $L'(x)=L(x)$ otherwise.  Observe that $G'$ satisfies $t(G')\ge B$ and that it is not $L'$-colorable.
The precolored path of $G'$ is $p_0p_1p_2v_1v_2v$.
Since $p_1$ is not adjacent to $z_1$, \refclaim{cl-nonearchords} implies that $G'$ satisfies~(S3).
No vertex with list of size two is adjacent to $p_1$ or $v_2$ and the only vertex with list of size two
adjacent to $v$ is $v_5$, thus $G'$ satisfies~(Q).  We conclude that $G'$ violates~(OBSTb); let $H$ be the image
of the relevant configuration from Figure~\ref{fig-obstb} that is contained in $G'$.  By \refclaim{cl-nonearchords}, $v_2$ has degree
two in $H$.  If $v$ had degree two, then $v_1v_2v_3v_4$ would be a subpath of a cycle $K$ of length at most seven in $G$,
such that the open disk bounded by $K$ contains $z_1$, $z_2$ and $w_3$.  This contradicts \refclaim{cl-inter}, hence
$v$ has degree three in $H$ and $H$ is OBSTb4.  Let $x$ be the common neighbor of $p_2$ and $v$ in $H$.  By \refclaim{cl-nonearchords},
$x$ is adjacent to $z_1$ in $G$.  In $H$, there exists a path $xyzv_5$ (for some vertices $y$ and $z$), and by \refclaim{cl-inter} we have $x=w_1$,
$y=w_2$, $z=w_4$ and $v_5=w_5$.  Then $G$ is the graph depicted in Figure~\ref{fig-B4}(b), which is $L$-colorable.
Therefore, $w_1$ and $w_5$ have a common neighbor $w$. 
By \refclaim{cl-nonearchords},
$|L(w)|=3$, and by \refclaim{cl-inter}, $w_2$ and $w_4$ have degree three.

Suppose that $w_1$ has no neighbor in $P$.
Then, there exists an $L$-coloring $\psi$ of the subgraph $G_0$ of $G$ induced by $V(P)\cup \{v_1,v_2,v_3,v_4,z_1,z_2,w_1\}$ such that 
$\psi(v_4)\not\in L(v_5)$ and either $\psi(w_1)\not\in L(w_2)$ or $L(w_2)\setminus\{\psi(w_1)\}\neq L(w_3)\setminus \{\psi(z_2)\}$.
Let $G'=G-\{v_3,v_4,z_2,w_2,w_3,w_4\}$ with the list assignment $L'$ such that
$L'(x)=\{\psi(x)\}$ for $x\in \{v_1,v_2,z_1\}$, $L'(w_1)=\{\psi(z_1),\psi(w_1)\}$, $L'(x)=L(x)\setminus \{\psi(v_4)\}$
if $x$ is a neighbor of $v_4$ and $L'(x)=L(x)$ otherwise.  Note that $G'$ is not $L'$-colorable and has precolored
path $p_0p_1p_2v_1v_2z_1$.
By \refclaim{cl-inter} and \refclaim{cl-nonearchords}, $G'$ satisfies~(I), and since $p_1$ is not adjacent to $z_1$, $G'$ satisfies~(S3).
Since $w_1$ has no neighbor in $P$ and $v_2$ has no neighbor with list of size two, $G'$ also satisfies~(Q).
We conclude that~(OBSTb) is violated.  Let $H$ be the image of the relevant configuration depicted in Figure~\ref{fig-obstb} that is contained in $G'$.
The inspection of the relevant configurations shows that if $v_2$ has degree three in $H$, then it is incident with
a path $v_2xyz$ with $|L'(z)|=2$, where $z\neq w_1$.  By \refclaim{cl-inter}, $z$ is not a neighbor of $v_4$, hence $|L(z)|=2$.
However, that contradicts \refclaim{cl-nonearchords}.  Therefore, $v_2$ has degree two in $H$.  Similarly, we conclude that
$v_1$ has degree two in $H$, thus $H$ is OBSTb1a, OBSTb1b or OBSTb4.  
Note that there are at least two possible choices for $\psi$ that match on $v_1$, $v_2$, $v_3$ and $v_4$
and differ on $z_1$ (and possibly also on $z_2$ or $w_1$).  Since neither of the choices extends to $G'$,
we conclude that $H$ is OBSTb1a (and both choices for $\psi$ assign the same color to $w_1$).  But then there exists a path
$v_2z_1x'y'p_0$ with $|L'(y')|=2$, where $x'$ is adjacent to $p_2$.  
By \refclaim{cl-schords} applied to $p_2x'y'$, we have $|L(y')|=3$, thus $y'$ is adjacent to $v_4$.
However, then $v_4y'p_0$ is a $2$-chord contradicting \refclaim{cl-nonearchords}.

Therefore, it remains to consider the case that $w_1$ has a neighbor $p_i\in V(P)$.  By \refclaim{cl-inter}, $z_1$ has degree three.
Observe that there exist colors $c_1\in L(w_1)\setminus L(p_i)$ and $c_2\in L(v_2)\setminus (L(v_1)\setminus L(p_2))$
such that $c_1=c_2$ or $c_1\not\in L(z_1)$ or $c_2\not\in L(z_1)$.
Let $G'$ be the graph obtained from $G-\{p_{i+1},\ldots, p_2,v_1, z_1,z_2,w_2,w_3,w_4\}$ by identifying
$w_1$ with $v_2$ to a new vertex $v$.  By \refclaim{cl-inter}, $v$ is not incident with a $(\le\!4)$-cycle, thus
$t(G')\ge B$ and $d(p_0\ldots p_iv)\ge B-4>r(3)$.  Let $c$ be a new color that does not appear in any of the
lists and $L'$ the list assignment such that $L'(v)=\{c\}$, $L'(v_3)=(L(v_3)\setminus \{c_2\})\cup \{c\}$,
$L'(x)=(L(x)\setminus\{c_1\})\cup \{c\}$ if $x$ is a neighbor of $w_1$ and $L'(x)=L(x)$ otherwise.
Observe that $G'$ is a counterexample to Theorem~\ref{thm-maingen} smaller than $G$, which is a contradiction.
\end{itemize}
Therefore, we obtained a contradiction in both subcases, showing that (B4) is false.
\end{proof}

\claim{cl-A4false}{(A4) is false.}
\begin{proof}
Suppose that~(A4) holds.  Note that $w_1\neq v_1$ and $w_5\neq v_5$, since
$v_2$ and $v_4$ have degree at least three.

Let us first suppose that there exists an $L$-coloring $\psi$ of
the subgraph induced by $V(P)\cup \{v_1,v_2,v_3,v_4,w_1,w_2\}$ such that
$\psi(v_4)\not\in L(v_5)$ and $|L(w_3)\setminus \{\psi(v_3),\psi(w_2)\}|\ge 2$.
Then, let $G'=G-\{v_3,v_4,w_3\}$ with the list assignment $L'$ such that
$L'(x)=\{\psi(x)\}$ for $x\in\{v_1,v_2,w_1\}$, $L'(w_2)=\{\psi(w_1),\psi(w_2)\}$,
$L'(x)=L(x)\setminus\{\psi(v_4)\}$ if $x$ is a neighbor of $v_4$ and $L'(x)=L(x)$ otherwise.
Note that $G'$ is not $L'$-colorable and has precolored path $p_0p_1p_2v_1v_2w_1$.
The choice of $\psi$ ensures that~(S3) holds.
By \refclaim{cl-inter}, no neighbor of $w_2$ is adjacent to $v_4$, as otherwise
$w_5$ would have degree two; thus, \refclaim{cl-nonearchords} implies that~(I) holds.
As $w_1$ has degree at least three, \refclaim{cl-inter} implies that $w_2$ is not
adjacent to a vertex of $P$ and~(Q) holds.  Therefore, $G'$ violates~(OBSTb) and contains an image $H$ of
one of the relevant configurations drawn in Figure~\ref{fig-obstb}.  No neighbor of $v_2$ has list of size two, thus $w_1$ belongs
to $H$.  If $v_1$ or $v_2$ had degree greater than two in $H$, then $G$ would contain a $(\le\!3)$-chord contradicting
\refclaim{cl-schords} or \refclaim{cl-nonearchords}; hence, $H$ is OBSTb1a, OBSTb1b or OBSTb4.  Since $w_1$ has degree at least three,
$H$ is not OBSTb1a.  If $H$ were OBSTb1b, then $G$ would contain a $(\le\!3)$-chord starting in $v_2$ contradicting \refclaim{cl-nonearchords}.
Finally, if $H$ is OBSTb4, then let $w_2yz$ be the path in the boundary of the outer face of $H$ with $|L'(z)|=2$.  If $z$ is a neighbor of
$v_4$, then by \refclaim{cl-inter} we have $y=w_4$ and $z=w_5$; however, then there exists a path $v_4w_5y'z'$ in the boundary of
the outer face of $H$ with $|L(z')|=2$, contradicting \refclaim{cl-nonearchords}.  Otherwise, we have $|L(z)|=2$.  Consider the subgraph
split off by $v_3w_3w_4w_2yz$.  Since both $v_3$ and $z$ have lists of size two and $w_3$ and $y$ have no common neighbor,
this subgraph satisfies the assumptions of Theorem~\ref{thm-maingen}, contradicting the minimality of $G$.

Therefore, we can assume that such a coloring $\psi$ does not exist.  By \refclaim{cl-inter} and \refclaim{cl-nonearchords},
this can only happen if $w_1$ is adjacent to $p_1$.  Since $w_5$ has degree at least three, \refclaim{cl-4chord}
implies that $w_4$ has no neighbor in $P$.  It follows that there exists an $L$-coloring $\psi'$
of the subgraph induced by $V(P)\cup \{v_1,v_2,v_3,v_4,w_1,w_2,w_3,w_4\}$ such that
$\psi'(v_4)\not\in L(v_5)$.  Consider the graph $G'=G-\{p_2,v_1,v_2,v_3,v_4,w_3\}$ with the list assignment $L'$
such that $L'(x)=\{\psi'(x)\}$ for $x\in\{w_1,w_2,w_4\}$,
$L'(x)=L(x)\setminus\{\psi'(v_4)\}$ if $x$ is a neighbor of $v_4$ and $L'(x)=L(x)$ otherwise.
Note that $G'$ is not $L'$-colorable and has precolored path $p_0p_1w_1w_2w_4$.
By \refclaim{cl-inter} and \refclaim{cl-nonearchords}, $w_2$ is not adjacent to any vertex with list of size two
and $w_5$ is the only neighbor of $w_4$ with list of size two.  Note that $w_5$ is not adjacent to
$p_0$ by \refclaim{cl-nonearchords}.  Furthermore,  $w_5$ is not adjacent to $p_1$, since (similarly to \refclaim{cl-4chord}) we
would have that the path $p_0p_1w_5v_4$ splits off a $5$-face, implying that $v_5$ is adjacent to $p_0$ and contradicting
\refclaim{cl-length}.  It follows that $G'$ satisfies~(Q).  Furthermore, $G'$ satisfies~(OBSTb), since by \refclaim{cl-nonearchords}
it does not contain a path $v_4w_5xy$ with $|L(y)|=2$.  Therefore, $G'$ a counterexample to Theorem~\ref{thm-maingen} smaller than $G$,
which is a contradiction.  It follows that (A4) is false.
\end{proof}

\claim{cl-B3false}{(B3) is false.}
\begin{proof}
Suppose that~(B3) holds.  Let $\psi$ be an $L$-coloring of the subgraph $G_0$ of $G$ induced by $V(P)\cup \{v_1,v_2,\ldots, v_6,w_2\}$
such that $\psi(v_6)\not\in L(v_7)$ ($w_2$ has no neighbor in $P$ by \refclaim{cl-inter} and \refclaim{cl-4chord}, thus such a coloring exists).
Let $L'$ be the list assignment such that
$L'(x)=\{\psi(x)\}$ for $x\in \{v_1,v_2,v_3\}$, $L'(v_4)=\{\psi(v_3),\psi(v_4)\}$,
$L'(x)=L(x)\setminus\{\psi(y)\}$ if $x$ has a neighbor $y\in \{w_2,v_6\}$
and $L'(x)=L(x)$ otherwise.  The graph $G'=G-\{v_5,v_6,w_2\}$ with the precolored induced path $p_0p_1p_2v_1v_2v_3$
is not $L'$-colorable, and by \refclaim{cl-inter}
and \refclaim{cl-4chord}, it satisfies~(I) and~(Q).  Furthermore, note that there exists another $L$-coloring $\psi'$ of
$G_0$ such that $\psi'(v_6)=\psi(v_6)$, $\psi'(w_2)=\psi(w_2)$, $\psi'(v_4)\neq\psi(v_4)$ and $\psi'(v_2)\neq\psi(v_2)$,
thus we can choose $\psi$ so that~(OBSTb) holds, unless $G'$ contains OBSTb3.  By \refclaim{cl-inter} and \refclaim{cl-nonearchords},
we then have that $z_1$ is adjacent to $p_1$ and $w_1$ is adjacent to $p_0$, and by \refclaim{cl-4chord} applied to $v_6w_2w_1p_0$,
$v_7$ is adjacent to $p_0$.  Nevertheless, such a graph $G$ is $L$-colorable.  This is a contradiction.
\end{proof}

By \refclaim{cl-neart} and \refclaim{cl-neartb} combined with \refclaim{cl-B2false}, \refclaim{cl-B4false}, \refclaim{cl-A4false}
and \refclaim{cl-B3false}, we have that
\claim{cl-subclaim}{$G$ satisfies~{\rm (A2)}, {\rm (A5)} or~{\rm (B1)}.}

Suppose that there exists a vertex $t\in V(T)\cap (V(P)\cup \{v_1\})$. Let $G'$ be the graph obtained from
$G$ by splitting $t$ to two vertices
$t'$ and $t''$ and adding a new vertex $v$ adjacent to $t'$ and $t''$, so that $G'$ is a plane graph
and $T$ becomes a $5$-face.
Let $\psi$ be an $L$-coloring of the subgraph of $G$ induced by $V(P)\cup \{t\}$,
$c$ a color that does not appear in any of the lists, and let $L'$ be the list assignment such that
$L'(t')=L'(t'')=\{\psi(t)\}$, $L(v)=\{c\}$ and $L'(x)=L(x)$ otherwise.
Note that $G'$ is not $L'$-colorable and contains a precolored path $R$ of length at most five.
By the minimality of $G$, we conclude that $G'$ violates~(OBSTb); let $H$ be the image of a relevant configuration from Figure~\ref{fig-obstb}
that is contained in $G'$.  In $H$, $v$ has degree two and is incident with a $5$-face.  If $t\in V(P)$, then
$\ell(R)=4$ and $H$ is OBSTb1 or OBSTb2; but then $G$ contains OBSTx1c or OBSTa6.  Therefore, we can assume that $t=v_1$.  If $H$ is OBSTb1, then $G$ contains OBSTx1;
if $H$ is OBSTb1a, then $G$ contains OBSTx1a;
if $H$ is OBSTb1b, then $G$ contains OBSTx1b;
if $H$ is OBSTb2b, then $G$ contains OBSTx4; and 
if $H$ is OBSTb5, then $G$ contains OBSTx2b.
It follows that $H$ is OBSTb4 or OBSTb6.  By \refclaim{cl-inter} and \refclaim{cl-schords}, we conclude that $G$ is equal
to the graph obtained from $H$ by removing $v$ and identifying $t'$ with $t''$.  However, then $G$ is $L$-colorable.
This is a contradiction; therefore,
\claim{cl-trfar}{we have $V(T)\cap (V(P)\cup \{v_1\})=\emptyset$.}

Let $X'$ be the subset of $\{v_s,v_{s-1},v_{s-2},v_{s-3}\}$ defined symmetrically to $X$ on the other side of $P$
(i.e., if $|L(v_s)|=3$, $|L(v_{s-1})|=2$ and $|L(v_{s-2})|=|L(v_{s-3})|=3$, then $X'=\{v_{s-1}\}$, and so on).
Note that a claim symmetric to \refclaim{cl-subclaim} holds, i.e., there exists a triangle $T'$
that either intersects $X'$ (the case~(A2)) or contains one of the vertices $z'_1$ or $z'_2$ incident with a $5$-face
$v_{s-1}v_{s-2}v_{s-3}z'_2z'_1$ (the cases~(B1) and~(A5)).
Observe that $d(T,T')\le 8<B$, and thus $T=T'$.

If (A5) or (B1) holds, then let $b$ be the first vertex in the sequence $v_2$, $z_1$, $z_2$ and $v_4$ that is incident with $T$;
otherwise, let $b$ be the first vertex in the sequence $v_2$, $v_3$ and $v_4$ that is incident with $T$.
Symmetrically, let $b'$ be the first such vertex among $v_{s-1}$, $z'_1$, $z'_2$ and $v_{s-3}$
or among $v_{s-1}$, $v_{s-2}$ and $v_{s-3}$.
Note that either $b=b'$ or $b$ and $b'$ are adjacent.

Suppose now that $V(T)\subseteq V(C)$.  In this case~(A5) does not hold.  By \refclaim{cl-nosimt}, we have $b\in \{v_3,v_4\}$ and
$b'\in \{v_{s-2},v_{s-3}\}$.  If $b'=v_{s-3}$, then $v_{s-3}\in X'$ and by the choice of $X'$, we have $|L(v_{s-2})|=2$.
This contradicts \refclaim{cl-nosimt}.  Thus $b'=v_{s-2}$ and symmetrically, $b=v_3$.  By \refclaim{cl-nosimt},
we have $|L(v_2)|=|L(v_{s-1})|=3$, and by \refclaim{cl-lv12l}, $|L(v_1)|=2$.  However, then $X=\{v_1\}$
and $b\not\in X$, which is a contradiction. It follows that
\claim{cl-tns}{$T$ shares at most two vertices with $C$.}

Let us now show that $b\neq v_4$ and $b'\neq v_{s-3}$.  It suffices to prove that
\claim{cl-no3no4}{if $v_4\in X\cap V(T)$, then $v_3\in V(T)$; and if $v_{s-3}\in X'\cap V(T)$, then $v_{s-2}\in V(T)$.}
\begin{proof}
Suppose that say $v_{s-3}\in X'\cap V(T)$ and $v_{s-2}\not\in V(T)$.  The choice of $X'$ implies that
$|L(v_{s-3})|=3$ and $|L(v_{s-2})|=|L(v_{s-4})|=2$.

If $\{v_2,v_3,v_4\}\cap V(T)=\emptyset$, then $b\in \{z_1,z_2\}$;
let $v\in \{v_2,v_4\}$ be the neighbor of $b$.  By \refclaim{cl-nonearchords} applied to $vbv_{s-3}$, we
conclude that $T=vbv_{s-3}$, contrary to the assumption that $v\not\in V(T)$.  It follows that a vertex of
$\{v_2,v_3,v_4\}\cap V(T)$ is equal to either $v_{s-3}$ or $v_{s-4}$.  By \refclaim{cl-length}, we have
$6\le s\le 8$.

If $s=8$, then $v_4=v_{s-4}$.  However, let us recall that $|L(v_{s-4})|=2$,
hence neither (A5) nor (B1) holds.  Furthermore, we have $v_4\not\in X$, and thus (A2) does not
hold either.  This is a contradiction.


Therefore, $s\le 7$.  By \refclaim{cl-nochords} and \refclaim{cl-tns}, $C$ has no chords.  If $t\in V(T)\setminus V(C)$
has a neighbor $w\in V(C)$, then $wt$ is an edge of $T$, as otherwise \refclaim{cl-inter} would imply that $v_{s-1}$ or $v_{s-5}$ (which have
lists of size three) has degree two.

Note that there exists at most one vertex with two neighbors in the path $p_0p_1p_2v_1v_2$
and another neighbor in $T$.
If such a vertex $v$ exists, then $v_{s-4}$ has degree two by \refclaim{cl-inter}, hence
$V(T)\cap V(C)=\{v_{s-3}\}$; let $Y=V(P)\cup V(T)\cup \{v_1,v_2,\ldots, v_{s-4},v\}$.
Otherwise, let $Y=V(P)\cup V(T)\cup \{v_1,v_2,\ldots, v_{s-4}\}$.
Observe that in both cases, there exists an $L$-coloring $\psi$ of the subgraph of $G$ induced by
$Y$ such that $\psi(v_{s-3})\not\in L(v_{s-2})$.
Let $G'=G-V(T)$ and let $L'$ be the list assignment given by
$L'(x)=\{\psi(x)\}$ for $x\in\{v_1,\ldots, v_{s-5}\}$, $L'(v_{s-4})=\{\psi(v_{s-5}),\psi(v_{s-4})\}$,
$L'(x)=L(x)\setminus \{\psi(y)\}$ if $x$ has a neighbor $y\in V(T)$, and $L'(x)=L(x)$ otherwise.  Note that $G'$ is not $L'$-colorable
and contains an induced precolored path $p_0p_1p_2v_1\ldots v_{s-5}$ of length at most four.
By \refclaim{cl-inter} and \refclaim{cl-nonearchords}, $G'$ satisfies~(I).  The choice of $Y$ and $\psi$ ensures that~(Q) holds as well.
Thus, $G'$ must violate~(OBSTb), and in particular $s=7$ and $v_3\not\in V(T)$.  Let $H$ be the image of a relevant configuration OBSTb1
or OBSTb2 contained in $G'$.  By \refclaim{cl-inter}, $v_s$ is the only vertex with list of size two adjacent to $p_0$, thus $v_s\in V(H)$.
Let $v_sxy$ be the path in the outer face of $H$ such that $|L'(y)|=2$.  By \refclaim{cl-inter}, we have $x=v_{s-1}$.
By~\refclaim{cl-nonearchords} applied to the appropriate path from $v_2$ to $y$ in $H$, we conclude that $y\neq v_{s-2}$,
and thus $y$ is a common neighbor of $v_{s-1}$ and a vertex of $V(T)\setminus V(C)$.  Similarly, we exclude the case that
$H$ is OBSTb1, hence $H$ is OBSTb2.  Thus, there exists a path of length three between $y$ and $v_2$ in the outer
face of $H$, say $yu_1u_2v_2$.  By \refclaim{cl-inter}, we have $u_2\neq v_3$,
hence $u_2$ has a neighbor in $T$.  This uniquely determines the graph $G$.  However, the described graph
is $L$-colorable, which is a contradiction.
\end{proof}

\claim{cl-bz1z2}{Either $\in \{z_1,z_2\}$ or $b'\in\{z'_1,z'_2\}$.}
\begin{proof}
If $b\not\in \{z_1,z_2\}$ and $b'\not\in\{z'_1,z'_2\}$, then since $\ell(C)>8$, $b\neq v_4$ and $b'\neq v_{s-3}$,
we have $s=6$, $b=v_3$ and $b'=v_{s-2}$.  The vertices $v_3$ and $v_4$ cannot both have list of size two.
By symmetry, we can assume that $|L(v_4)|=3$, and since $v_4\in X'$, the choice of $X'$ implies
$|L(v_5)|=2$, $|L(v_6)|=3$ and $|L(v_3)|=2$.  Consequently, $|L(v_2)|=3$ and $|L(v_1)|=2$.  Let $\psi$ be a coloring of the subgraph of $G$ induced
by $V(P)\cup V(T)\cup \{v_1,v_2\}$ such that $\psi(v_4)\not\in L(v_5)$; note that \refclaim{cl-inter}
implies that the vertex of $V(T)\setminus V(C)$ is not adjacent to a vertex of $P$, ensuring that such a coloring exists.
Let $G'=G-V(T)$ and let $L'$ be the list assignment such that $L'(v_1)=\{\psi(v_1)\}$, $L'(v_2)=\{\psi(v_1),\psi(v_2)\}$,
$L'(x)=L(x)\setminus \{\psi(y)\}$ if $x$ has a neighbor $y\in V(T)$, and $L'(x)=L(x)$ otherwise.
The graph $G'$ is not $L'$-colorable, and by \refclaim{cl-inter} and \refclaim{cl-nonearchords}, it satisfies~(I) and~(Q).
Since the precolored path $p_0p_1p_2v_1$ in $G'$ has length three, this contradicts the minimality of $G$.
\end{proof}
By symmetry, we can assume that $b\in \{z_1,z_2\}$, say.

\claim{cl-bz2}{We have $b=z_2$.}
\begin{proof}
Suppose that $b=z_1$.
By~\refclaim{cl-nonearchords}, we have $b\neq b'$ and $b'\in \{z'_1,z'_2\}$.  Let $V(T)=\{b,b',t\}$,
let $v'\in\{v_{s-1},v_{s-3}\}$ be the neighbor of $b'$ and let $G_2$ be the subgraph split off by
$v_2z_1b'v'$.

If $T\not\subset G_2$, then \refclaim{cl-4chord} implies that $v_2$ and $v'$ have a common neighbor
with list of size two, hence $v'=v_4=v_{s-3}$ and $b'=z_2$.  Since $\ell(C)>8$, we have $b'\neq z'_1$,
and thus $b'=z'_2$.
Note that $t\neq z'_1$, as otherwise $z'_1$ would be chosen to be $b'$.  If $t$ has a neighbor in $P$, then since $z'_1$ has degree at least three,
\refclaim{cl-inter} implies $tp_0,z'_1p_1\in E(G)$.  This would uniquely determine $G$; however, the resulting graph is $L$-colorable.
It follows that $t$ has
no neighbor in $P$.  Similarly, $z_1$ and $z_2$ have no neighbors in $C$ other than $v_2$ and $v_4$, and no neighbor of $v_7$ is
adjacent to a vertex of $T$.  There exists an $L$-coloring of the subgraph of $G$ induced by $V(P)\cup V(T)\cup \{v_1,v_2,v_3\}$
such that $|L(v_4)\setminus \{\psi(v_3),\psi(z_2)\}|\ge 2$.  Let $G'=G-(V(T)\cup \{v_3,v_4,v_5\})$ with the list assignment $L'$
such that $L'(v_1)=\{\psi(v_1)\}$, $L'(v_2)=\{\psi(v_1),\psi(v_2)\}$, $L'(x)=L(x)\setminus \{\psi(y)\}$ if $x$ has a neighbor $y\in V(T)$,
and $L'(x)=L(x)$ otherwise.  Observe that $G'$ with the precolored path
$p_0p_1p_2v_1$ satisfies the assumptions of Theorem~\ref{thm-maingen} and is not $L'$-colorable,
contradicting the minimality of $G$.

Let us now consider the case of $T\subset G_2$.  Since $t$ has degree at least three, we conclude that
the subgraph of $G$ split off by the path
$v_2z_1tb'v'$ is OBSTb1, $t=z_2$ and either $z'_2=z_2$, $b'=z'_1$ and $s=7$,
or $b'=z'_2$ and $s=9$.  Suppose that $b$ or $b'$ has a neighbor in $P$.  If $s=7$,
then the resulting graph would be $L$-colorable.  If $s=9$, then \refclaim{cl-inter} implies that $z'_1$ has degree two.
This is a contradiction, hence neither $b$ nor $b'$ has a neighbor in $P$.
Let $\psi$ be an $L$-coloring of the subgraph of $G$ induced by
$V(P)\cup V(T)\cup \{v_1,v_2,v_3\}$ such that $|L(v_4)\setminus \{\psi(v_3),\psi(t)\}|\ge 2$.
Let $G'=G-\{v_3,v_4,v_5,t\}$ if $s=7$ and $G'=G-\{v_3,v_4,v_5,v_6,v_7,t\}$ if $s=9$, with the list assignment $L'$ such that $L'(x)=\{\psi(x)\}$ if $x\in \{v_1,v_2,z_1\}$,
$L'(b')=\{\psi(b'),\psi(z_1)\}$ and $L'(x)=L(x)$ otherwise.  Note that $G'$ with the precolored induced path
$p_0p_1p_2v_1v_2z_1$ is not $L'$-colorable.  By the minimality of $G$, we conclude that $G'$
violates~(OBSTb).  Since $b'$ and $v_s$ are the only vertices with list of size two,
$G'$ contains OBSTb1a, OBSTb1b or OBSTb3 as a subgraph; and if $s=9$, \refclaim{cl-inter} implies
that $z'_1$ belongs to this subgraph.  However, in all the cases the resulting graph is $L$-colorable,
which is a contradiction.
\end{proof}

\claim{cl-bpnc}{We have $b'\not\in V(C)$.}
\begin{proof}
Suppose that $b'\in V(C)$.  If $b'=v_4$, then \refclaim{cl-no3no4} implies that $v_4\not\in X$,
thus~(A5) holds and $v_5\in V(T)$.  This is a contradiction, as we would choose $b=v_5$.  Therefore, $b'\neq v_4$,
and \refclaim{cl-nonearchords} implies that the $2$-chord
$v_4bb'$ splits off $T$, thus $b'=v_5$.  Since $v_3\not\in V(T)$, we have $v_4\not\in X$ and~(A5) holds by \refclaim{cl-no3no4}.
However, since $|L(v_4)|=|L(v_5)|=3$, we have $v_5\not\in X'$.  This implies that
$X'$ satisfies (A5) as well, and $z_2$ should have been chosen as $b'$.  This is a contradiction.
\end{proof}

We can assume that $b'\neq z'_1$, since we already excluded the
symmetric case that $b=z_1$ in \refclaim{cl-bz2}.  Therefore, we have $b'=z'_2$.
\claim{cl-bneq}{The vertices $b$ and $b'$ are different.}
\begin{proof}
Suppose that $b=b'$.   By \refclaim{cl-nonearchords}, we have $v_{s-3}\in\{v_4,v_5\}$.

If $v_{s-3}=v_4$, then let $V(T)=\{b,t,t'\}$, and note that $\{t,t'\}\cap \{z_1,z'_1\}=\emptyset$, by the choice of $b$ and $b'$.
Since $z_1$ and $z'_1$ have degree at least three, \refclaim{cl-inter} implies that the vertices of $T$ have no neighbors in $P$,
and that the distance between $T$ and $\{v_1,v_7\}$ is at least three.  There exists an $L$-coloring $\psi$ of the subgraph of $G$
induced by $V(P)\cup V(T)\cup\{v_1,v_2,v_3\}$ such that $|L(v_4)\setminus \{\psi(v_3),\psi(b)\}|\ge 2$.  Let $G'=G-\{v_3,v_4,v_5,b,t,t'\}$
and $L'$ the list assignment such that $L'(v_1)=\{\psi(v_1)\}$, $L'(v_2)=\{\psi(v_2)\}$, $L'(x)=L(x)\setminus \{\psi(y)\}$ if $x$ has
a neighbor $y\in V(T)$ and $L'(x)=L(x)$ otherwise.  Observe that $G'$ is not $L'$-colorable,
contains an induced precolored path $p_0p_1p_2v_1v_2$ of length four and satisfies~(I).  Since $z_1$ has degree at least three
and it is not adjacent to $p_0$, \refclaim{cl-inter} implies that $G'$ satisfies~(Q).   It follows that $G'$ contains
an image $H$ of OBSTb1 or OBSTb2.
By \refclaim{cl-inter}, we have $z_1, v_7\in V(H)$.  If $H$ is OBSTb1, then $C$ has a $3$-chord $v_2z_1xv_7$ contradicting \refclaim{cl-nonearchords}.
If $H$ is OBSTb2, then $G$ contains a path $v_2z_1xyzv_7$, where $y$ has a neighbor in $T$.  However, then $t$ or $t'$ has degree two by \refclaim{cl-inter},
which is a contradiction.

If $v_{s-3}=v_5$, then both $X$ and $X'$ satisfy~(A5).  By \refclaim{cl-nonearchords}, we have $z_1\neq z'_1$.  Since both $z_1$ and $z'_1$ have degree
at least three, \refclaim{cl-inter} implies that $b$ has no neighbor in $P$ and has distance at least three from $\{v_1,v_7\}$.  Let $\psi$
be an $L$-coloring of the subgraph of $G$ induced by $V(P)\cup V(T)\cup \{v_1,v_2,v_3\}$ such that $\psi(v_5)\not\in L(v_6)$.
Let $G'=G-\{v_3,v_4,v_5,v_6,b\}$ and let $L'$ be the list assignment such that $L'(v_1)=\{\psi(v_1)\}$, $L'(v_2)=\{\psi(v_2)\}$,
$L'(x)=L(x)\setminus \{\psi(y)\}$ if $x$ has a neighbor $y\in V(T)$ and $L'(x)=L(x)$ otherwise.  Observe that $G'$ is
not $L'$-colorable, contains a precolored induced path $p_0p_1p_2v_1v_2$ of length four and satisfies~(I) and~(Q).  By the minimality of $G$,
it follows that $G'$ contains an image $H$ of OBSTb1 or OBSTb2.
The distance between the neighbors of $b$ is at least three, thus at most one of them belongs to $H$ and has list of size two.
It follows that $H$ is OBSTb1 and $v_7\in V(H)$.  However, then $z_1$ or $z'_1$ has degree two by \refclaim{cl-inter}, which is a contradiction.
\end{proof}

Finally, let us consider the case that $b\neq b'$.  Since $T$ has two vertices that do not belong to $C$, neither $X$ nor $X'$ satisfies~(A5).
Since $v_3\not\in V(T)$, by \refclaim{cl-no3no4} we have $v_4\not\in V(T)$, and symmetrically, $v_{s-3}\not\in V(T)$.
Note that $v_4$ is adjacent to $b$, $v_{s-3}$ is adjacent to $b'$, and neither $v_4$ nor $v_{s-3}$ is incident with a triangle;
hence, we have $v_{s-3}\neq v_4$.  Let $\{t\}=V(T)\setminus\{b,b'\}$.
Consider the $3$-chord $Q=v_4bb'v_{s-3}$ and the subgraph $G_2$ split off by it.  If $T\not\subset G_2$, then \refclaim{cl-4chord}
implies that $v_4$ and $v_{s-3}$ have a common neighbor, and thus $s=9$.  If $T\subset G_2$, then we similarly conclude that
$v_4btb'v_{s-3}$ splits off OBSTb1, i.e., $s=11$ and $t$ is adjacent to $v_6$.

Let $S_1=L(v_2)\setminus (L(v_1)\setminus L(p_2))$ and $S_2=L(v_{s-1})\setminus (L(v_s)\setminus L(p_0))$.
By the minimality of $G$, we have $|S_1|=|S_2|=2$, as otherwise we can remove the edge $v_1v_2$ or $v_{s-1}v_s$.
Suppose now that there exists an $L$-coloring $\psi$ of $T$ such that for every $c_1\in S_1$ and $c_2\in S_2$,
there exists an $L$-coloring $\varphi$ of the subgraph of $G$ induced by $V(T)\cup \{v_2,v_3,\ldots, v_{s-1}\}$
such that $\varphi(v_2)=c_1$, $\varphi(v_{s-1})=c_2$ and $\varphi(x)=\psi(x)$ for $x\in V(T)$.
Let $G'=G-(V(T)\cup \{v_3,v_4,\ldots,v_{s-2}\})$ and let $L'$ be the list assignment such that $L'(x)=L(x)\setminus\{\psi(y)\}$
if $x$ has a neighbor $y$ in $V(T)$ and $L'(x)=L(x)$ otherwise.  The choice of $\psi$ implies that every $L'$-coloring
of $G'$ corresponds to an $L$-coloring of $G$, thus $G'$ is not $L'$-colorable.  The precolored path $P$
in $G'$ has length two.  Note that no vertex of $T$ is adjacent
to a vertex of $P$ and that the distance between $T$ and $\{v_1,v_s\}$ is at least three,
since otherwise \refclaim{cl-inter} would imply that $z_1$ or $z'_1$ has degree two.
Thus, $G'$ satisfies~(S3) and~(I).  Furthermore, it satisfies~(OBSTa), since $t(G)\ge B$.
Therefore, $G'$ would be a counterexample
to Theorem~\ref{thm-maingen} contradicting the minimality of $G$.

We conclude that no such $L$-coloring $\psi$ exists.  In particular, for any color $c\in L(b)$, the list $L(v_4)\setminus\{c\}$
has size two and intersects $L(v_3)$.  It follows that $L(v_3)\subseteq L(v_4)=L(b)$, and symmetrically, $L(v_{s-2})\subseteq L(v_{s-3})=L(b')$.
Similarly, we conclude that $L(v_3)=S_1$, $L(v_{s-2})=S_2$, $L(v_5)\subseteq L(v_4)$, $L(v_{s-4})\subseteq L(v_{s-3})$,
and if $s=11$, then $L(v_5),L(v_7)\subseteq L(v_6)=L(t)$.  If $L(v_3)=L(v_5)=S_1$, then choose $\psi(b)\in S_1$ arbitrarily.
Now, regardless of the values of $c_1$, $c_2$ and the rest of $\psi$, we can choose the color of $v_4$ to be the unique color in
$L(v_4)\setminus S_1$, and the $L$-coloring $\varphi$ will exist.  Therefore, $L(v_5)\neq S_1$ and $L(v_{s-4})\neq S_2$.
Similarly, if $s=11$, then $L(v_5)\neq L(v_7)$.

Let us now define an $L$-coloring $\psi$ of $T$.
Let $\{c_3\}=L(v_5)\cap S_1$.  Set
$\psi(b)$ to be the unique color in $S_1\setminus L(v_5)$.  Furthermore, if $s=11$ then let $\psi(t)=c_3$, and if $s=9$ then let $\psi(b')=c_3$.
Observe that $\psi$ (extended to the third vertex of $T$ arbitrarily) has the required property---if $c_1\neq\psi(b)$, then we can color $v_3$ by $\psi(b)$, so that two neighbors of $v_4$
have the same color.  And if $c_1=\psi(b)$, then we can color $v_3$ by $c_3$, $v_4$ by the color in $L(v_4)\setminus S_1$
and $v_5$ with $c_3$, so that $v_6$ has two neighbors with the same color.  In both cases, we can greedily extend the
coloring of $v_{s-1}$ by $c_2$ to the rest of the vertices.  This contradiction finishes the proof of
Theorem~\ref{thm-maingen}.
\end{proof}

\section{Concluding remarks}

Theorem~\ref{thm-main} follows by choosing an arbitrary vertex $p$ in the outer face of $G$ and
a color $c\in L(p)$, changing the list of $p$ to $\{c\}$ and applying Theorem~\ref{thm-maingen} with the path consisting only of $p$.

The proof of Theorem~\ref{thm-maingen} follows the lines of Thomassen's original proof~\cite{thomassen1995-34}.
However, a basically unavoidable part of the proof---the need to handle $2$-chords, so that we can color
and remove a $5$-face in \refclaim{cl-neartb}---forces us to deal with a large number of counterexamples to the claim
``every precoloring of a path of length two can be extended.'' Especially painful is the obstruction OBSTx1,
which even applies to a path of length one.  One could ask whether we could not avoid this by forbidding vertices
with list of size two in triangles completely.  This cuts down the number of obstructions significantly,
and indeed, this was our original aim.  However, at the final stage of the proof, we would only end up knowing
that there is a triangle whose distance is at most two from a vertex on each side of the precolored path $P$.
This is a quite small amount of structure to work with, making the arising case analysis extremely difficult.
Additionally, one runs into trouble if these two vertices are in fact identical, which would essentially force
extending Corollary~\ref{cor-sepcycles} to precolored cycles of length at most $10$.  The number of obstacles 
for such cycles then becomes rather large, and it is not quite clear how such an extension
of Corollary~\ref{cor-sepcycles} could be proved.

Another point where one could hope to save on obstructions is by only considering the precoloring of a path of
length at most $4$ in case that $(\le\!4)$-cycles are far enough from it.  However, there are many places throughout
the proof where it is useful to extend the coloring of a path of length two to a coloring of a path of length five,
and it is unclear how to handle these situations using only paths of length four.

Consequently, we end up with a nontrivial number of obstructions, and the proof becomes rather technical.
Despite the length of this paper, still a large amount of work is hidden in the need to carefully verify all the claims;
in particular, we in general do not give detailed proofs of colorability of the described graphs.
We feel that doubling the length of the paper by spelling out all these
technical details would not make the exposition any clearer or more believable.
Similar remarks apply to other results proved using this technique (even the original paper of Thomassen~\cite{dkt},
although written quite shortly, becomes rather long when all details are worked out).
Given the rather repetitive nature of the arguments, one wonders whether
it would not be possible to apply computer to obtain such proofs.   Let us however note that many of the reductions
appearing in our proof are quite tricky and it is not immediately obvious how they could be obtained mechanically.

On the positive side, Theorem~\ref{thm-maingen} is somewhat interesting even for graphs of girth five, since
it describes which precolorings of a path of length at most five can be extended.  This might be useful as a technical
tool in further study of $4$-critical graphs of girth five.  Similarly, Theorem~\ref{thm-maingen} and Corollary~\ref{cor-sepcycles}
give interesting information regarding graphs with exactly one cycle of length at most four.

Compared with the solution to Havel's problem~\cite{dkt}, our proof is rather elementary, not using any deeper results.
Would it be possible to apply the techniques of \cite{dkt} instead?  Possibly, but it would require developing a new proof
of $3$-choosability of planar graphs of girth $5$ based on reducible configurations and discharging.  While our initial
inquiry in that direction was somewhat encouraging, it seems inevitable that the set of reducible configurations needed
would be rather large (possibly hundreds as opposed to about $10$ needed in \cite{dkt} for the case of $3$-coloring), so
the proof would become of somewhat dubious value.

Finally, let us remark that we could require a much weaker assumption on the distance between $4$-cycles, since in
most of the arguments only triangles cause problems.  However, for obvious reasons we did not want to complicate the
proof any more.

\section*{Acknowledgements}
I would like to thank an anonymous referee who actually managed to read through the paper
up to this point, both for his or her patience and for useful remarks on the presentation.

\bibliographystyle{siam}
\bibliography{far34b}

\end{document}